\documentclass[11pt]{article}

\usepackage{latexsym}
\usepackage{amssymb}
\usepackage{amsthm}
\usepackage{amscd}
\usepackage{amsmath}
\usepackage{mathrsfs}
\usepackage{graphicx}
\usepackage{shuffle}
\usepackage[colorlinks=true, pdfstartview=FitV, linkcolor=blue, citecolor=blue, urlcolor=blue]{hyperref}
\usepackage{mathtools}
\usepackage[bbgreekl]{mathbbol}
\usepackage{mathdots}
\usepackage{ytableau}
\usepackage[enableskew,vcentermath]{youngtab}
\usepackage{tikz-cd}

\usepackage{mathtools}
\usepackage{amssymb}

\usepackage[all]{xy}
\input xy \xyoption{frame}
\xyoption{dvips}

\usepackage[colorinlistoftodos]{todonotes}
\usetikzlibrary{chains,scopes}
\usetikzlibrary{shapes.geometric,positioning}

\usepackage{stmaryrd}

\DeclareSymbolFontAlphabet{\mathbb}{AMSb}
\DeclareSymbolFontAlphabet{\mathbbl}{bbold}

\def\Sp{\mathsf{Sp}}
\def\O{\mathsf{O}}
\def\K{\mathsf{K}}

\definecolor{darkred}{rgb}{0.7,0,0} 
\newcommand{\defn}[1]{{\color{darkred}\emph{#1}}} 

\renewcommand{\mid}{:}

\numberwithin{equation}{section}
\allowdisplaybreaks[1]
\UseCrayolaColors

\theoremstyle{definition}
\newtheorem* {theorem*}{Theorem}
\newtheorem* {corollary*}{Corollary}
\newtheorem* {conjecture*}{Conjecture}
\newtheorem{theorem}{Theorem}[section]

\newtheorem{problem}[theorem]{Problem}
\theoremstyle{definition}

\newtheorem* {example*}{Example}

\newtheorem{lemma}[theorem]{Lemma}
\theoremstyle{definition}
\newtheorem{definition}[theorem]{Definition}
\theoremstyle{definition}

\newtheorem{conjecture}[theorem]{Conjecture}
\newtheorem{proposition}[theorem]{Proposition}
\newtheorem{corollary}[theorem]{Corollary}

\newtheorem{remark}[theorem]{Remark}
\newtheorem*{remark*}{Remark}
\theoremstyle{definition}
\newtheorem {example}[theorem]{Example}
\theoremstyle{definition}

\theoremstyle{definition}

\theoremstyle{definition}

\xyoption{dvips}

\def\({\left(}
\def\){\right)}

\newcommand{\CC}{\mathbb{C}}
\newcommand{\QQ}{\mathbb{Q}}

\newcommand{\cZ}{\mathcal{Z}}
\def\cX{\mathcal{X}}
\def\cY{\mathcal{Y}}

\def\CC{\mathbb{C}}

\def\ZZ{\mathbb{Z}}

\def\GL{\textsf{GL}}

\def\spanning{\textnormal{-span}}

\def\diag{\mathrm{diag}}

\newcommand{\cL}{\mathcal{L}}

\def\fk{\mathfrak}

\def\barr{\begin{array}}
\def\earr{\end{array}}
\def\ba{\begin{aligned}}
\def\ea{\end{aligned}}
\def\be{\begin{equation}}
\def\ee{\end{equation}}

\def\qquand{\qquad\text{and}\qquad}
\def\quand{\quad\text{and}\quad}

\def\I{\mathcal{I}}

\def\cH{\mathcal H}

\def\DesR{\mathrm{Des}_R}

\def\fkS{\fk S}

\def\ben{\begin{enumerate}}
\def\een{\end{enumerate}}

\def\bei{\begin{itemize}}
\def\eei{\end{itemize}}

\def\fpf{{\textsf{fpf}}}

\def\D{\hat D}

\def\DO{ D^{\O}}

\def\Des{\mathrm{Des}}

\def\Ifpf{I^{\fpf}}

\def\e{\textbf{e}}

\newcommand{\Fl}{\textsf{Fl}}

\newcommand{\cB}{\mathcal{B}}

\def\arcstart{\ \xy<0cm,-.15cm>\xymatrix@R=.1cm@C=.3cm }
\newcommand{\arcstartc}[1]{\ \xy<0cm,-.15cm>\xymatrix@R=.1cm@C=#1cm}

\def\iS{\hat \fkS}

\newcommand{\row}{\operatorname{row}}

\def\simFKK{\mathbin{\overset{\textsf{Sp}}\approx}}

\newcommand{\Sym}{\textsf{Sym}}

\newcommand{\weight}{\operatorname{wt}}

\def\row{\textsf{row}}
\def\revrow{\textsf{revrow}}

\def\path{\mathsf{path}_{\textsf{OEG}}}

\def\whSym
{\mathfrak{m}\textsf{Sym}}

\def\whQSym
{\mathfrak{m}\textsf{QSym}}

\def\D{\textsf{D}}
\def\SD{\textsf{SD}}
\def\DSym{\textsf{D}^{\textsf{Sym}}}

\def\GQ{G\hspace{-0.2mm}Q}
\def\GP{G\hspace{-0.2mm}P}

\def\PP{\mathbb{P}}
\def\NN{\mathbb{N}}
\def\diag{\mathfrak{d}}

\newcommand{\ytabsmall}[1]{
\ytableausetup{boxsize = .35cm,aligntableaux=center}
{\small\begin{ytableau}  #1  \end{ytableau}}
}

\def\fkSO{\fkS^{\O}}
\def\fkSS{\fkS^{\Sp}}

\def\fkGS{\fkG^{\Sp}}
\def\fkGO{\fkG^{\O}}

\def\T{\mathsf{T}}

\def\diag{\mathsf{diag}}

\def\cyc{\mathsf{cyc}}

\newcommand{\abs}[1]{\lvert #1 \rvert}

\numberwithin{equation}{section}
\allowdisplaybreaks[1]
\UseCrayolaColors

\def\Ifpf{I^{\mathsf{fpf}}}

\def\half{\mathsf{half}_\leq}
\def\shalf{\mathsf{half}_<}

\def\lambdafpf{\lambda^{\mathsf{ss}}}
\def\I{I}
\def\DSym{\D^{\mathsf{sym}}}
\def\fkG{\mathfrak{G}}

\usepackage{bbm}

\def\hf{\mathrm{HF}}
\def\bhf{\mathrm{BHF}}

\def\codim{\mathrm{codim}}

\renewcommand{\mid}{:}

\def\bpi{\pi^{(\beta)}}
\def\bopi{{\overline\pi}^{(\beta)}}
\def\bpartial{\partial^{(\beta)}}
\def\oL{\overline L}
\def\ILambda{I\hspace{-0.2mm}\Lambda}
\def\CK{C\hspace{-0.2mm}K}

\newcommand{\flagvar}{\Fl}  
\newcommand{\Schub}{\mathfrak{S}}
\newcommand{\key}{\kappa}

\newcommand{\iso}{\cong}

\def\rowCounts{\rho}
\def\strictRowCounts{\widetilde\rho}

\def\columnCounts{\gamma}
\def\strictColumnCounts{\widetilde\gamma}

\def\otL{\overline{\widetilde{L}}}
\def\tqlascoux{\widetilde{L}^\O}
\def\tqlascouxatom{\otL{}^\O}

\def\cCO{\cY}
\def\cCSp{\cX}
\def\cKSp{\cZ}

\def\DO{D^\O}
\def\DSp{D^\Sp}

\def\cHSp{\cH^\Sp}
\def\cBSp{\cB^\Sp}
\def\hfSp{\hf^\Sp}
\def\bhfSp{\bhf^\Sp}
\def\cSp{c^\Sp}

\def\lambdaSp{\lambda^\Sp}
\def\lambdaO{\lambda^\O}

\def\qsymbol{\mathsf{Q}}
\def\qkey{\kappa^\qsymbol}
\def\qatom{\overline\kappa^\qsymbol}
\def\qlascoux{L^\qsymbol}
\def\qlascouxatom{\overline L{}^\qsymbol}

\def\psymbol{\mathsf{P}}
\def\pkey{\kappa^\psymbol}
\def\patom{\overline\kappa^\psymbol}
\def\plascoux{L^\psymbol}
\def\plascouxatom{\overline L{}^\psymbol}

\usepackage{fullpage}

\usepackage{listings}
\lstdefinelanguage{Sage}[]{Python}
{morekeywords={False,sage,True},sensitive=true}
\definecolor{dblackcolor}{rgb}{0.0,0.0,0.0}
\definecolor{dbluecolor}{rgb}{0.01,0.02,0.7}
\definecolor{dgreencolor}{rgb}{0.2,0.4,0.0}
\definecolor{dgraycolor}{rgb}{0.30,0.3,0.30}

\usepackage[backend=biber,url=false,isbn=false]{biblatex}
\begin{document}
\title{Key and Lascoux polynomials for symmetric orbit closures}
\author{
Eric Marberg \\ Department of Mathematics \\  Hong Kong University of Science and Technology \\ {\tt eric.marberg@gmail.com}
\and 
Travis Scrimshaw \\ Department of Mathematics \\ Hokkaido University \\ {\tt tcscrims@gmail.com}
}

\date{}

\maketitle

\begin{abstract}
We introduce shifted analogues of key polynomials related to symplectic and orthogonal orbit closures in the complete flag variety. Our definitions are given by applying  isobaric divided difference operators to the analogues of Schubert polynomials for orbit closures that correspond to dominant involutions. We show that our shifted key polynomials are linear combinations of key polynomials with nonnegative integer coefficients. We also prove that they are partial versions of the classical Schur $P$- and $Q$-polynomials. Finally, we examine $K$-theoretic generalizations of these functions, which give shifted forms of Lascoux polynomials. In the symplectic case, these generalizations are partial versions of the $GP$-polynomials introduced by Ikeda and Naruse. Besides developing basic properties, we identify a number of conjectures and open problems.
\end{abstract}

\setcounter{tocdepth}{2}
\tableofcontents

\section{Introduction}

The purpose of this article is to introduce ``shifted'' analogues of \defn{key polynomials} (and \defn{Lascoux polynomials}, their $K$-theoretic generalizations). Key polynomials are the characters of certain modules of the Borel subgroup of upper triangular matrices in the complex general linear group. Our constructions are motivated by the decomposition of \defn{Schubert polynomials} into key polynomials, which we review in the next section before summarizing our main results.
 
 \subsection{Schubert polynomials and key polynomials}

For each positive integer $n$, 
the \defn{complete flag variety} $\flagvar_n$ is the set of strictly increasing sequences of complex vector spaces $\{0 = V_0 \subsetneq V_1 \subsetneq \cdots \subsetneq V_n = \CC^n\}$.
After identifying $\flagvar_n$ as the quotient of the complex general linear group $\GL_n $ by the stabilizer of a fixed standard flag,
we have $\flagvar_n = B_n^- \backslash \GL_n$ where $B_n^-$ is the Borel subgroup of lower triangular matrices in $ \GL_n$.
The opposite Borel subgroup $B_n\subseteq \GL_n$ of invertible upper triangular matrices
acts on $\flagvar_n$ with finitely many orbits, which are naturally indexed by the symmetric group $S_n$.

The $B_n$-orbit associated to $w \in S_n$ is an open affine cell in $\Fl_n$
and its
closure is the \defn{Schubert variety} $X_w$.
We index these varieties so that the codimension of $X_w$ is the length $\ell(w)$ of the permutation $w \in S_n$.
A classical result of Borel~\cite{Borel53} gives an isomorphism
$
H^*(\flagvar_n; \ZZ) \iso \ZZ[x_1,x_2, \dotsc, x_n] / \ILambda_n
$
where $\ILambda_n$ is the ideal generated by
 the symmetric polynomials in $\ZZ[x_1,x_2,\dotsc,x_n] $ without constant term.
 The cohomology classes $[X_w]$ for $w \in S_n$ are a basis for this ring, and the \defn{Schubert polynomials}  $\Schub_w$~\cite{LS82,LS82II}
 are certain homogeneous elements of $\ZZ[x_1,x_2,\dotsc,x_n] $ satisfying $\Schub_w  + \ILambda_n = [X_w] \in H^*(\flagvar_n; \ZZ)$ under Borel's isomorphism.

Let $w_0:=n\cdots 321$ denote the reverse permutation in $S_n$.
For each integer partition $\lambda$ with at most $n$ parts, there is an irreducible highest weight $\GL_n$-representation $V(\lambda)$ whose character is the \defn{Schur function} $s_{\lambda}(x_1, \dotsc, x_n)$.
The Borel--Weil theorem provides a geometric definition of this object: $V(\lambda)$ is isomorphic to $\Gamma(\flagvar_n, L_{-\lambda})$, the module of global sections of a $\GL_n$-equivariant principal line bundle on $\flagvar_n$ whose fiber is the one-dimensional $B_n$-representation corresponding to the weight $-\lambda$. 
Replacing the complete flag variety $\Fl_n$ in this construction by the single Schubert variety $X_{w_0w}$ for $w \in S_n$
gives the definition of the \defn{Demazure module} $V_w(\lambda) := \Gamma(X_{w_0 w}, L_{-\lambda})$~\cite{Demazure74}.
The character of $V_w(\lambda)$ is the \defn{key polynomial} $\key_{w,\lambda} \in \NN[x_1,x_2,\dotsc,x_n] $~\cite{Andersen85,Demazure74II,Joseph85}.

Schubert polynomials and key polynomials are both easily computable via certain recurrence relations involving \defn{divided difference operators}, as will be reviewed in Section~\ref{coh-sect}.
Key polynomials have been extensively studied from a combinatorial perspective and have many interesting formulas; see, e.g.,~\cite{BBBG21,BSW,Kohnert91,LS1,LS4,Mason09,ReinerShimozono}.
We mention just a few notable properties here:
\begin{itemize}
\item 
Each key polynomial $\key_{w,\lambda}$ is a sum of certain \defn{atom polynomials} $\overline{\key}_{u,\lambda} \in \NN[x_1, x_2, \ldots]$, where the sum is over all $u \leq w$ in Bruhat order; see Section~\ref{coh-subsect1}.
The monomial-positivity of $\key_{w,\lambda}$ and $\overline\key_{w,\lambda}$ follows from tableau generating function formulas derived in~\cite{LS1}; see also~\cite[Thm.~1]{ReinerShimozono}.

\item The families of key and atom polynomials are both uniquely indexed by \defn{weak compositions} and give two different $\ZZ$-bases for the polynomial ring $\ZZ[x_1,x_2,\ldots]$~\cite{ReinerShimozono}; see \eqref{leading-eq}. 
Key and atom polynomials also both arise as specializations of \defn{non-symmetric Macdonald polynomials}~\cite{Mason09}.

\item One has $\key_{w_0,\lambda} = s_{\lambda}(x_1, \dotsc, x_n)$
by \cite[Remark, \S2]{ReinerShimozono}, and in this sense key polynomials are partial versions of Schur functions.
This formula reflects the fact that the character of $V(\lambda)$ is $s_\lambda(x_1,\dotsc,x_n)$ and $X_{w_0w}= \Fl_n$ when $w=w_0$, since the $B_n$-orbit of $1$ is dense.

\item Each  $\fkS_w$ is a linear combination of key polynomials with nonnegative integer coefficients~\cite{LS2},
and an algorithm is known to compute the terms in this decomposition~\cite[Thm.~4]{ReinerShimozono}.

\item A Schubert polynomial $\fkS_w$ is equal to a key polynomial if and only if $w \in S_n$ is \defn{vexillary} in the sense of being $2143$-avoiding~\cite[Thm.~24]{ReinerShimozono}.

\end{itemize}

\subsection{Key polynomials for orbit closures}

There are other groups $\K\subseteq\GL_n$ that act on the flag variety $\flagvar_n$ with finitely many orbits.
Such \defn{spherical subgroups} have been completely classified~\cite{Matsuki79}. We will focus on the case when $\K$ is either the orthogonal group $\O_n$ (for any $n$) or the symplectic group $\Sp_n$  (when $n$ is even).
These are instances of \defn{symmetric subgroups} in the sense of~\cite{RichSpring}.
The $\O_n$-orbits in $\Fl_n$ are uniquely indexed by all involutions $z=z^{-1}\in S_n$.
When $n$ is even, the $\Sp_n$-orbits in $\Fl_n$ are 
uniquely indexed by the  subset of involutions $z=z^{-1}\in S_n$ that are \defn{fixed-point-free} in the sense that $z(i)\neq i$ 
for all $i \in [n]$.

We denote the closures of the $\K$-orbits corresponding to $z=z^{-1} \in S_n$ by $X_z^\O$ and $X_z^\Sp$, respectively.
Wyser and Yong~\cite{WyserYong} identified a natural family of ``$\K$-Schubert polynomials'' that represent the cohomology classes
of these orbit closures, which have been further studied in~\cite{HMP6,HMP1,Pawlowski2019}.
We write these polynomials as $\fkSO_z$ and $\fkSS_z$ and refer to them as \defn{involution Schubert polynomials} of orthogonal and symplectic types; see Section~\ref{ischub-sect} for a self-contained algebraic definition.

In this paper, we begin the study of two new families of`shifted key polynomials, which we refer to as \defn{$P$- and $Q$-key polynomials}.
The ordinary key polynomial $\kappa_{w,\lambda}$ is obtained by applying the \defn{isobaric divided difference operator} $\pi_w$ to
the monomial $x^\lambda$, which is the Schubert polynomial of a \defn{dominant} permutation of shape $\lambda$ (as defined in Section~\ref{ischub-sect}). 
We construct our shifted key polynomials $\pkey_{w,\lambda}$ and $\qkey_{w,\lambda}$ by the same divided difference equation, but replacing $x^\lambda$
with certain \defn{dominant} involution Schubert polynomials (of symplectic type in the $P$-case and orthogonal type in the $Q$-case). 

This leads to an explicit algebraic definition of polynomials  $\pkey_{w,\lambda}$ and $\qkey_{w,\lambda}$ that will be presented in Section~\ref{coh-subsect2}.
We are motivated to study these functions by a number of interesting properties in parallel with key polynomials and conjectural connections to involution Schubert polynomials.
Most significantly:
\begin{itemize}
\item The $P$- and $Q$-key polynomials $\pkey_{w,\lambda}$ and $\qkey_{w,\lambda}$ are sums of \defn{$P$- and $Q$-atom polynomials} analogous to the usual atoms; see Section~\ref{coh-subsect2}. 
When the partition $\lambda$ is \defn{strict} in the sense of having all distinct parts (which we assume for the rest of this introduction),
the $P$- and $Q$-key polynomials, as well as the $P$- and $Q$-atom polynomials, all have nonnegative coefficients. In fact, shifted key polynomials are $\NN$-linear combinations of key polynomials,
while shifted atoms are $\NN$-linear combinations of atom polynomials by Theorem~\ref{thm:key_osp_key_positive}.

\item The families of $Q$-key and $Q$-atom polynomials are naturally indexed by weak compositions
that sort to \defn{symmetric partitions} (partitions that are invariant under transpose), which we refer to as \defn{symmetric weak compositions}.
The families of $P$-key and $P$-atom polynomials are similarly indexed by
certain \defn{skew-symmetric weak compositions} defined in Section~\ref{coh-subsect2}.

\item Just as key polynomials are partial versions of Schur functions,
$P$- and $Q$-key polynomials are partial versions of the \defn{Schur $P$- and $Q$-functions} $P_\lambda$ and $Q_\lambda$.
Specifically, if $n\gg 0$ then 
$\pkey_{w_0,\lambda} = P_{\lambda}(x_1, \dotsc, x_n)$
and $\qkey_{w_0,\lambda} = Q_{\lambda}(x_1, \dotsc, x_n)$
for $w_0=n\cdots 321 \in S_n$ by Theorem~\ref{stab-thm}.

\item One has $Q_\lambda = 2^{\ell(\lambda)} P_\lambda$ for all strict partitions $\lambda$.
A typical $Q$-key polynomial is not a ra\-tion\-al linear combination of $P$-key polynomials or vice versa. 
However, all $Q$-key polynomials indexed by symmetric weak compositions $\alpha = (\alpha_1,\alpha_2,\ldots)$ with first part $\alpha_1=0$ are equal to $P$-key polynomials multiplied by a power of two; see Theorem~\ref{alpha1-thm}. 
These scalar relationships do not extend to the $K$-theoretic generalizations of key polynomials that we discuss later.

\item Based on extensive calculations, we expect that the involution Schubert polynomials of symplectic and orthogonal types are sums of $P$- and $Q$-key polynomials, respectively.
These sums appear to be multiplicity-free in the symplectic case and have predictable powers of two as coefficients in the orthogonal case.
See Conjectures~\ref{fkSS-conj} and \ref{fkSO-conj}.

\item Modulo these conjectures, we can prove that 
 $\fkSO_z$ is equal to a $Q$-key polynomial if and only if $z$ is vexillary, while $\fkSS_z$ is equal to a $P$-key polynomial if and only if $z$ has an analogous
 \defn{fpf-vexillary} property. We can also characterize precisely when $P$- and $Q$-key polynomials are scalar multiples of ordinary key polynomials.
 This is discussed in Section~\ref{coincid-subsect}.
\end{itemize}
Developing these algebraic properties is the main focus of this article.

Despite the positive results listed above, proving shifted versions of many other basic facts about key and atom polynomials seems to be difficult.
For instance, we have not identified \emph{unique} indexing sets for the shifted keys and atoms, except possibly in the $Q$-key case (see Conjecture~\ref{sym-conj}).
Unlike ordinary keys and atoms, 
our families of shifted polynomials are all linearly dependent over $\ZZ$.
Whereas the leading term of a key polynomial has a very simple description~\eqref{leading-eq}, we only have conjectural formulas for the leading terms of 
$P$- and $Q$-key polynomials; see Section~\ref{leading-subsect}.
Furthermore, it is an open problem to find a geometric or representation-theoretic interpretation of our shifted key constructions.
There is at least a natural crystal-theoretic interpretation of $\pkey_{w,\lambda}$ and $\qkey_{w,\lambda}$, which
we will discuss in a subsequent paper.

\subsection{\texorpdfstring{$K$}{K}-theoretic generalizations}

A more modern approach to Schubert calculus has been to replace the cohomology ring with the (connective) $K$-theory ring.
Now, the Schubert polynomials $\Schub_w \in \NN[x_1,x_2,\dots] $ become the \defn{\mbox{($\beta$-)}Grothendieck polynomials} $\fkG_w \in \NN[\beta][x_1,x_2,\dots]$ introduced in~\cite{FominKirillov,LS82,LS82II}, which are defined in terms of $\beta$-deformed divided difference operators.
The polynomial $\fkG_w $ reduces to $\Schub_w$ when one sets $\beta=0$. There are similar \defn{symplectic and orthogonal (involution) Grothendieck polynomials} $\fkGS_z$ and $\fkGO_z$  
that generalize $\fkSS_z$ and $\fkSO_z$~\cite{MP2019a,WyserYong}, as will be reviewed in Section~\ref{k-sect}.

The Grothendieck polynomial of a dominant permutation is the same as the corresponding Schubert polynomial.
By applying
$\beta$-deformed divided difference operators to such polynomials, one obtains interesting 
$K$-theoretic analogues of key and atom polynomials,
called \defn{Lascoux polynomials} 
and \defn{Lascoux atoms}~\cite{Lascoux01,Mon16}.
These functions are partial versions of the \defn{symmetric Grothendieck polynomials} $G_\lambda$ introduced in~\cite{Buch2002} in the sense of Corollary~\ref{partial-lascoux-cor}.

One can define $K$-theoretic generalizations of the $P$-key and $P$-atom polynomials in
a similar way, namely, by applying $\beta$-deformed  divided difference operators to  dominant involution Grothendieck polynomials (which now are distinct from dominant involution Schubert polynomials).
In Section~\ref{p-lasc-sect}, we prove that these \defn{$P$-Lascoux polynomials} and \defn{$P$-Lascoux atoms}
are $\NN[\beta]$-linear combinations of Lascoux polynomials and Lascoux atoms, respectively. 
We also show that $P$-Lascoux polynomials are partial versions of 
Ikeda and Naruse's \defn{$K$-theoretic Schur $P$-functions} $\GP_\lambda$~\cite{IkedaNaruse} (Corollary~\ref{gp-dom-cor}).
Finally, we propose a $K$-theoretic generalization of Conjecture~\ref{fkSS-conj}: each symplectic Grothendieck polynomial $\fkGS_z$ appears to be a  multiplicity-free sum of $P$-Lascoux polynomials when $\beta=1$ (Conjecture~\ref{fkGS-conj}).
 
We discuss $K$-theoretic generalizations of $Q$-key polynomials in Section~\ref{q-lasc-sect}.
The details of this construction are more subtle compared to the classical and symplectic cases.
Although orthogonal Grothendieck polynomials $\fkGO_z$ are defined geometrically for all involutions $z =z^{-1} \in S_n$,
they only have an algebraic formula in terms of divided difference operators when $z$ is vexillary.
We expect that the correct notion of \defn{$Q$-Lascoux polynomials} should coincide with $\fkGO_z$ in the vexillary case.
Extending this definition to a family of polynomials $\qlascoux_\alpha$ indexed by all symmetric weak compositions is an open problem (Problem~\ref{prob:nonvex}).
Such an extension should at least have the following properties:
(1) each $\qlascoux_\alpha$ should reduce to $\qkey_\alpha$ when 
$\beta=0$, (2) the $Q$-Lascoux polynomials should give partial versions of Ikeda and Naruse's \defn{$K$-theoretic Schur $Q$-functions} $\GQ_\lambda$~\cite{IkedaNaruse} in the sense of Corollary~\ref{gp-dom-cor},
and (3) each $\fkGO_z$ should be a $\NN[\beta]$-linear combination of $\qlascoux_\alpha$'s.

Besides our treatment of $P$- and $Q$-Lascoux polynomials, Section~\ref{k-sect} also proves 
a number of new properties of involution Grothendieck polynomials, their stable limits, and the $\GP$- and $\GQ$-functions.  For example, we show that $\GP_\lambda$ and $\GQ_\lambda$ both have positive expansions 
into stable Grothendieck polynomials $G_\mu$'s and give precise bounds on the number of parts of $\mu$ for terms that appear (Theorems~\ref{eff-thm} and~\ref{GQ-thm}).
In the $\GQ$-case, this resolves  an open question from~\cite{MP2019a}.

\subsection{Open problems}

We conclude this introduction by highlighting the open problems that arise from this work in connection with Conjectures~\ref{fkSS-conj}, \ref{fkSO-conj}, and ~\ref{fkGS-conj}.

All of the polynomials we consider are defined algebraically, and so we would like these to be generating functions for combinatorial objects; see Problems~\ref{prob:monomial_expansion}, \ref{probl:key_expansion}, and~\ref{prob:K_monomial_expansion}.
The $P$- and $Q$-keys span different subspaces (Example~\ref{ex:contain_the_span}; Corollary~\ref{cor:span_noncontainment}), and so we want to determine the $\QQ$-subspaces they span; see Problem~\ref{space-prob}.
This would likely allow us to describe the $\QQ$-span and answer our leading term conjectures (Conjectures~\ref{leading-conj1} and~\ref{leading-conj2}).
We expect that being able to solve these open problems will be necessary steps to establishing our motivating conjectures.

\subsection{Outline}

This paper is organized as follows.
Section~\ref{coh-sect} starts with a quick review of the basic properties of key and atom polynomials, and then introduces our shifted analogues of these functions.  This section contains our main results and conjectures. Section~\ref{k-sect} then discusses $K$-theoretic generalizations 
of Schubert polynomials, key polynomials, and their shifted versions.

%

\section{Shifted key and atom polynomials}\label{coh-sect}

Throughout, $n$ is a positive integer, $[n] = \{1,2,\dotsc,n\}$, $\NN = \{0,1,2,\ldots\}$, and $\PP = \{1,2,3,\ldots\}$.
Let $\e_1,\e_2,\dots,\e_n$ be the standard basis for $\ZZ^n$ and 
let $x_1,x_2,x_3,\ldots$ be  commuting variables.
Given a formal power series $f \in \ZZ\llbracket x_1,x_2,\ldots\rrbracket$, we write $f(x_1,x_2,\dotsc,x_n)$ 
to denote the polynomial obtained from $f$ by setting $x_{n+1}=x_{n+2}=\cdots=0$.
When we say that a set $\cB$ is a basis for a submodule $M \subseteq \ZZ[x_1, x_2, \ldots]$, we mean every element of $M$ can be uniquely written as a finite $\ZZ$-linear combination of the elements in $\cB$.

Let $s_i = (i \; i+1)$ for $i \in \PP$ be the permutation of $\PP$ that interchanges $i$ and $i+1$ while fixing all other integers.
Define $S_\infty := \langle s_i \mid i \in \PP\rangle$ and $S_n := \langle s_i \mid i \in [n-1]\rangle$.
Let $\leq$ denote the \defn{(strong) Bruhat order} on $S_{\infty}$.
A \defn{reduced word} for $w \in S_\infty$ a minimal length sequence of integers $i_1i_2\cdots i_l$ with
$w=s_{i_1}s_{i_2}\cdots s_{i_l}$. We write $\ell(w) = l$ for the common length of every reduced word for $w$.

\subsection{Key and atom polynomials}\label{coh-subsect1}
 
 This subsection briefly reviews the definitions of key polynomials from~\cite{LS1,ReinerShimozono}.
 The group $S_\infty$ acts on  $\ZZ[x_1,x_2,\dots]$ by permuting variables.
For each $i \in \PP$, let $\partial_i $ be the \defn{divided difference operator} on $\ZZ[x_1,x_2,\ldots]$ defined in terms of this action by
$ 
\partial_i f := (f-s_if)/(x_i-x_{i+1})
$.  
The \defn{isobaric divided difference operators} $\pi_i$ for $i \in \PP$ are  
then given by
\be \pi_i f := \partial_i (x_i f) = f + x_{i+1} \partial_i f = (x_i f - x_{i+1} s_i f)/(x_i-x_{i+1}).\ee  
We also let $\overline\pi_i := \pi_i - 1$. Then
 $\partial_i ^2=0$, $\pi_i ^2=\pi_i$, and $\overline\pi_i ^2=-\overline\pi_i$, and all three families of operators satisfy the 
 Coxeter braid relations for $S_\infty$.
Thus for each $w \in S_\infty$ we can define
$ \partial_w := \partial_{i_1}\partial_{i_2}\cdots \partial_{i_l}$,
$\pi_w := \pi_{i_1}\pi_{i_2}\cdots \pi_{i_l}$,
and $\overline\pi_w := \overline\pi_{i_1}\overline\pi_{i_2}\cdots \overline\pi_{i_l}$,
where 
$i_1i_2\cdots i_l$ is any reduced word for $w$.

A \defn{weak composition} is a sequence $\alpha = (\alpha_1,\alpha_2,\ldots)$ of nonnegative integers with finite sum.
Given such a sequence, we set $ x^\alpha := \prod_{i} x_i^{\alpha_i}$ and $\abs{\alpha} := \sum_i \alpha_i$ and say that $\alpha$ is a weak composition of $\abs{\alpha}$.
When convenient, we identify the finite word $\alpha_1\alpha_2\cdots \alpha_l$ with the weak composition $(\alpha_1,\alpha_2,\dotsc,\alpha_l,0,0,\ldots)$.
Define $\ell(\alpha)$ for a weak composition $\alpha$ to be the largest index $i$ with $\alpha_i \neq 0$, or set $\ell(\alpha) = 0$ if $\alpha=\emptyset := (0,0,\ldots)$.
We also identify $\NN^n$ with the set of weak compositions $\alpha$ that have $\ell(\alpha) \leq n$.
A \defn{partition} is a weak composition that is weakly decreasing.

\begin{definition}\label{key-def}
Given $w \in S_\infty$ and a partition $\lambda$, let 
$ \kappa_{w,\lambda} := \pi_w x^\lambda
$
and
$
\overline\kappa_{w,\lambda} := \overline\pi_w x^\lambda.
$
\end{definition}

We refer to $\kappa_{w,\lambda}$ as a \defn{key polynomial}
and to $\overline\kappa_{w,\lambda}$ as an \defn{atom polynomial}.

\begin{example}\label{key-ex-1}
If $w = 3142 = s_2s_1s_3$ and $\lambda=(2,1,1,0)$ then
\[\ba \kappa_{w,\lambda} &= \pi_2\pi_1\pi_3 (x^{2110})
 = x^{2110} + x^{1210} + x^{1120} + x^{2101} + x^{2011} + x^{1201} + x^{1111} + x^{1021}
\ea\]
while
$\overline\kappa_{w,\lambda} = (\pi_2-1)(\pi_1-1)(\pi_3-1) (x^{2110})
 = x^{1021} + x^{1111}.$
\end{example}

Every key polynomial is a sum of atom polynomials,
since if $w \in S_\infty$ then $\pi_{w} = \sum_{v\leq w} \overline\pi_{v}$~\cite[\S3]{LS1}, and therefore
$
\kappa_{w,\lambda} = \sum_{v\leq w} \overline\kappa_{v,\lambda}
$ for any partition $\lambda$.

%

The indexing conventions in Definition~\ref{key-def} are redundant: If $\kappa_{v,\mu} = \kappa_{w,\lambda}$
then $\mu = \lambda$ but $v,w \in S_\infty$ may be different.
One can avoid this issue by using weak compositions to index key and atom polynomials. 
%
%
Given a weak composition $\alpha$, let $\lambda(\alpha)$ be the unique partition formed by rearranging its parts.
There is a left action of $S_{\infty}$ on weak compositions which permutes entries by
$
w \colon \alpha = (\alpha_1,\alpha_2,\alpha_3,\ldots) \mapsto (\alpha_{w^{-1}(1)}, \alpha_{w^{-1}(2)},  \alpha_{w^{-1}(3)},\ldots).
$
For any weak composition $\alpha$
there exists a unique 
minimal-length permutation $u(\alpha)\in S_\infty$ with $\alpha   = u(\alpha)\lambda(\alpha)$.

\begin{definition}\label{key-def1}
 For a weak composition $\alpha$,
let
$
 \kappa_{\alpha} :=\kappa_{u(\alpha),\lambda(\alpha)}$ 
and
$ \overline\kappa_{\alpha} := \overline\kappa_{u(\alpha),\lambda(\alpha)}$. 
 \end{definition}

In this notation, the polynomials in Example~\ref{key-ex-1} are $\kappa_{1021}$ and $\overline\kappa_{1021}$
since if  $\alpha = (1,0,2,1)$ then $\lambda(\alpha) = (2,1,1,0)$ and $u(\alpha) = 3142 = s_2s_1s_3$.

\begin{remark}\label{u-rem}
Suppose $\alpha$ is a weak composition and $i \in \PP$.
If $\alpha_i > \alpha_{i+1}$ then it is a routine exercise to check that $u(s_i\alpha) = s_i u(\alpha)$.
Since $\pi_i^2 = \pi_i$ and $\overline\pi_i^2 = -\overline\pi_i$,
it follows that
\be
 \pi_i  \kappa_{\alpha} =  \begin{cases}
 \kappa_{s_i\alpha} &\text{if }\alpha_i>\alpha_{i+1}, \\
 \kappa_{\alpha} &\text{if }\alpha_i<\alpha_{i+1},
 \end{cases} 
 \quand
  \overline\pi_i \overline\kappa_\alpha 
  =  \begin{cases}
 \overline\kappa_{s_i\alpha} &\text{if }\alpha_i>\alpha_{i+1}, \\
 -\overline\kappa_{\alpha} &\text{if }\alpha_i<\alpha_{i+1}.
 \end{cases} 
 \ee
If $\lambda := \lambda(\alpha)$ and $J := \{ j \in \PP \mid \lambda_j=\lambda_{j+1}\}$,
then the stabilizer subgroup of $\lambda$ in $S_\infty$
is the standard parabolic subgroup $\langle s_j : j \in J\rangle$,
and if $\alpha_i = \alpha_{i+1}$ then $s_i u(\alpha) = u(\alpha) s_j$ for some $j \in J$ (e.g., by~\cite[Prop.~2.4.4]{CCG}). As $\pi_j x^\lambda = x^\lambda$ for all $j \in J$,
this means that 
$
\pi_i \kappa_\alpha = \kappa_{\alpha}$ and $  \overline\pi_i\overline\kappa_\alpha =0$ if $\alpha_i=\alpha_{i+1}.
$
\end{remark}

The properties in the preceding remark make it clear that each $\kappa_{w,\lambda}$  
is equal to $\kappa_\alpha$  for some $\alpha$. Similarly, each $\overline\kappa_{w,\lambda}$
is either zero or equal to $\overline\kappa_\alpha$  for some $\alpha$.
Write $<_{\mathrm{lex}}$ for \defn{lexicographic order} on weak compositions that has $\alpha <_{\mathrm{lex}} \gamma$
if there exists a positive index $j$ such that $\alpha_j < \gamma_j$ and $\alpha_i=\gamma_i$ for all $1\leq i < j$.
It turns out~\cite[Cor.~7]{ReinerShimozono} that
\be\label{leading-eq} \kappa_\alpha \in x^\alpha +  \NN\spanning\{ x^\delta: \alpha <_{\mathrm{lex}} \delta\}
\quand
\overline\kappa_\alpha \in x^\alpha + \NN\spanning\{ x^\delta: \alpha <_{\mathrm{lex}} \delta\},
\ee
so weak compositions uniquely index  key and atom polynomials.
More strongly, the identities~\eqref{leading-eq} imply that
 $\{\kappa_\alpha\}$ and $\{\overline\kappa_\alpha\}$ are both $\ZZ$-bases for $\ZZ[x_1,x_2,\dots]$.


\subsection{Shifted analogues}\label{coh-subsect2}

We now investigate two shifted analogues of $\kappa_\alpha$.
Key polynomials are closely related to Schubert polynomials and are generalizations of Schur polynomials.
Our shifted analogues are motivated by similar connections to
certain \defn{involution Schubert polynomials} from~\cite{WyserYong} 
and will turn out to be generalizations of \defn{Schur $P$- and $Q$-polynomials}.  

Schur $P$- and $Q$-polynomials are generating functions for semistandard shifted tableaux~\cite[\S{III}.8]{Macdonald}, whereas Schur polynomials
are generating functions for semistandard tableaux. This accounts for our use of the term ``shifted''.

We will explain these motivations more fully in Section~\ref{ischub-sect}.
Here, we focus on the definitions and basic properties. 
The \defn{Young diagram} of a partition $\lambda$ is the set of integer pairs 
$\D_\lambda := \{ (i,j) \in \PP\times \PP : 1\leq j \leq \lambda_i\}$ (which we will draw in English convention, with positions oriented as in a matrix).
A partition $\lambda$ is \defn{strict} if its nonzero parts are all distinct.

 \begin{definition}\label{spo-key-def}
Given a permutation $w \in S_\infty$ and a strict partition $\mu$, define 
\[ 
\ba
\pkey_{w,\mu} := \pi_w \(\prod_{(i,j) \in \D_\mu } (x_i + x_{i+j})\)
\quand
\qkey_{w,\mu} := \pi_w \(\prod_{(i,j) \in \D_\mu } (x_i + x_{i+j-1})\).
\ea
\]
We refer to these functions as \defn{(shifted) $P$- and $Q$-key polynomials}.
We also define
\[
\ba
\patom_{w,\mu} := \overline\pi_w \(\prod_{(i,j) \in \D_\mu} (x_i + x_{i+j})\)
\quand
\qatom_{w,\mu} := \overline\pi_w \(\prod_{(i,j) \in \D_\mu} (x_i + x_{i+j-1})\).
\ea
\] 
We refer to these functions as \defn{(shifted) $P$- and $Q$-atom polynomials}.
\end{definition}

The coefficients of $\qkey_{w,\mu}$ and $\qatom_{w,\mu}$ are all divisible by $2^{\ell(\mu)}$ since this is evidently true of the product $\prod_{(i,j) \in \D_\mu } (x_i + x_{i+j-1})$, which is divisible by $\prod_{i=1}^{\ell(\mu)} 2 x_i$ (by taking $j = 1$).
 
\begin{example}\label{sp-o-ex1}
If $w = 3142 = s_2s_1s_3$ and $\mu=(3,1,0,0)$ then
\[\ba \pkey_{w,\mu} &= \pi_2\pi_1\pi_3 \bigl( (x_1+x_2)(x_1+x_3)(x_1+x_4)(x_2+x_3) \bigr)
 = \kappa_{0022} + \kappa_{0031} + \kappa_{0112}
\\
 \qkey_{w,\mu} &= \pi_2\pi_1\pi_3 (4x_1x_2(x_1+x_2)(x_1+x_3))
 = 4 \kappa_{103} + 4 \kappa_{202} + 4 \kappa_{1021}
\ea\]
while
\[
\ba
\patom_{w,\mu} &= (\pi_2-1)(\pi_1-1)(\pi_3-1) ((x_1+x_2)(x_1+x_3)(x_1+x_4)(x_2+x_3))= \overline\kappa_{0022} + \overline\kappa_{0031}
\\
\qatom_{w,\mu} &= (\pi_2-1)(\pi_1-1)(\pi_3-1) (4x_1x_2(x_1+x_2)(x_1+x_3))
 =0.
 \ea
 \]
\end{example}

Shifted key polynomials are sums of shifted atom polynomials.

\begin{proposition}
We have
$
\pkey_{w,\mu}
=
\sum_{v\leq w} \patom_{v,\mu}
$
and
$
\qkey_{w,\mu}
=
\sum_{v\leq w} \qatom_{v,\mu}
$.
\end{proposition}

\begin{proof}
As for key polynomials, this follows from the identity
$\pi_{w}
=
\sum_{v\leq w} \overline\pi_{v}
$~\cite[\S3]{LS1}. 
\end{proof}

\begin{example}
The definitions of $\pkey_{w,\mu}$ and $\qkey_{w,\mu}$ make sense if $\mu$ is any partition, 
but if $\mu$ is not strict
then these polynomials may have  negative coefficients or be zero. For example:
\[
\ba
\pkey_{4213,2210}&=&\pi_1\pi_3\pi_2\pi_1\pkey_{2210} &= -x_1^2 x_2 x_3 x_4 - x_1 x_2^2 x_3 x_4 + x_1 x_3^2 x_4^2 + x_2 x_3^2 x_4^2, \\
\pkey_{3421,2210}&=&\pi_1\pi_2\pi_1\pi_3\pi_2\pkey_{2210} &= -x_1 x_2 x_3 x_4^2 - x_1 x_2 x_3^2 x_4 - x_1 x_2^2 x_3 x_4 - x_1^2 x_2 x_3 x_4, \\
\qkey_{4213,3321} &=&\pi_1\pi_3\pi_2\pi_1\qkey_{3321}&= -16 x_1^3 x_2^2 x_3^2 x_4^2 - 16 x_1^2 x_2^3 x_3^2 x_4^2 + 16 x_1^2 x_2 x_3^3 x_4^3 + 16 x_1 x_2^2 x_3^3 x_4^3, \\
\qkey_{3421,3321}&=&\pi_1\pi_2\pi_1\pi_3\pi_2\qkey_{3321} &= -16x_1^2 x_2^2 x_3^2 x_4^3 - 16 x_1^2 x_2^2 x_3^3 x_4^2 - 16 x_1^2 x_2^3 x_3^2 x_4^2 - 16 x_1^3 x_2^2 x_3^2 x_4^2.
\ea
\]
\end{example}

When $\mu$ is strict, each $\pkey_{w,\mu}$ and  $ \qkey_{w,\mu}$ has the following positivity property.

\begin{theorem}
\label{thm:key_osp_key_positive}
Suppose $w \in S_\infty$ and $\mu$ is a strict partition.
Then $\pkey_{w,\mu}$ and  $ \qkey_{w,\mu}$ are nonzero linear combinations of 
key polynomials $\kappa_\alpha$ with nonnegative integer coefficients. 
Similarly,
$\patom_{w,\mu}$ and  $\qatom_{w,\mu}$ are linear combinations of 
atom polynomials $\overline\kappa_\alpha$ with nonnegative integer coefficients. 
\end{theorem}

We omit the proof of
Theorem~\ref{thm:key_osp_key_positive} since it will follow by setting $\beta=0$ in 
 Propositions~\ref{sp-lascoux-positive-prop} and \ref{o-lascoux-positive-prop}. We prove these more general results in Section~\ref{k-sect}.

\begin{corollary}
\label{cor:monomial_positive}
If $\mu$ is strict then $\pkey_{w,\mu}$,  $\qkey_{w,\mu}$, $\patom_{w,\mu}$, and $\qatom_{w,\mu}$ all belong to
$\NN[x_1,x_2,\dots]$.
\end{corollary}

While one can compute the monomial expansions of the $P$- and $Q$-atoms and keys using the divided difference operators, a natural open problem suggested by Corollary~\ref{cor:monomial_positive} is the following.

\begin{problem}
\label{prob:monomial_expansion}
Determine combinatorial objects and weight functions whose corresponding generating functions are equal to $\pkey_{w,\mu}$,  $\qkey_{w,\mu}$, $\patom_{w,\mu}$, and $\qatom_{w,\mu}$.
\end{problem}
 
 The following gives a case where we can explicitly compute the relevant key expansions.
 
\begin{proposition}
\label{prop:hook_shapes}
For all integers $n > 0$, we have
\[
\pkey_{1,(n)} = \key_{01^n} + \sum_{m=2}^n \key_{m0^m1^{n-m}}
\quand
\qkey_{1,(n)}= 2 \sum_{m=1}^n \key_{m0^{m-1}1^{n-m}}.
\]
\end{proposition}

\begin{proof}
We want to express
$
\pkey_{1,(n)} = \prod_{j=1}^n (x_1 + x_{1+j})
$
and
$
\qkey_{1,(n)} = \prod_{j=1}^n (x_1 + x_j) = 2 x_1 \prod_{j=2}^n (x_1 + x_j)
$
in terms of key polynomials. It is a straightforward exercise in algebra 
 to check directly
from the definition in terms of divided difference operators that for all integers $m,k>0$ and $j \geq 0$
one has 
\be
\label{eq:hook_type}
\key_{m0^j1^k} = (\pi_{j+1} \dotsm \pi_{j+k}) \dotsm (\pi_2 \dotsm \pi_{k+1}) \bigl( x_1^m x_2 \dotsm x_{k+1} \bigr) = x_1^m \sum_{1 < i_1 < \dotsc < i_k \leq j+k} x_{i_1} \dotsm x_{i_k}
\ee
as well as $\key_{0^j1^k}  =  \sum_{1 \leq i_1 < \dotsc < i_k \leq j+k} x_{i_1} \dotsm x_{i_k}$.
The desired formulas follow from the binomial theorem and these identities
by comparing terms with $x_1^m$ on both sides.
\end{proof}



The sum in \eqref{eq:hook_type} is the $k$th elementary symmetric function $e_k(x_2, x_3,\dotsc, x_{j+k})$.
Via this equation, replacing the variables $x_1,x_2,x_3,\dots,x_{n+1}$ by $t,x_1,x_2,\dots,x_n$ turns the formula for $\pkey_{1,(n)}$
in Proposition~\ref{prop:hook_shapes} into the well-known identity $\prod_{i=1}^n (t + x_i) = \sum_{k=0}^n t^{n-k} e_k(x_1,x_2,\dots,x_n)$.
Making these substitutions in the formula for $\qkey_{1,(n+1)}$ results in the same identity multiplied by $2t$.

\begin{remark}\label{tabular-rmk}
For more key expansions of $P$- and $Q$-key polynomials, see Tables~\ref{table:pkey_expansions} and \ref{table:qkey_expansions} below.
In many small examples like the ones in these tables,
the key expansions of $\pkey_{w,\mu}$ and  $2^{-\ell(\mu)} \qkey_{w,\mu}$ 
are multiplicity-free.
This does not hold in general: 
There are key terms with coefficients $1$, $2$, and $3$ in 
the expansions of both
$\pkey_{w,\mu}$ and  $2^{-\ell(\mu)} \qkey_{w,\mu}$ 
for $w=1$ and $\mu = (8, 5, 4, 3)$. 
The numbers 1, 2, 3, and 4 are all coefficients in the key expansion of $\pkey_{w,\lambda}$
for $w=54321$ and $\lambda=(7,5,3,1)$.
\end{remark}

As with $\kappa_{w,\lambda}$ and $\overline\kappa_{w,\lambda}$, our initial notation 
 in 
Definition~\ref{spo-key-def} allows us to specify the same polynomial in many different ways.
We would like to identify more natural indexing sets for these polynomials.
However, we cannot simply index $\pkey_{w,\mu}$,  $\qkey_{w,\mu}$, $\patom_{w,\mu}$, and $\qatom_{w,\mu}$
by the compositions formed by permuting  $\mu$,
since the products in Definition~\ref{spo-key-def} are not always fixed by $\pi_i$
when $\mu_i=\mu_{i+1}$. Instead, we propose a different idea.

A partition $\lambda$ is \defn{symmetric} if $\lambda=\lambda^\T$, or equivalently if $\D_\lambda = \{ (j,i) : (i,j) \in \D_\lambda\}$.
We define a weak composition $\alpha$ to be \defn{symmetric} if  $\lambda(\alpha)=\lambda(\alpha)^\T$.
For any symmetric partition $\lambda$, 
define $\shalf(\lambda)$ and $\half(\lambda)$ to be the strict partitions
whose nonzero parts are $\{ \lambda_i - i \mid i \in \PP\} \cap \PP$
and $\{ \lambda_i - (i-1) \mid i \in \PP\} \cap \PP$, respectively. For example,
if 
\[ \lambda = (4,3,3,1)= \ytabsmall{ \ & \ & \ & \ \\ \ & \ & \ \\ \ & \ & \ \\ \ }
\quad\text{then}\quad \shalf(\lambda) = (3,1)\quand \half(\lambda)=(4,2,1).\]
The weak composition $\alpha = (3,0,1,4,0,0,3)$ is symmetric since $\lambda(\alpha) = (4,3,3,1)$ is a symmetric partition.

\begin{definition}
For each symmetric weak composition $\alpha$ with $u=u(\alpha)$ and $\lambda=\lambda(\alpha)$,
 define 
\[
  \pkey_\alpha  := \pkey_{u, \shalf(\lambda)}\text{\ \ and\ \ } \patom_\alpha  := \patom_{u, \shalf(\lambda)}
  \quad\text{along with}\quad
 \qkey_\alpha  := \qkey_{u, \half(\lambda)}
 \text{\ \ and\ \ } \qatom_\alpha  := \qatom_{u, \half(\lambda)}
.\]
\end{definition}

This notation is automatically well-defined, but it is not yet clear that every shifted key/atom polynomial
can be written in this form.
The polynomials from Example~\ref{sp-o-ex1} are given in this notation as
$ \pkey_{3143}$ and  $\patom_{3143}$
along with
$ \qkey_{2031}$ and  $\qatom_{2031}$.

\begin{lemma}\label{lam-lam-lem}
If $\lambda=\lambda^\T$ is a symmetric partition then
\[
  \Bigl\{ i \in \PP \mid \lambda_i=\lambda_{i+1}\Bigr\} \subseteq \Bigl\{ i \in \PP \mid \pi_i \pkey_\lambda= \pkey_\lambda\Bigr\}
  \text{ and }
 \Bigl\{ i \in \PP \mid \lambda_i=\lambda_{i+1}\Bigr\} =\Bigl\{ i \in \PP \mid \pi_i \qkey_\lambda= \qkey_\lambda\Bigr\}
 .
  \]
\end{lemma}

\begin{proof}
Suppose $\lambda=\lambda^\T$ is a symmetric partition.
If $j:= \lambda_i =\lambda_{i+1} \geq i+1 $ then $\qkey_\lambda$ is equal to 
\[
 \prod_{1\leq a < i} (x_a + x_i)(x_a+ x_{i+1}) \cdot 4x_ix_{i+1}(x_i + x_{i+1}) \cdot \prod_{i +1 < b \leq j} (x_i+x_b)(x_{i+1}+x_b)
\]
(which is symmetric in $x_i$ and $x_{i+1}$)
times a polynomial not involving $x_i$ or $x_{i+1}$, so $s_i \qkey_\lambda= \qkey_\lambda$ and $\pi_i \qkey_\lambda =\qkey_\lambda$.
In this case $\pkey_\lambda$ is equal to 
the polynomial in the displayed equation divided by $ 4x_ix_{i+1}$ times a polynomial not involving $x_i$ or $x_{i+1}$,
so   $s_i \pkey_\lambda= \pi_i \pkey_\lambda =\pkey_\lambda$.
We cannot have $\lambda_i =\lambda_{i+1} = i$ since $\lambda = \lambda^\T$.
If $j := \lambda_i =\lambda_{i+1} < i$, then $\qkey_\lambda$ and $\pkey_\lambda$ are both equal to $ \prod_{1\leq a \leq j} (x_i+ x_a)(x_{i+1}+ x_a)$ times polynomials not involving $x_i$ or $x_{i+1}$, so both are again fixed by $s_i$ and $\pi_i$.
If $ \lambda_{i} >\lambda_{i+1}$ then it holds similarly that $\qkey_\lambda$ is equal to a polynomial that is not symmetric in $x_i$ and $x_{i+1}$ times a polynomial not involving $x_i$ or $x_{i+1}$, so we have $s_i\qkey_\lambda\neq \qkey_\lambda$ and  $\pi_i\qkey_\lambda\neq \qkey_\lambda$.
\end{proof}

\begin{proposition}\label{formulas-prop}
Suppose $\alpha$ is a symmetric weak composition and $i \in \PP$.
\ben
\item[(a)] If $\alpha_i > \alpha_{i+1}$ then 
$ \pi_i \pkey_\alpha =  \pkey_{s_i\alpha}
$ and
$\pi_i \qkey_\alpha =\qkey_{s_i\alpha}$ and
$\overline\pi_i \patom_\alpha =\patom_{s_i\alpha}$
and
$  \overline\pi_i \qatom_\alpha = \qatom_{s_i\alpha}$.

\item[(b)] If $\alpha_i < \alpha_{i+1}$ then 
$ \pi_i \pkey_\alpha =  \pkey_{\alpha}
$ and
$\pi_i \qkey_\alpha =\qkey_{\alpha}$ and
$\overline\pi_i \patom_\alpha =-\patom_{\alpha}$
and
$  \overline\pi_i \qatom_\alpha = -\qatom_{\alpha}$.

\item[(c)] If $\alpha_i = \alpha_{i+1}$ then 
$ \pi_i \pkey_\alpha =  \pkey_{\alpha}
$ and
$\pi_i \qkey_\alpha =\qkey_{\alpha}$ and
$\overline\pi_i \patom_\alpha =   \overline\pi_i \qatom_\alpha = 0$.
\een
 \end{proposition}
 
 \begin{proof}
This follows by the reasoning in Remark~\ref{u-rem} using Lemma~\ref{lam-lam-lem}.
 \end{proof}

\def\vgap{\\[-12pt]}
\begin{table}
{\small\begin{alignat*}{3}
\pkey_{\emptyset} &= \pkey_{1,\emptyset} &&= \key_{\emptyset} 
\vgap\\
\pkey_{22} &= \pkey_{1,1}&&=\key_{01}
\vgap\\
\pkey_{311} &= \pkey_{1,2}&& = \key_{011} + \key_{2}
\vgap\\ 
\pkey_{333} &= \pkey_{1,21}&&=  \key_{012}
\\
\pkey_{4111} &= \pkey_{1,3}&&= \key_{0111} + \key_{2001} + \key_{3}
\vgap\\ 
\pkey_{4311} &= \pkey_{1,31}&& = \key_{0121} + \key_{202} + \key_{301}
\\
\pkey_{51111} &= \pkey_{1,4}&& = \key_{01111} + \key_{20011} + \key_{30001} + \key_{4}
\vgap\\ 
\pkey_{4422} &= \pkey_{1,32}&& = \key_{0122} + \key_{0311} + \key_{23}
\\
\pkey_{53311} &= \pkey_{1,41}&&= \key_{01211} + \key_{20201} + \key_{30101} + \key_{302} + \key_{401}
\\
\pkey_{611111} &= \pkey_{1,5}&&  = \key_{011111} + \key_{200111} + \key_{300011} + \key_{400001} + \key_{5}
\vgap\\ 
\pkey_{4444} &= \pkey_{1,321}&& = \key_{0123}
\\
\pkey_{54221} &= \pkey_{1,42}&& = \key_{01221} + \key_{03111} + \key_{2022} + \key_{23001} + \key_{3012} + \key_{33} + \key_{4011} + \key_{42}
\\
\pkey_{633111} &= \pkey_{1,51}&&= \key_{012111} + \key_{202011} + \key_{301011} + \key_{302001} + \key_{401001} + \key_{402} + \key_{501}
\\
\pkey_{7111111} &= \pkey_{1,6}&& = \key_{0111111} + \key_{2001111} + \key_{3000111} + \key_{4000011} + \key_{5000001} + \key_{6}
\end{alignat*}
\caption{Key expansions of $\pkey_{\lambda}$ up to degree $6$ for skew-symmetric partitions $\lambda$.}\label{table:pkey_expansions}}
\end{table}

\begin{table}
{\small\begin{alignat*}{3}
\qkey_{\emptyset}  & = \qkey_{1,\emptyset} &&= \key_{\emptyset} 
\vgap\\ 
\qkey_{1} &= \qkey_{1,1} &&= 2\key_{1}
\vgap\\ 
 \qkey_{21} & = \qkey_{1,2} &&= 2(\key_{11} + \key_{2})
\vgap\\ 
 \qkey_{22} & = \qkey_{1,21} &&= 4\key_{12}
\\
 \qkey_{311} & = \qkey_{1,3} &&= 2(\key_{111} + \key_{201} + \key_{3})
\vgap\\ 
 \qkey_{321} & = \qkey_{1,31} &&= 4(\key_{121} + \key_{22} + \key_{31})
\\
\qkey_{4111} & = \qkey_{1,4} &&= 2(\key_{1111} + \key_{2011} + \key_{3001} + \key_{4})
\vgap\\ 
 \qkey_{332} & = \qkey_{1,32} &&= 4(\key_{122} + \key_{131} + \key_{23})
\\
 \qkey_{4211} &= \qkey_{1,41} &&= 4(\key_{1211} + \key_{2201} + \key_{3101} + \key_{32} + \key_{41})
\\
\qkey_{51111} & = \qkey_{1,5} &&= 2(\key_{11111} + \key_{20111} + \key_{30011} + \key_{40001} + \key_{5})
\vgap\\ 
 \qkey_{333} & = \qkey_{1,321} &&= 8 \key_{123}
\\
\qkey_{4321} & = \qkey_{1,42} &&= 4(\key_{1221} + \key_{1311} + \key_{222} + \key_{2301} + \key_{312} + \key_{33} + \key_{411} + \key_{42})
\\
 \qkey_{52111} &= \qkey_{1,51} &&= 4(\key_{12111} + \key_{22011} + \key_{31011} + \key_{32001} + \key_{41001} + \key_{42} + \key_{51})
\\
 \qkey_{611111} &= \qkey_{1,6} &&= 2(\key_{11111} + \key_{20111} + \key_{30011} + \key_{40001} + \key_{500001} + \key_{6})
\end{alignat*}
\caption{Key expansions of $\qkey_{\lambda}$ up to degree $6$ for symmetric partitions $\lambda$.}\label{table:qkey_expansions}}
\end{table}

Using symmetric weak compositions $\alpha$ to index $P$-key polynomials is problematic when the inclusion in Lemma~\ref{lam-lam-lem} is not an equality
for $\lambda =\lambda(\alpha)$.
To correct this, we will work with a smaller set of indices.
 
An \defn{addable corner} of a partition $\lambda$ is a cell $(i,j) \notin \D_\lambda$ such that $\D_\lambda \sqcup \{(i,j)\} = \D_\mu$
for some partition $\mu$.
A \defn{removable corner} of $\lambda$ is a cell $(i,j) \in \D_\lambda$ such that $\D_\lambda \setminus \{(i,j)\} = \D_\nu$
for some partition $\nu$. 
We define a partition $\lambda$ to be \defn{skew-symmetric} if $\lambda=\lambda^\T$ and both
\ben
\item[(a)] if $(i,i)$ is a removable corner of $\lambda$ then  $(i,i+1)$ is not an addable corner, and
\item[(b)] if $(i,i)$ is an addable corner of $\lambda$ then $(i,i-1)$ is not a removable corner.
\een
For example, the symmetric partitions 
\be\label{ss-ex1}
\ytableausetup{boxsize = .3cm,aligntableaux=center}
(3,2,1)=\begin{ytableau}
\  & \ &  \ \\ 
\ & \   & \none[\cdot]  \\ 
\ & \none[\cdot] & \none[\cdot]    
 \end{ytableau}
 \qquand
 (3,3,2)=\begin{ytableau}
\  & \ &  \ \\
 \ &  \  & \  \\ 
 \ &\  & \none[\cdot] 
 \end{ytableau}
\ee
are  not skew-symmetric, but these symmetric partitions are skew-symmetric:
\be\label{ss-ex2}
(3,1,1)=\begin{ytableau}
\  & \ &  \ \\ 
\ & \none[\cdot]    & \none[\cdot]  \\ 
\ & \none[\cdot] & \none[\cdot]    
 \end{ytableau}
 \qquand
 (3,3,3)=\begin{ytableau}
\ &\  &\ \\
 \ &  \  & \  \\ 
\  & \ &  \
 \end{ytableau}
 \ee
 The Young diagram of any symmetric partition $\lambda$ either has a unique addable corner of the form $(i,i)$ 
 or a unique removable corner of the form $(i,i)$. If we define $\nu$ from $\lambda$ by adding $(i,i)$ to the Young diagram
 in the first case and by removing $(i,i)$ in the second, then exactly one of $\lambda$ or $\nu$ is skew-symmetric 
 and it holds that $\pkey_\lambda = \pkey_\nu$
 and $\patom_\lambda = \patom_\nu$.
The examples in~\eqref{ss-ex1} and~\eqref{ss-ex2} readily generalize to show these properties.

For each strict partition $\mu$, there is a unique symmetric (respectively, skew-symmetric) partition $\lambda$ such that $\mu = \half(\lambda)$ (respectively, $\mu = \shalf(\lambda)$).
It is often useful to identify a strict partition $\mu$ with its \defn{shifted Young diagram} 
$\SD_\mu:= \{ (i,i+j-1) : (i,j) \in \D_\mu\}.$
Notice that if $\mu = \half(\lambda)$  then $\SD_\mu$ is the set of positions $(i,j) \in \D_\lambda$ with $i\leq j$. Similarly, if $\mu = \shalf(\lambda)$  then $\SD_\mu$ is given by translating the set of positions $(i,j) \in \D_\lambda$ with $i< j$ one column to the left.
  
The sequence of numbers counting the symmetric partitions of $n$ also counts the partitions of $n$ with all odd parts, which appears as~\cite[A000700]{OEIS}.
The sequence
\[
  1, 0, 0, 0, 1, 1, 0, 1, 0, 2, 0, 2, 1, 2, 1, 2, 3, 2, 3, 2, 5, 3, 5, 3, 7, 5, 7, 6, 9, 8, 9, \ldots
\]
counting the skew-symmetric partitions of $n$ does not presently appear in~\cite{OEIS}.
It can be generated by the following \textsc{SageMath}~\cite{sage} code:
\begin{lstlisting}
def num_skew_symmetric(n):
    ret = [1] + [0] * n  # special case for the empty partition
    for k in range(1, n // 2 + 2):
        for la in Partitions(k, max_slope=-1):
            size = 2 * sum(la) + len(la) + int(la[-1] == 1)
            if size <= n:
                ret[size] += 1
    return ret
\end{lstlisting}

\begin{lemma}\label{ss-lam-lam-lem}
One has $\Bigl\{ i \in \PP \mid \lambda_i=\lambda_{i+1}\Bigr\} =\Bigl\{ i \in \PP \mid \pi_i \pkey_\lambda= \pkey_\lambda\Bigr\}$
for a symmetric partition $\lambda=\lambda^\T$
 if and only if 
 $\lambda$ is skew-symmetric. 
\end{lemma}

\begin{proof}
Fix a partition $\lambda$ such that $\lambda = \lambda^\T$. 
If $(i,i)$ is a removable corner of $\lambda $ and $(i,i+1)$ is an addable corner,
then $\lambda_i = i > \lambda_{i+1} = i-1$ and $s_i \pkey_\lambda=\pi_i \pkey_\lambda =\pkey_\lambda$.
Likewise, if $(i,i)$ is an addable corner of $\lambda$ and $(i,i-1)$ is a removable corner,
then $\lambda_i = i-1 <\lambda_{i-1} = i$ and $s_{i-1} \pkey_\lambda=\pi_{i-1} \pkey_\lambda =\pkey_\lambda$.
Thus if $\lambda$ is not skew-symmetric then there exists $i \in \PP$ with 
$\lambda_i \neq \lambda_{i+1}$ and
$\pi_i \pkey_\lambda=\pkey_\lambda$.

Conversely, by factoring $\pkey_\lambda$ as in the proof of Lemma~\ref{lam-lam-lem},
one checks that the only way to have $ \lambda_{i} \neq \lambda_{i+1}$ and 
$\pi_i \pkey_\lambda =\pkey_\lambda$ (which is equivalent to $s_i \pkey_\lambda =\pkey_\lambda$)
is if $\lambda_i = i > \lambda_{i+1}=i-1$ or $\lambda_i = i+1 > \lambda_{i+1} = i$,
in which case $\lambda$ is not skew-symmetric.
\end{proof}

A weak composition $\alpha$ is \defn{skew-symmetric} if it sorts to a skew-symmetric partition $\lambda(\alpha)$.
Notice that skew-symmetric compositions (respectively, partitions) make up a subset of symmetric compositions (respectively, partitions),
rather than a disjoint class.

\begin{proposition}
Each $Q$-key (respectively, nonzero $Q$-atom) polynomial occurs as $ \qkey_\alpha$  (respectively, $ \qatom_\alpha$)  for some symmetric weak composition $\alpha$.
Each $P$-key (respectively, nonzero $P$-atom) polynomial occurs as $ \pkey_\alpha$  (respectively, $ \patom_\alpha$) for some skew-symmetric weak composition $\alpha$.
\end{proposition}

\begin{proof}
Suppose $w \in S_\infty$ and $\mu$ is a strict partition.
If $\lambda$ is the symmetric partition with $\half(\lambda) =\mu$ and $i_1\cdots i_l$ is a reduced word for $w$, 
then $ \qkey_{w,\mu} = \pi_{i_1} \cdots \pi_{i_l} \qkey_{\lambda}$
and 
$ \qatom_{w,\mu} = \overline\pi_{i_1}\cdots \overline\pi_{i_l} \qatom_{\lambda}$.
When nonzero, these polynomials have the form $\qkey_\alpha$ and $\qatom_\alpha$
by Proposition~\ref{formulas-prop}.
%
 %
The second claim follows by 
a similar argument starting with $\lambda$ as the skew-symmetric partition having $\shalf(\lambda) =\mu$.
%
\end{proof}

These (skew-)symmetric weak composition indexing sets are still larger than optimal.
The $P$- and $Q$-atom polynomials can both be zero,
and may coincide when nonzero. For example,
\be\qatom_{012} = \patom_{022} = 0,
\qquad 
\qatom_{30023} = \qatom_{21014} \neq 0,
\qquand
\patom_{402402} = \patom_{313501}\neq 0.
\ee
There are also coincidences among $P$-key polynomials (but see Corollary~\ref{skew-sym-coin-cor}).
It appears that if $\alpha$ is a skew-symmetric weak composition of 
the form $(0,0,\dotsc,0,a,b_1,b_2,\ldots)$
where $0<a = \min(\{a,b_1,b_2,\dots\}\setminus\{0\})$ and $a=\alpha_i$,
then any way of rearranging the first $i$ terms of $\alpha $ gives a skew-symmetric weak composition indexing 
the same $P$-key polynomial.  For example, 
\be
\label{sp-red-ex1}
\pkey_{0011042} = \pkey_{0101042} = \pkey_{1001042}.\ee
We may also have things like
\be
\label{sp-red-ex2}
  \pkey_{4331} \neq
  \pkey_{4313} =  \pkey_{4133} 
\neq
\pkey_{3413} =  \pkey_{1433}
\neq
\pkey_{3143} =  \pkey_{1343}
\neq
\pkey_{3134} =  \pkey_{1334}.
\ee
On the other hand, we have not found any symmetric weak compositions 
$\alpha \neq \gamma$ with  $\qkey_\alpha = \qkey_\gamma$.
We do not have a heuristic for the following conjecture but it is supported by 
computations.

\begin{conjecture}\label{sym-conj}
If $\alpha$ and $\gamma$ are distinct symmetric weak compositions then $\qkey_\alpha \neq \qkey_\gamma$.
\end{conjecture}

We have tested that there are no coincidences among the 20,288 distinct $Q$-key polynomials 
indexed by symmetric weak compositions $\alpha$ with $\ell(\alpha)\leq 8$ and $| \half(\lambda(\alpha))|\leq 8$.

Even if Conjecture~\ref{sym-conj} holds (that is, if the $Q$-key polynomials are uniquely indexed), they are still not linearly 
independent. For example, it holds that
$
\qkey_{123} + \qkey_{0321} = \qkey_{132} + \qkey_{0231}
$
although the four polynomials $\qkey_\alpha$ are all distinct.
The set of distinct $P$-key polynomials $\{\pkey_\alpha\}$ is also linearly dependent over $\ZZ$.
Finally, it is not hard to find some linear relationships between various $P$- and $Q$-key polynomials, such as the following.

\begin{corollary}
\label{cor:P_linear_comb_Q}
For all integers $n \geq 2$, we have
$
2\pkey_{(n+1)1^n} = \qkey_{n101^{n-2}} - \qkey_{1n01^{n-2}} + \qkey_{0n1^{n-1}}.
$
\end{corollary}

\begin{proof}
This follows from a straightforward computation using Proposition~\ref{prop:hook_shapes}, which identifies the key decompositions of $\pkey_{(n+1)1^n} = \pkey_{1,(n)}$ and $\qkey_{n1^{n-1}} = \qkey_{1,(n)}$.
\end{proof} 

There is also a simple condition that conjecturally determines when a $Q$-key polynomial is a scalar multiple of a $P$-key polynomial; see Theorem~\ref{alpha1-thm} and Conjecture~\ref{conj:alpha1_converse}.

Nevertheless, there are $\pkey_{\alpha}$'s that are not rational linear combinations of $\qkey_{\alpha}$'s and vice versa.

\begin{example}\label{ex:contain_the_span}
It is sufficient to examine $\pkey_{\alpha}$ (respectively, $\qkey_{\alpha}$) when $\alpha=\lambda(\alpha)$ is a skew-symmetric (respectively, symmetric) partition.
In degree $1$ the polynomials of this form  are
\[\pkey_{22} = x_1 + x_2 =\kappa_{01} \quand \qkey_{1} = 2x_1 = 2\kappa_1.\]
Thus every $P$-key polynomial with degree $1$ is given by $\kappa_{0^n1}$ for some $n>0$ and no $\QQ$-linear combination of these 
is $\qkey_{1}$.
To find a $P$-key polynomial that is not a $\QQ$-linear combination of $Q$-key polynomials, we have to go to degree $5$.
%
%
There, we have
\[
\pkey_{4422} = \key_{23} + \key_{0122} + \key_{0311} 
\]
while the $\qkey_{\alpha}$'s of degree $5$ with $\alpha$ a symmetric partition are
\begin{align*}
\qkey_{51111} & = 2(\key_{5} + \key_{11111} + \key_{20111} + \key_{30011} + \key_{40001}),
\\ \qkey_{4211} & = 4(\key_{32} + \key_{41} + \key_{1211} + \key_{2201} + \key_{3101}),
\\ \qkey_{332} & = 4(\key_{23} + \key_{122} + \key_{131}).
\end{align*}
There is no way to write $\pkey_{4422}$ as a $\QQ$-linear combination of polynomials $\pi_u\qkey_{51111}$, $\pi_v\qkey_{4211}$, and $\pi_w\qkey_{332}$ with $u,v,w\in S_\infty$.
Consider the key term $\key_{0311}$ in $\pkey_{4422}$, which can only be obtained from
\begin{align*}
\qkey_{1421} & = \pi_1 \pi_2 \qkey_{4211} = 4(\key_{032} + \key_{041} + \key_{1121} + \key_{0221} + \key_{0311}),\text{ or }
\\
\qkey_{1412} & = \pi_3 \pi_1 \pi_2 \qkey_{4211} = 4(\key_{0302} + \key_{0401} + \key_{1112} + \key_{0212} + \key_{0311}).
\end{align*}
However, these are the only $Q$-keys that contribute the key terms $\key_{041}$ and $\key_{0401}$, respectively, neither of which is in $\pkey_{4422}$.
Hence, $\pkey_{4422}$ is not a $\QQ$-linear combination of $Q$-key polynomials.
\end {example}

This example suggests the following open problem.

\begin{problem}\label{space-prob}
Characterize the distinct $\QQ$-vector spaces spanned by the $ \pkey_{\alpha}$'s and $ \qkey_{\alpha}$'s.
\end{problem}

By contrast, recall that the usual key polynomials are a basis for the entire polynomial ring.
We will show that the two vector spaces described in Problem~\ref{space-prob} are distinct; see Corollary~\ref{cor:span_noncontainment}.


 \subsection{Leading term conjectures}\label{leading-subsect}

We explain in this section some conjectural formulas for the leading terms of $\qkey_\alpha$ and $\pkey_\alpha$.
Throughout, let $\alpha$ be a symmetric weak composition, and write $\lambda =\lambda(\alpha)$ and $u=u(\alpha)$.
Define 
\[
\DSym_\alpha = \{ (u(i), u(j)) \mid (i,j) \in \D_\lambda\}.
\]
Notice that the weak composition $\alpha$ may be recovered from the symmetric diagram $\DSym_\alpha$ by counting the positions in each row (equivalently, each column).
Define 
the  \defn{sub-diagonal row and column counts} of $\alpha$ to be the sequences
$\rowCounts(\alpha) = (\rho_1,\rho_2,\ldots)$ and $\columnCounts(\alpha) = (\gamma_1,\gamma_2,\ldots)$, where 
\[
 \rho_j = | \{ (a,b) \in \DSym_\alpha \mid j = a\geq b\}|
\quand
\gamma_j = | \{ (a,b) \in \DSym_\alpha \mid a \geq b = j\}|
\]
for each $j \in \PP$. 
These sequences count the number of cells weakly below the diagonal in $\DSym_\alpha$ by rows and columns, respectively.
For example, if $\alpha =(2,3,0,1)$, so $\lambda=(3,2,1)$ and $u = 2143$, then
\[
\ytableausetup{aligntableaux=center}
 \DSym_\lambda = \D_\lambda = 
\begin{ytableau}
 \ & \ & \ & \none[\cdot] \\
 \  & \   & \none[\cdot] & \none[\cdot] \\ 
  \  & \none[\cdot] & \none[\cdot]   & \none[\cdot]  \\
\none[\cdot] & \none[\cdot] & \none[\cdot] & \none[\cdot]
 \end{ytableau}
 \qquand
  \DSym_\alpha = 
\begin{ytableau}
 \ & \ & \none[\cdot]& \none[\cdot] \\
  \ &\  & \none[\cdot]  & \  \\ 
      \none[\cdot] & \none[\cdot]& \none[\cdot]& \none[\cdot] \\
 \none[\cdot] & & \none[\cdot]& \none[\cdot] 
 \end{ytableau}\,,
\]
and thus we have $\rowCounts(\alpha) = (1,2,0,1)$ and $\columnCounts(\alpha) = (2,2,0,0)$.

\begin{proposition}\label{dsym-prop1}
 Suppose $\alpha$ is a symmetric weak composition. Then $\alpha$ is uniquely determined by  $\rowCounts(\alpha)$ and $\columnCounts(\alpha)$.
 Moreover, one has $\rowCounts(\alpha)  \leq_{\mathrm{lex}} \columnCounts(\alpha)$
 with equality if and only if $\abs{\alpha}\leq 1$.
\end{proposition}

\begin{proof}
We have $\DSym_\alpha = \varnothing$ if and only if $\rowCounts(\alpha) = \columnCounts(\alpha) = (0,0,0,\ldots)$.
Assume $\DSym_\alpha$ is nonempty so that $\rowCounts(\alpha)=(\rho_1,\rho_2,\ldots)$ and $\columnCounts(\alpha)=(\gamma_1,\gamma_2,\ldots)$ are both nonzero.
The sum $\rho_i + \gamma_i - 1$ is equal to $\alpha_i$ if $(i,i) \in \DSym_\alpha$ and is otherwise $\alpha_i - 1$.
Let $I$ be the set of indices $i$ at which $\rho_i + \gamma_i - 1$ is maximized.
Since we have $(i,i) \in \DSym_\alpha$ for all indices at which $\alpha_i = \max(\alpha)$,
it follows that $\alpha_i = \max(\alpha)$ if and only if $i \in I$. 
Moreover, as $\DSym_\alpha$
is obtained by permuting the coordinates of the diagram of a symmetric partition,
it follows that there is a  finite subset $J\subset \PP$ with
\[\DSym_\alpha  \cap (I \times \PP) = I\times J\quand \DSym_\alpha  \cap (\PP \times I) = J\times I \quand \DSym_\alpha \subseteq J\times J.\]
Specifically, $J $ is just the image of $\{1,2,\dotsc,\max(\alpha)\}$ under the permutation $u(\alpha)$.
Define
\[
\rho_i' := \begin{cases} 0 &\text{if }i \in I\text{ or }\rho_i=0 \\
 \rho_i - \abs{I_{<i}} &\text{if } i \notin I\text{ and }\rho_i\neq 0\end{cases}
\quand
 \gamma_i' := \begin{cases} 0 &\text{if }i \in I\text{ or }\gamma_i=0 \\
 \gamma_i - \abs{I_{>i}} &\text{if } i \notin I\text{ and }\gamma_i \neq 0,\end{cases}
 \]
where $I_{<i} := \{ j \in I \mid j < i \}$ and $I_{>i} := \{ j \in I \mid j > i \}$.
Finally, let $\alpha'=(\alpha'_1,\alpha'_2,\ldots)$, where 
\[ \alpha'_i := \begin{cases} 0&\text{if }i \in I \text{ or }\alpha_i = 0, \\ 
\alpha_i - |I|&\text{if }i \notin I \text{ and }\alpha_i \neq 0.\end{cases}
\] 
Then $\alpha'$ is a symmetric weak composition, whose shape is formed from $\lambda(\alpha)$ by removing all maximal hooks, and we have $\rowCounts(\alpha') = (\rho_1',\rho_2',\ldots)$ and $\columnCounts(\alpha') = (\gamma_1',\gamma_2',\ldots)$.
By induction we can recover $\DSym_{\alpha'}$ from these sequences.
This symmetric set contains no positions in the rows indexed by $I$ or in the columns indexed by $I$,
and we recover $\DSym_\alpha$ from $\DSym_{\alpha'}$ by adding the positions $(i,j)$ and $(j,i)$
for each $i \in I$ and $j \in \PP$ with $\rho_j' < \rho_j$ or $\gamma'_j < \gamma_j$.

Finally, if $\rho_1=0$ then we split our argument into two cases.
If $\gamma_1 > 0=\rho_1$, then the claim is immediate.
Therefore, we now assume $\gamma_1 = 0=\rho_1$, and so $\alpha_1 = 0$ since $\DSym_{\alpha}$ is symmetric.
In this case we may assume by induction on $\ell(\alpha)$ that $(\rho_2,\rho_3,\ldots) \leq_{\mathrm{lex}}  (\gamma_2,\gamma_3,\ldots)$, with equality if and only $\abs{\alpha}=1$, since these sequences are the sub-diagonal row and column counts for $(\alpha_2,\alpha_3,\ldots)$.
This implies the desired ordering of $\rowCounts(\alpha)$ and $\columnCounts(\alpha)$ when $\rho_1=0$.
Alternatively, we always have $\rho_1 \leq 1$ and if $\rho_1=1$ then $\gamma_1 \geq 1$ with equality if and only if $\DSym_\alpha = \{(1,1)\}$, so if $\rho_1=1$ then $\rowCounts(\alpha) <_{\mathrm{lex}} \columnCounts(\alpha)$ or $\rowCounts(\alpha) = \columnCounts(\alpha) = (1,0,0,\ldots)$.
\end{proof}

In general, a finite symmetric subset of $\PP\times\PP$ is not uniquely determined by its sub-diagonal row and column counts.
For example,
 $\rowCounts(\alpha) = (0,0,1,1)$ and $\columnCounts(\alpha) = (1, 1, 0, 0)$
 are the sub-diagonal row and column counts for both of the following configurations:
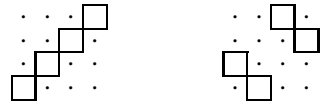
\[
 \begin{ytableau}
  \none[\cdot] & \none[\cdot] & \none[\cdot] & \\
   \none[\cdot] & \none[\cdot] & & \none[\cdot]  \\ 
    \none[\cdot] & & \none[\cdot]& \none[\cdot] \\
 \ & \none[\cdot] & \none[\cdot]& \none[\cdot] 
 \end{ytableau}
 \qquad\qquad
 \begin{ytableau}
  \none[\cdot] & \none[\cdot] & & \none[\cdot] \\
   \none[\cdot] & \none[\cdot] & \none[\cdot]  & \\ 
    & \none[\cdot] & \none[\cdot]& \none[\cdot] \\
 \none[\cdot] & & \none[\cdot]& \none[\cdot] 
 \end{ytableau}
\]

Define $\diag(\alpha) = |\{ (i,j) \in \DSym_\alpha \mid i=j\}|$ for a symmetric weak composition $\alpha$.
The sequences $\rowCounts(\alpha)$ and $\columnCounts(\alpha)$ appear to be related to $\qkey_\alpha$ in the following way:

\begin{conjecture}\label{leading-conj1}
If $\alpha$ is a symmetric weak composition with $\rowCounts(\alpha)  \neq    \columnCounts(\alpha)$ and $n = \ell(\alpha)$ then
\[  2^{-\diag(\alpha)}\qkey_\alpha \in    x^{\rowCounts(\alpha)} +  x^{\columnCounts(\alpha)} + \sum_{
\substack{\delta \in \NN^n \\ 
\rowCounts(\alpha) <_{\mathrm{lex}} \delta }} \NN x^\delta.\]
\end{conjecture}

If $\rowCounts(\alpha)  =    \columnCounts(\alpha)$ then either $\alpha =\emptyset$ and $\qkey_\alpha = 1$,
or $\alpha  = \e_n$ and 
 $2^{-\diag(\alpha)}\qkey_\alpha  = x_1 + x_2 + \dots +x_n$.

\begin{remark}
Combined with~\eqref{leading-eq} and Theorem~\ref{thm:key_osp_key_positive}, Conjecture~\ref{leading-conj1} implies that 
$2^{-\diag(\alpha)}\qkey_\alpha \in \kappa_{\rowCounts(\alpha)} +  \sum_{\rowCounts(\alpha) <_{\mathrm{lex}} \delta \in \NN^n } \NN \kappa_\delta$.
There is no need for $\kappa_{\columnCounts(\alpha)}$ to appear in this expansion.
We have $2^{-\diag(\alpha)}\qkey_\alpha = \kappa_{\rowCounts(\alpha)}$ when $\alpha = (2,2)$ for example.
The coefficient of $x^{\columnCounts(\alpha)}$ in $\qkey_\alpha$ can also be greater than $2^{\diag(\alpha)}$: 
If $\alpha = (1,2)$ then $\rowCounts(\alpha) =(0,2)$ and $\columnCounts(\alpha) = (1,1)$ but 
$2^{-\diag(\alpha)}\qkey_{12} = x_2^2 + 2x_1 x_2 + x_1^2$.
\end{remark}

Continue to let $\alpha$ be a symmetric weak composition, 
and define the \defn{strict sub-diagonal row and column counts} of $\alpha$ to be the sequences
$\strictRowCounts(\alpha) = (\rho_1,\rho_2,\ldots)$ and $\strictColumnCounts(\alpha) = (\gamma_1,\gamma_2,\ldots)$ 
where 
\[
 \rho_j = | \{ (a,b) \in \DSym_\alpha \mid j = a> b\}|
\quand
\gamma_j = | \{ (a,b) \in \DSym_\alpha \mid a > b = j\}|
\]
for each $j \in \PP$. For example,
 if $\alpha =(1,3,0,1)$, so $\lambda(\alpha)=(3,1,1)$ and $u(\alpha) = 2143$, then
\[
\ytableausetup{aligntableaux=center}
 \DSym_\lambda = \D_\lambda = 
\begin{ytableau}
 \ & \ & \ & \none[\cdot] \\
 \  & \none[\cdot]   & \none[\cdot] & \none[\cdot] \\ 
  \  & \none[\cdot] & \none[\cdot]   & \none[\cdot]  \\
\none[\cdot] & \none[\cdot] & \none[\cdot] & \none[\cdot]
 \end{ytableau}
 \qquand
  \DSym_\alpha = 
\begin{ytableau}
 \none[\cdot] & \ & \none[\cdot]& \none[\cdot] \\
  \ &\  & \none[\cdot]  & \  \\ 
      \none[\cdot] & \none[\cdot]& \none[\cdot]& \none[\cdot] \\
 \none[\cdot] & & \none[\cdot]& \none[\cdot] 
 \end{ytableau}\,,
\]
and thus we have $\strictRowCounts(\alpha) = (0,1,0,1)$ and $\strictColumnCounts(\alpha) = (1,1,0,0)$.


\begin{proposition}
 Suppose $\alpha$ is a skew-symmetric weak composition. Then $\alpha$ is uniquely determined by  $\strictRowCounts(\alpha)$ and $\strictColumnCounts(\alpha)$.
 Moreover, one has $\strictRowCounts(\alpha)  \leq_{\mathrm{lex}} \strictColumnCounts(\alpha)$
 with equality if and only if $\abs{\alpha}= 0$.
\end{proposition}

\begin{proof}
The proof of the first claim is the same as for Proposition~\ref{dsym-prop1},
except that if $\strictRowCounts(\alpha)=(\rho_1,\rho_2,\ldots)$ and $\strictColumnCounts(\alpha)=(\gamma_1,\gamma_2,\ldots)$ are both nonzero
then  the quantity $\rho_i + \gamma_i + 1$   is equal to $\alpha_i$ if $(i,i) \in \DSym_\alpha$ and to $\alpha_i + 1$ otherwise.
If we redefine $I$ to be the set of indices $i$ at which $\rho_i + \gamma_i + 1$ is maximized
then we can proceed exactly as in our earlier argument. This works because removing all maximal hooks from a skew-symmetric weak composition
results in a composition that is still skew-symmetric.

The second claim that $\strictRowCounts(\alpha)  \leq_{\mathrm{lex}} \strictColumnCounts(\alpha)$,
 with equality if and only if $\abs{\alpha}= 0$, follows from Proposition~\ref{dsym-prop1}
 since $\rowCounts(\alpha) -  \strictRowCounts(\alpha) = \columnCounts(\alpha)-\strictColumnCounts(\alpha) = \sum_{(i,i) \in \DSym_\alpha} \e_i$
 and there are no skew-symmetric weak compositions $\alpha$ with $\abs{\alpha}=1$.
\end{proof}

The strict sub-diagonal row and column counts do not uniquely identify a finite symmetric subset of $\PP\times \PP$.
For example, these sequences are the same for the Young diagrams of $(3,2,1)$ and $(3,1,1)$ shown
in~\eqref{ss-ex1} and~\eqref{ss-ex2}.

Computations support a symplectic analogue of Conjecture~\ref{leading-conj1}:

\begin{conjecture}\label{leading-conj2}
For a skew-symmetric weak composition $\alpha$ with $\strictRowCounts(\alpha)  \neq    \strictColumnCounts(\alpha)$ and $n = \ell(\alpha)$:
\[
\pkey_\alpha \in   x^{\strictRowCounts(\alpha)} +    x^{\strictColumnCounts(\alpha)} + \sum_{\substack{\delta \in \NN^n \\ \strictRowCounts(\alpha) <_{\mathrm{lex}} \delta}} \NN x^\delta.
\]
\end{conjecture}

We only have $\strictRowCounts(\alpha)  =    \strictColumnCounts(\alpha)$ when $\alpha=\strictRowCounts(\alpha)  =   \emptyset$ in which case $\pkey_\alpha = 1$.

\begin{remark}
Combined with~\eqref{leading-eq} and Theorem~\ref{thm:key_osp_key_positive}, Conjecture~\ref{leading-conj2} would imply that 
$\pkey_\alpha \in \kappa_{\strictRowCounts(\alpha)} +  \sum_{\strictRowCounts(\alpha) <_{\mathrm{lex}} \delta \in \NN^n } \NN \kappa_\delta$.
There is no need for $\kappa_{\strictColumnCounts(\alpha)}$ to appear in this expansion.
We have $\pkey_\alpha = \kappa_{\strictRowCounts(\alpha)}$ when $\alpha = (2,2)$ for example.
The coefficient of $x^{\strictColumnCounts(\alpha)}$ in $\pkey_\alpha$ can again be greater than one:  
If $\alpha = (1,3,1)$ then $\strictRowCounts(\alpha) =(0,1,1)$ and $\strictColumnCounts(\alpha) = (1,1,0)$ but 
$\pkey_{131} = x_2 x_3 + 2x_1 x_2+ x_2^2 + x_1 x_3  + x_1^2$.
\end{remark}

The triangularity properties in Conjectures~\ref{leading-conj1} and \ref{leading-conj2}
suggest a more efficient algorithm to search for a $\NN$-linear expansion of a given polynomial
into $P$- or $Q$-key polynomials. This informs the conjectures in the following section.

\subsection{Schubert polynomials}\label{ischub-sect}

In this section we review the definition of \defn{Schubert polynomials} and their relationship
to key polynomials. Then we will explain several conjectures that connect $\pkey_\alpha$ and $\qkey_\alpha$ in a similar way to analogues of Schubert polynomials 
 introduced in~\cite{WyserYong}.
  
The \defn{(Rothe) diagram} of  $w \in S_\infty$ is the set $D(w) := \{ (i, w(j)) \mid i,j \in \PP,\ i<j,\ w(i)>w(j)\}.$
This set is the complement in $\PP\times \PP$ of the south-and-east hooks through each nonzero position $(i,w(i))$ in 
the permutation matrix of $w$. For example, we have
\be
\label{eq:rothe_example}
D(53124) = \left\{ \square \in \left[\begin{smallmatrix} 
\square & \square & \square & \square & 1 \\
\square & \square & 1 & \cdot & \cdot \\
1 & \cdot & \cdot & \cdot & \cdot \\
\cdot & 1 & \cdot & \cdot & \cdot \\
\cdot & \cdot & \cdot & 1 & \cdot  \end{smallmatrix} \right]\right\} = \{ (1,1),(1,2),(1,3),(1,4),(2,1),(2,2)\}.
\ee
This description makes it easy to see that the \defn{(right) descent set} $\DesR(w) := \{ i  : w(i) >w(i+1)\}$ is contained in $[n]$ if and only if $D(w)\subseteq [n]\times \PP$.

A permutation $w$ is \defn{dominant} of shape $\lambda=\lambda(w)$ if its Rothe diagram $D(w)$
is equal to the Young diagram $\D_{\lambda}$. 
The example shown in~\eqref{eq:rothe_example} is dominant of shape $\lambda = (4,2)$.
There is a unique dominant permutation of each partition shape, and the family of all dominant elements of $S_\infty$ is exactly the set of $132$-avoiding permutations~\cite[Ex.~2.2.2]{Manivel}.

The \defn{Schubert polynomials} $\fkS_w$ are the unique elements of $\ZZ[x_1,x_2,\dots]$ indexed by $w \in S_\infty$ such that 
$\fkS_w = x^{\lambda(w)}=\kappa_{\lambda(w)}$ if $w$ is dominant and $\partial_i \fkS_w = \fkS_{ws_i}$ if $w(i) > w(i+1)$~\cite[\S2.3.1 and \S2.6.4]{Manivel}.
%
Each $\fkS_w$ is homogeneous of degree $|D(w)| = \ell(w)$ 
and $\{ \fkS_w : D(w) \subseteq [n]\times \PP\}$ is a $\ZZ$-basis for $\ZZ[x_1,x_2,\dotsc,x_n]$~\cite[Prop.~2.5.4]{Manivel}.

By results in~\cite{LS2}, 
each Schubert polynomial is a sum of (not necessarily distinct) key polynomials.
An explicit algorithm is known to identify the relevant summands; see~\cite[Thm.~4]{ReinerShimozono}.

 \begin{example}
 If $w=214365 \in S_6$ then $\fkS_w = \kappa_3 + \kappa_{102} + \kappa_{10101} + \kappa_{20001}$.
 \end{example}

 We now explain the definition of certain shifted analogues of Schubert polynomials introduced in~\cite{WyserYong}.
Let $\Ifpf_\infty$ be the 
$S_\infty$-conjugacy class
 of the infinite product of cycles \be1_\fpf :=  (1 \; 2)(3 \; 4)(5 \; 6)\cdots  \ee
mapping $i \mapsto i - (-1)^i$ for all $i \in \PP$.
Let $\Ifpf_n$ be the subset of elements $z \in \Ifpf_\infty$ with $z([n]) = [n]$ and $z(i) = 1_\fpf(i)$ for all $i > n$.
This set is empty if $n$ is odd, and is in bijection with the fixed-point-free involutions in $S_n$ when $n$ is even.

The Rothe diagram for $z \in \Ifpf_\infty$  is 
defined in the same way as for elements of $S_\infty$.
If $\lambda$ is a skew-symmetric partition,
then there is a unique $z \in \Ifpf_\infty$
with $\{ (i,j) \in D(z) \mid i \neq j\} = \{ (i,j) \in \D_\lambda \mid i\neq j\}$~\cite[Prop.~4.31]{HMP6},
which we call
 the \defn{dominant} element of $ \Ifpf_\infty$ with shape $\lambdafpf(z) := \lambda$.

\begin{example}
Each fixed-point-free involution $z\in S_{2n}$ extends to an element of $\Ifpf_{2n}$
mapping 
$2i \mapsto 2i-1$ for all $i\notin[2n]$.
We often identify $z$ with this extension.
Under this convention,  $z=(2n)\cdots 321$ is the dominant element of $\Ifpf_\infty$ with shape $\lambdafpf(z)$ is $(2n-1,\dotsc,3,2,1) - \e_{n}$.
\end{example}

Let $I_\infty :=\{ z \in S_\infty \mid z=z^{-1}\}$ be the set of involutions in $S_\infty$.
We have already mentioned that for each partition $\lambda$ 
there is a unique dominant element $w \in S_\infty$ with $D(w) = \D_\lambda$.
Notice that $D(w)^\T = D(w^{-1})$,
so if $w$ is dominant with shape $\lambda$ then $w^{-1}$ is dominant with shape $\lambda^\top$.
This means that if $w$ is dominant and its shape is a symmetric partition $\lambda$, then we must have $w=w^{-1}$.
Thus, for each symmetric partition $\lambda=\lambda^\T$ there is a unique
dominant $z \in I_\infty$ with $D(z) = \D_\lambda$.

\begin{example}\label{321-ex}
The permutation $n\cdots 321 \in I_\infty$ is dominant with $\lambda(n\cdots 321)= (n-1,\dotsc,2,1)$.
\end{example}

Results in~\cite{HMP5,HMP1,WyserYong} show that there are unique polynomials $\{\fkSS_y\}_{y \in \Ifpf_\infty}$  and $\{\fkSO_z\}_{z \in I_\infty}$
with
\be\label{schub-dom-eq}
\fkSS_y = \prod_{\substack{(i,j) \in D(y) \\ i>j}} (x_i+x_j) = \pkey_{\lambdafpf(y)}
\quand
\fkSO_z = \prod_{\substack{(i,j) \in D(z) \\ i\geq j}} (x_i+x_j)=\qkey_{\lambda(z)}
\ee
if $y \in \Ifpf_\infty$ and $z\in I_\infty$ are dominant (see~\cite[Thm.~1.3]{HMP1} and~\cite[Thm.~4.2]{HMP5}) and with
\be\label{ischub-eq}
  \partial_i \fkSS_{y} = \begin{cases} 0&\text{if }y(i) < y(i+1) ,
\\
0&\text{if } y(i)=i+1, \\ 
  \fkSS_{s_i y s_i} &\text{otherwise,}
  \end{cases}
\quand
  \partial_i \fkSO_{z} = \begin{cases} 0&\text{if }z(i) < z(i+1), \\ 
 2 \fkSO_{zs_i} &\text{if }z(i) =i+1,\\
  \fkSO_{s_i z s_i} &\text{otherwise,}
    \end{cases}
    \ee
    for all $i \in \PP$.
Following~\cite{HMP1}, we refer to
$\{\fkSS_y\}_{y \in \Ifpf_\infty}$  
and 
 $\{\fkSO_z\}_{z \in I_\infty}$
 as \defn{involution Schubert polynomials} of symplectic and orthogonal types, respectively.

Let  $\DSp(y) := \{(i,j) \in D(y) \mid i>j\}$ and $\DO(z) :=  \{(i,j) \in D(z) \mid i\geq j\}$ for $y \in \Ifpf_\infty$ and $z \in I_\infty$.
Then $\fkSS_y$ and $\fkSO_z$ are each homogeneous of degrees $|\DSp(y)|$ and $|\DO(z)|$, respectively~\cite[Prop.~3.6]{HMP1}.
Involution Schubert polynomials are nonzero linear combinations of disjoint sets of   Schubert polynomials~\cite[\S1.5]{HMP1},
which implies that $\{\fkSS_y \mid y \in \Ifpf_\infty\}$   and  $\{\fkSO_z \mid z \in I_\infty\}$ are both linearly independent over $\ZZ$.
The following property is implicit in results in~\cite{HMP6,HMP1}: 

\begin{proposition}\label{lexmin-prop}
Let $y \in \Ifpf_\infty$ and $z \in I_\infty$.
Then $\fkSS_y$ (respectively~$\fkSO_z$) is contained in $ \ZZ[x_1,x_2,\dotsc,x_n]$ 
if and only if the diagram $\DSp(y)$
(respectively~$\DO(z)$) is contained in $ [n]\times [n] $.
\end{proposition}

\begin{proof}
The lexicographically minimal monomials in  $\fkSS_y$ and $\fkSO_z$
are in $ \ZZ[x_1,\dotsc,x_n]$  if and only if $\DSp(y)$
and $\DO(z)$ are subsets of $ [n]\times [n] $~\cite[Prop.~3.14]{HMP1}.
When this occurs,~\cite[Thm~1.1]{HMP6} combined with~\cite[Thms.~5.8 and 5.14]{HMP6} implies that all other monomials are also in  $ \ZZ[x_1,\dotsc,x_n]$.
\end{proof}

\begin{remark}\label{visible-remark}
As $\DSp(y)$ and $\DO(z)$ are weakly below the diagonal, the sets are contained in $ [n]\times[n] $ if and only if they are contained in $[n]\times \PP$.
Whether   this occurs can be characterized in terms of certain  \defn{visible descent sets}.
For $y \in \Ifpf_\infty$ and $z \in I_\infty$ let
\be\label{visible-eq}
\Des^\fpf_V(y) :=\{ i  \in \PP : y(i) > y(i+1) < i\}
\quand
\Des_V(z) :=\{ i \in \PP : z(i) >z(i+1) \leq i\}.
\ee
Then  $ \DSp(y) \subseteq [n]\times [n]$ if and only if $\Des^\fpf_V(y)\subseteq [n] $
by~\cite[Lem.~4.14]{HMP4} 
and
$ \DO(z) \subseteq [n]\times [n]$  if and only if $\Des_V(z)\subseteq [n] $
by~\cite[Lem.~4.8]{HMP5}.
\end{remark}

Computations indicate that $\fkSS_y$ and $\fkSO_z$ expand positively into shifted
key polynomials, just as Schubert polynomials expand into key polynomials. 
The following conjectures are one of our primary motivations for studying $\pkey_\alpha$ and $\qkey_\alpha$:

\begin{conjecture}\label{fkSS-conj}
Let $y \in \Ifpf_\infty$.
Then
  $
  \fkSS_y
  $ 
is a sum of distinct $P$-key polynomials.
More specifically,  there is a finite set of skew-symmetric weak compositions $\cCSp(y)$ with $\fkSS_y = \sum_{\alpha \in \cCSp(y)} \pkey_\alpha$.
\end{conjecture}

We have used a computer to verify Conjecture~\ref{fkSS-conj} for all $y \in \Ifpf_n$ for $n\leq 10$.
  
\begin{example}\label{fkSS-conj-ex}
For most $y \in \Ifpf_8$ it holds that
 $\fkSS_y$ is equal to a $P$-key polynomial.
For the 13 of 105 elements $y \in \Ifpf_8$ without this property, we have the following expansions:
{\small\[
\ba
\fkSS_{(1 \; 3)(2 \; 5)(4 \; 7)(6 \; 8)} &= \pkey_{3303} + \pkey_{140101}, \\
\fkSS_{(1 \; 4)(2 \; 3)(5 \; 7)(6 \; 8)} &= \pkey_{333} + \pkey_{411001}, \\
\fkSS_{(1 \; 3)(2 \; 4)(5 \; 8)(6 \; 7)} &= \pkey_{33003} + \pkey_{1400011}, \\
\fkSS_{(1 \; 3)(2 \; 5)(4 \; 8)(6 \; 7)} &= \pkey_{3304001} + \pkey_{1501011}, \\
\fkSS_{(1 \; 5)(2 \; 3)(4 \; 7)(6 \; 8)} &= \pkey_{4133} + \pkey_{511101}, \\
\fkSS_{(1 \; 4)(2 \; 3)(5 \; 8)(6 \; 7)} &= \pkey_{3340001} + \pkey_{41303} + \pkey_{5110011},  \\
\fkSS_{(1 \; 5)(2 \; 3)(4 \; 8)(6 \; 7)} &= \pkey_{4224} + \pkey_{5133001} + \pkey_{6111011}.  
\ea
\qquad
\ba
\fkSS_{(1 \; 5)(2 \; 4)(3 \; 7)(6 \; 8)} &= \pkey_{4242} + \pkey_{533101}, \\
\fkSS_{(1 \; 3)(2 \; 6)(4 \; 8)(5 \; 7)} &= \pkey_{24042} + \pkey_{1503301}, \\
\fkSS_{(1 \; 5)(2 \; 4)(3 \; 8)(6 \; 7)} &= \pkey_{4252001} + \pkey_{6331011}, \\
\fkSS_{(1 \; 6)(2 \; 3)(4 \; 8)(5 \; 7)} &= \pkey_{51242} + \pkey_{6113301}, \\
\fkSS_{(1 \; 6)(2 \; 4)(3 \; 8)(5 \; 7)} &= \pkey_{52522} + \pkey_{6241201}, \\
\fkSS_{(1 \; 6)(2 \; 5)(3 \; 8)(4 \; 7)} &= \pkey_{54542} + \pkey_{6444101}, \\ \
\ea
\]}%
\end{example}

Let $\cyc(z) := |\{ i \in \PP : i  < z(i)\}|$ denote the number of nontrivial cycles in $z \in I_\infty$.
    We have verified the following conjecture
    for all $z \in I_n$ for $n\leq 9$.
    
\begin{conjecture}\label{fkSO-conj}
Let  $z \in I_\infty$.
Then
  $
  \fkSO_z
  $ 
is a linear combination of $Q$-key polynomials with coefficients that are integral powers of two.
More specifically, 
 there is a finite set of symmetric weak compositions $\cCO(z)$ with
  $\fkSO_z = \sum_{\alpha \in \cCO(z)} 2^{\cyc(z) - \diag(\alpha)} \qkey_\alpha$
and $\diag(\alpha) \leq \cyc(z)$ for all $\alpha \in \cCO(z)$.
  \end{conjecture}

  \begin{example}\label{fkSO-conj-ex}
For most $z \in  I_5$ it holds that
 $\fkSO_z$ is equal to a $Q$-key polynomial.
For the 5 of 26 elements $z \in I_5$ without this property, we have the following expansions:
\[
\ba
\fkSO_{(1 \; 2)(3 \; 4)} &= 2 \qkey_{201}, \\
\fkSO_{(1 \; 2)(4 \; 5)} &= 2 \qkey_{2001},  \ea \qquad \ba 
\fkSO_{(2 \; 3)(4 \; 5)} &= 2 \qkey_{0201}, \\ \ \ea \qquad\ba
\fkSO_{(1 \; 2)(3 \; 5)} &= \qkey_{202} + 2 \qkey_{3011}, \\
\fkSO_{(1 \; 3)(4 \; 5)} &= \qkey_{22} + 2 \qkey_{3101}. 
\ea\]
The decomposition of $\fkSO_z$ into $Q$-key polynomials is not always unique. For example,
\[
\ba
\fkSO_{(1 \; 4)(2 \; 5)(6 \; 8)} 
&= \qkey_{343001} + \qkey_{3340001} + 2\qkey_{442002} + 2 \qkey_{3520011} \\
&= \qkey_{334001} + \qkey_{3430001} + 2\qkey_{442002} + 2 \qkey_{3520011},
\ea
\]
even though 
$\qkey_{343001}$, $\qkey_{3340001}$, $\qkey_{334001}$, and $\qkey_{3430001}$ are all distinct.
\end{example}

\subsection{Coincidences among shifted polynomials}\label{coincid-subsect}

Modulo the preceding conjectures, we can characterize which
involution Schubert polynomials
are equal to shifted key polynomials. 
First, however, we need to explain how $\pkey_\alpha$ and $\qkey_\alpha$
are related to the classical Schur $P$- and $Q$-functions.

Let $w_n := n\cdots 321\in S_n$ denote the reverse permutation.
For each $v \in S_\infty$, $y \in \Ifpf_\infty$, and $z \in I_\infty$, there are unique symmetric functions $F_v$, $P_y$, and $Q_z$
with
\be\label{FPQ-eq}
F_v(x_1,\dotsc,x_n) = \pi_{w_n} \fkS_v,
\quad
P_y(x_1,\dotsc,x_n) = \pi_{w_n} \fkSS_y,
\quand
Q_z(x_1,\dotsc,x_n) = \pi_{w_n} \fkSO_z
\ee
whenever the respective diagrams $D(v)$, $\DSp(y)$, and $\DO(z)$ are contained in $[n]\times \PP$
\cite[Thm.~3.39]{HMP1}.
The formal power series $F_v$ is the \defn{Stanley symmetric function} introduced in~\cite{Stanley},
while $P_y$ and $Q_z$ are the \defn{involution Stanley symmetric functions} studied in~\cite{HMP5} and~\cite[\S4.5]{HMP4}.

 Recall that the ring of bounded degree symmetric functions $\Sym\subset \ZZ\llbracket x_1,x_2,\dots\rrbracket $ has a basis given by the \defn{Schur functions} $s_\lambda$.
For each strict partition $\mu$, there is an associated \defn{Schur $P$-function} $P_\mu$  and \defn{Schur $Q$-function} $Q_\mu = 2^{\ell(\mu)} P_\mu$, which make up $\ZZ$-bases for two subrings of  $\Sym$~\cite[\S{III}.8]{Macdonald}. 
In fact, the power series $F_v$, $P_y$, and $Q_z$ are always $\NN$-linear combinations of Schur functions $s_\lambda$, Schur $P$-functions $P_\mu$,
and Schur $Q$-functions $Q_\nu$, respectively; see~\cite{EG,HMP5,HMP4}.

A permutation $v$ is \defn{vexillary} if $F_v = s_\lambda$ is equal to a single Schur function.
Such permutations are characterized by a simple pattern avoidance condition: a permutation is vexillary if and only if it is $2143$-avoiding~\cite{Stanley}.
It also holds that $Q_z = Q_\nu$ is  a single Schur $Q$-function if and only if $z \in I_\infty$ is vexillary~\cite[Thm.~4.67]{HMP4}.
However, for Schur $P$-functions, the situation is slightly different.
Define $y \in \Ifpf_\infty$ to be \defn{fpf-vexillary} if $P_y=P_\mu$ is equal a single Schur $P$-function.
There is a more complicated (but still explicit) pattern avoidance characterization of fpf-vexillary involutions~\cite[Cor.~7.9]{HMP5}.
In all of these cases, it is known how to compute $\lambda$, $\mu$, and $\nu$ from $v$, $y$, and $z$;
see~\cite[Thm.~4.1]{Stanley},~\cite[Thm.~1.13]{HMP4}, and~\cite[Thm.~1.4]{HMP5}.

The first part of the following theorem is well-known property of key polynomials that follows from their combinatorial description~\cite{LS1} or representation-theoretic interpretation~\cite{Demazure74}.


 \begin{theorem}\label{stab-thm}
 Let $\lambda\in \NN^n$ be a partition with at most $n$ parts.
 \ben
 \item[(a)] It holds that  $s_{\lambda}(x_1,x_2,\dotsc,x_n) = \kappa_{(\lambda_n,\dotsc,\lambda_2,\lambda_1)}$.
\item[(b)]  If $\lambda $ is symmetric with $\mu := \shalf(\lambda)$ and $\nu := \half(\lambda)$ then
 \[ P_{\mu}(x_1,x_2,\dotsc,x_n) = \pkey_{(\lambda_n,\dotsc,\lambda_2,\lambda_1)}
\quand
 Q_{\nu}(x_1,x_2,\dotsc,x_n) = \qkey_{(\lambda_n,\dotsc,\lambda_2,\lambda_1)}.\]
 \een
 \end{theorem}
 
 \begin{proof}
 For part (a), see, e.g.,~\cite[\S2~Remark]{ReinerShimozono}.  
 The first identity in part (b) will follow by setting $\beta=0$ in Corollary~\ref{gp-dom-cor}.
 To derive the second identity in part (b),
 let $z \in I_\infty$ be dominant of shape $\lambda$. 
 Then $z$ is also vexillary by \cite[Prop.~2.2.7]{Manivel}.
 We also have $\DO(z) \subseteq D(z) =\D_\lambda \subseteq [n]\times \PP$,
and so
$
\qkey_{(\lambda_n,\dotsc,\lambda_2,\lambda_1)} = \pi_{w_n}\fkSO_{ z} = Q_z(x_1,\dotsc,x_n) = Q_{\nu}(x_1, \dotsc, x_n)
$
for a strict partition $\nu$, where the last equality is because $z$ is vexillary.
The partition $\nu$ must be $\half(\lambda)$ as this is the index of the leading term of the Schur $Q$-expansion of $Q_z$ by~\cite[Thm.~1.13]{HMP4}.
\end{proof}

\begin{corollary}\label{skew-sym-coin-cor}
If $\alpha$, $\gamma$ are skew-symmetric weak compositions with $\pkey_\alpha = \pkey_\gamma$ then $\lambda(\alpha) =\lambda(\gamma)$.
\end{corollary}

\begin{proof} 
Let $\mu := \shalf(\lambda(\alpha))$ and $\nu := \shalf(\lambda(\gamma))$.
Applying $\pi_{w_n}$ for sufficiently large $n$ to both sides of $\pkey_\alpha = \pkey_\gamma$ gives $P_\mu(x_1,x_2,\dotsc,x_n) =P_\nu(x_1,x_2,\dotsc,x_n)\neq 0 $ by Theorem~\ref{stab-thm}.
This implies that $\mu=\nu$, and hence that $\lambda(\alpha) =\lambda(\gamma)$ since the latter partitions are skew-symmetric.
\end{proof}

The \defn{(Lehmer) code} of a permutation $w$ is the integer sequence $c(w) = (c_1,c_2,c_3,\ldots)$ with 
\[c_i = \abs{\{ j \mid (i,j) \in D(w)\}}.\]
Another classical result is that $\fkS_w = \kappa_\alpha$ if and only if $w \in S_\infty$ is vexillary and $\alpha = c(w)$; see Theorem~\ref{vex-thm} for a slight generalization.
We can prove (partial) shifted analogues of this fact.

\begin{proposition}\label{ss-vex-prop}
Let $y \in \Ifpf_\infty$. If $\fkSS_y = \pkey_\alpha$ for a skew-symmetric weak composition $\alpha$ then 
$y$ is fpf-vexillary.
Conversely, if Conjecture~\ref{fkSS-conj} holds and $y$ is fpf-vexillary then  $\fkSS_y = \pkey_\alpha$ for some  $\alpha$.
\end{proposition}

\begin{proof}
If $\fkSS_y = \pkey_\alpha$ where $\lambda =\lambda(\alpha)$ and $\mu = \shalf(\lambda)$,
 then for all sufficiently large $n$ we have 
\[
 P_y(x_1,\dotsc,x_n) = \pi_{w_n} \fkSS_z = \pi_{w_n} \pkey_\alpha =\pkey_{(\lambda_n,\dotsc,\lambda_2,\lambda_1)}= P_\mu(x_1,\dotsc,x_n)
\]
 by Theorem~\ref{stab-thm}, which implies that $P_y = P_\mu$ so $y$ is fpf-vexillary.

Conversely, if $\fkSS_y$ is a sum of $k$ $P$-key polynomials,
then $P_y(x_1,\dotsc,x_n) = \pi_{w_n} \fkSS_y$ is a sum of $k$ (not necessarily distinct) Schur $P$-polynomials for all sufficiently large $n$,
so $P_y$ is a sum of $k$  (not necessarily distinct) Schur $P$-functions.
 In this case, it can only hold that $y$ is fpf-vexillary if $k=1$, which means that $\fkSS_y = \pkey_\alpha$ for some  $\alpha$.
\end{proof}

\begin{proposition}\label{so-vex-prop}
Let $z \in I_\infty$. If $\fkSO_z = \qkey_\alpha$ for some symmetric weak composition $\alpha$ then 
$z$ is vexillary.
Conversely, if Conjecture~\ref{fkSO-conj} holds and $z$ is vexillary then  $\fkSO_z = \qkey_\alpha$ for some  $\alpha$.
\end{proposition}

\begin{proof}
The proof is essentially the same as for the previous proposition, \emph{mutatis mutandis};
one just needs to change each ``$\Sp$'' to ``$\O$'', each ``$P$'' to ``$Q$'', and ``$\mu = \shalf(\lambda)$'' to ``$\mu = \half(\lambda)$''.
\end{proof}

As mentioned above, if $w \in S_\infty$ is vexillary then $\fkS_w = \kappa_{c(w)}$. 
An analogous formula appears to hold in the $Q$-shifted case; we have checked this conjecture
  for all $z \in I_n$ for $n\leq 9$:

\begin{conjecture}\label{vex-conj}
If $z \in I_\infty$ is vexillary then $\fkSO_z = \qkey_{c(z)}$.
\end{conjecture}

If $w \in S_\infty$ is vexillary then $F_w = s_{\lambda(c(w))^\top}$ by \cite[Thm.~4.1]{Stanley} and $F_{w^{-1}} = s_{\lambda(c(w))}$
by \cite[Lem.~5.4]{Marberg2019a} (with $\beta=0$).
Hence, if $z =z^{-1}\in \I_\infty$ is vexillary then we must have $\lambda(c(z))^\top= \lambda(c(z))$ and so $c(z)$ is a symmetric weak composition.
 
Conjecture~\ref{fkSO-conj} predicts that there is a multiplicity-free expansion of the 
renormalized involution Schubert polynomials $\iS_z := 2^{-\cyc(z)} \fkSO_z$ considered in~\cite{HMP6,HMP1,HMP4}
into renormalized $Q$-key polynomials $\hat \kappa_\alpha := 2^{-\diag(\alpha)} \qkey_\alpha$.
One reason to keep the powers of two is because we can have $\iS_z = \hat \kappa_\alpha$ without $z$ being vexillary;
for example, $\iS_{(1,2)(3,4)} = \hat\kappa_{201}$.

We have not identified a way to predict $\alpha$ such that $\fkSS_y = \pkey_\alpha$ from 
a fpf-vexillary $y \in \Ifpf_\infty$.

\begin{example}\label{ss-vex-prop-ex}
In many cases with $\fkSS_y = \pkey_\alpha$,
it holds that $c_i - \alpha_i \in \{0,1\}$ for all $i \in \PP$, where $c(y) =(c_1,c_2,\ldots)$;
e.g., if $y=(1 \; 3)(2 \; 7)(4 \; 5)(6 \; 8)$
then $\fkSS_y = \pkey_{150111}$ and $c(y) = (2, 5, 0, 2, 1, 2)$.
This does not always occur, however.
For instance, if $y=(1 \; 3)(2 \; 4)(5 \; 7)(6 \; 8) \in \Ifpf_8$
then we have $\fkSS_y = \pkey_{130001}$ and $  c(y) = (2,2,0,0,2,2),$
and if $y=(1 \; 4)(2 \; 5)(3 \; 6)(7 \; 10)(8 \; 11)(9 \; 12)$ then we have
$\fkSS_y = \pkey_{234000012}$ and $ c(y) = (3,3,3,0,0,0,3,3,3).$
\end{example}

As one other application of the results above, we can classify when $\kappa_\alpha$ is equal to a $P$-key polynomial.
We briefly note one lemma; here, let $\lambda$ and $\mu$ be partitions with $\mu$ strict. 

\begin{lemma}\label{Ps-lem}
It holds that $s_\lambda=P_\mu$ if and only if $\mu = \lambda = (n,\dotsc,2,1)$ for some $n \in \NN$.
\end{lemma}

\begin{proof}
This follows by combining~\cite[Thm.~3.35]{HMP1} and~\cite[Thm.~4.20]{HMP4}; see also~\cite[Thm.~V.3]{DeWitt}.
\end{proof}

Clearly $\kappa_\emptyset = \pkey_\emptyset=1$.
The other situations where $\kappa_\alpha = \pkey_\gamma$ are all as follows:

 \begin{proposition}\label{ss-key-prop}
 Let $\alpha$ and $\gamma$ be weak compositions with $\alpha$ nonempty and $\gamma$ skew-symmetric.
Then $\kappa_\alpha = \pkey_\gamma$ if and only if there are positive integers $1\leq i_0 < i_1 < i_ 2 < \cdots <i_n$ 
such that
$
\alpha = \e_{i_1} + 2 \e_{i_2}+3 \e_{i_3}+ \cdots + n \e_{i_{n}}
$
and
$
\gamma = (n+1)(\e_{i_0} + \e_{i_1} + \e_{i_2}+ \dots + \e_{i_n})
$.
\end{proposition}

\begin{proof}
It follows from Theorem~\ref{stab-thm} and Lemma~\ref{Ps-lem} that for any $n>0$ we have
\[
\kappa_{(0,1,2,\dotsc,n)} = s_{(n,\dotsc,3,2,1)}(x_1,x_2,\dotsc,x_{n+1}) = P_{(n,\dotsc,3,2,1)}(x_1,x_2,\dotsc,x_{n+1}) = \pkey_{(n+1,n+1,\dotsc,n+1)}.
\]
By applying isobaric divided difference operators to this identity,
we deduce that 
$\kappa_\alpha = \pkey_\gamma$ 
if   
$
\alpha =  \e_{i_1} + 2 \e_{i_2}+3 \e_{i_3}+ \dots + n \e_{i_{n}}$
and $\gamma = (n+1)(\e_{1} + \e_{i_1} + \e_{i_2}+ \dots + \e_{i_n})$
for some indices $1< i_1 < i_ 2 < \cdots <i_n$.
Since in this case $\kappa_\alpha$ is symmetric in $x_1,x_2,\dotsc,x_{i_1-1}$, the polynomial $\pkey_\gamma$  is fixed by $\pi_j$ for all $1\leq j <i_1-1$,
so one also has $\kappa_\alpha=\pkey_\gamma = \pkey_{\gamma + \e_{i_0} - \e_1}$ for any $1\leq i_0 <i_1$.

Conversely, suppose $\kappa_\alpha = \pkey_\gamma$ when $\alpha$ and $\gamma$ are arbitrary weak compositions with $\alpha$ nonempty and $\gamma$ skew-symmetric. In view of Theorem~\ref{stab-thm}, applying $\pi_{w_N}$ to both sides when $N$ is sufficiently large 
implies that $s_{\lambda(\alpha)} = P_{\shalf(\lambda(\gamma))}$,
so $\lambda(\alpha) = \shalf(\lambda(\gamma)) = (n,\dotsc,2,1)$ for some  $n \in \PP$  by Lemma~\ref{Ps-lem}.
However $ \shalf(\lambda(\gamma))$ can only be equal to $(n,\dotsc,2,1)$ if $\lambda(\gamma) = (n+1,n+1,\dots,n+1) \in \NN^{n+1}$,
so we must have $
\gamma = (n+1)(\e_{i_0} + \e_{i_1} + \e_{i_2}+ \dots + \e_{i_n})
$ for some $1\leq i_0 < i_1 < i_ 2 < \cdots <i_n$. By the preceding paragraph this means that $\kappa_\alpha = \pkey_\gamma =\kappa_{ \e_{i_1} + 2 \e_{i_2}+3 \e_{i_3}+ \dots + n \e_{i_{n}}}$, so we must also have $\alpha = \e_{i_1} + 2 \e_{i_2}+3 \e_{i_3}+ \dots + n \e_{i_{n}}$ as key polynomials are uniquely indexed by weak compositions.
\end{proof}

We can only have $\qkey_\alpha = \kappa_\gamma$ if $\alpha=\gamma = \emptyset$ 
since  
 all coefficients in $\qkey_\alpha$ are divisible by $2^{\diag(\alpha)}$. 
 If we  exclude this factor then there is a nontrivial
 analogue of the preceding result.

\begin{proposition}
\label{prop:key_pkey_equality}
 Let $\alpha$ and $\gamma$ be weak compositions with $\alpha$ nonempty and $\gamma$ symmetric.
Then $\kappa_\alpha = 2^{-\diag(\gamma)}\qkey_\gamma$ if and only if 
 there are positive integers $1\leq  i_1 < i_ 2 < \cdots <i_n$ 
such that
$
\alpha=  \e_{i_1} +2\e_{i_2} + 3 \e_{i_3}+ \cdots +n \e_{i_n}
$
and
$
\gamma= n(\e_{i_1} + \e_{i_2} + \cdots + \e_{i_n})
$.
\end{proposition}

\begin{proof}
The ``if'' direction of this result follows by applying isobaric divided difference operators to the identity
\[
\kappa_{(1,2,\dotsc,n)} = s_{(n,\dotsc,3,2,1)}(x_1,x_2,\dotsc,x_{n}) = 2^{-n}Q_{(n,\dotsc,3,2,1)}(x_1,x_2,\dotsc,x_{n}) = 2^{-n}\qkey_{(n,n,\dotsc,n)},
\]
which holds by Theorem~\ref{stab-thm} and Lemma~\ref{Ps-lem} (as $Q_\mu = 2^{\ell(\mu)}P_\mu$).
For the ``only if'' direction, observe
that 
if 
$\kappa_\alpha = 2^{-\diag(\gamma)}\qkey_\gamma$ then applying $\pi_{w_N}$ to both sides when $N$ is sufficiently large 
gives $s_{\lambda(\alpha)} = P_{\half(\lambda(\gamma))}$ by Theorem~\ref{stab-thm} since $\diag(\gamma) =\diag(\lambda(\gamma))= \ell(\half(\lambda(\gamma)))$.
In this case, Lemma~\ref{Ps-lem} implies that $\lambda(\alpha) = \half(\lambda(\gamma)) = (n,\dotsc,2,1)$ for some  $n \in \PP$,
which can only hold if $\lambda(\gamma) = (n,n,\dots,n) \in \NN^n$ and 
$\gamma= n(\e_{i_1} + \e_{i_2} + \cdots + \e_{i_n})$ for some indices $1\leq  i_1 < i_ 2 < \cdots <i_n$,
and then we have $\kappa_\alpha = 2^{-\diag(\gamma)}\qkey_\gamma = \kappa_{\e_{i_1} +2\e_{i_2} + 3 \e_{i_3}+ \cdots +n \e_{i_n}}$
so also $\alpha = \e_{i_1} +2\e_{i_2} + 3 \e_{i_3}+ \cdots +n \e_{i_n}$.
\end{proof}

\def\etaone{\eta^{(1)}}
\def\etatwo{\eta^{(2)}}

Let $\alpha$ be a symmetric weak composition with $\alpha_1=0$.
Then there is a symmetric weak composition $\etaone$ with
\[\DSym_{\etaone} = \DSym_\alpha \sqcup \{ (i,j)  : i=1\text{ and }(j,j) \in \DSym_\alpha,\text{ or }j=1\text{ and } (i,i) \in \DSym_\alpha\}\]
and  there is a symmetric weak composition $\etatwo$ with 
\[\DSym_{\etatwo}=\DSym_{\etaone} \sqcup \{(1,1)\}.\]
If $\lambda=\lambda(\alpha)$ and $h = \diag(\alpha)$
then 
\[\lambda(\etaone) = (1+\lambda_1,1+\lambda_2, \dots,1 + \lambda_h , h, \lambda_{h+1},\lambda_{h+2},\ldots)
\quand
\lambda(\etatwo) = \lambda(\etaone) + \e_{h+1}.\]
Exactly one of $\etaone$ or $\etatwo$ is skew-symmetric; 
specifically, we always have $\lambda_h \geq h > \lambda_{h+1}$, and $\etaone$ is skew-symmetric if and only if $\lambda_h > h$.
Let $ \eta(\alpha)$ be whichever of $\etaone$ or $\etatwo$ is skew-symmetric.

\begin{example}
If $\alpha =(0,3,1,2)$ then $\etaone = (2,4,1,3)$ and $\etatwo=(3,4,1,3)$,
and we have
 \[ \DSym_\alpha = 
 \ytabsmall{
 \none[\cdot] &  \none[\cdot] &  \none[\cdot] &  \none[\cdot] \\
 \none[\cdot] &  \ & \ &  \ \\
 \none[\cdot] &  \ &  \none[\cdot] &  \none[\cdot] \\
 \none[\cdot] & \ &  \none[\cdot] &  \ \\
 }
 \qquand
 \DSym_{\etaone} = 
 \ytabsmall{
 \none[\cdot] &  \ &  \none[\cdot] &  \ \\
\ &  \ & \ &  \ \\
 \none[\cdot] &  \ &  \none[\cdot] &  \none[\cdot] \\
\ & \ &  \none[\cdot] &  \ \\
 }\qquand
 \DSym_{\etatwo} = 
 \ytabsmall{
\ & \ &  \none[\cdot] &  \\\
\ &  \ & \ &  \ \\
 \none[\cdot] &  \ &  \none[\cdot] &  \none[\cdot] \\
\ & \ &  \none[\cdot] &  \ \\
 }
 \]
 (so $\DSym_{\etatwo} = \DSym_{\etaone} \sqcup \{(1,1)\}$)
along with $\eta(\alpha) = (3,4,1,3) = \etatwo$.
\end{example}

\begin{theorem}\label{alpha1-thm}
Suppose $\alpha$ is a symmetric weak composition with $\alpha_1=0$.
Then $\qkey_\alpha = 2^{\diag(\alpha)} \pkey_{\eta(\alpha)}$.
\end{theorem}

\begin{proof}
First let $\widetilde\alpha = (\alpha_2,\alpha_3,\ldots)$
and observe that $\pi_{u(\alpha)} = \pi_{1\times u(\widetilde\alpha)}  \pi_1\pi_2\cdots \pi_n$.
From~\cite[(4.22)]{MacdonaldSchubertNotes}, it follows that if $\fkS_w \in \ZZ[x_1,x_2,\dotsc,x_n]$ then $\pi_1\pi_2\cdots \pi_n(\fkS_w) = \fkS_{1\times w}$,
where $1 \times w \in S_\infty$ denotes the element fixing $1$ and mapping $i+1 \mapsto w(i)+1$ for all $i \in \PP$.
Each $\fkSO_y \in \ZZ[x_1,x_2,\dotsc,x_n]$ is a linear combination of Schubert polynomials $\fkS_w\in \ZZ[x_1,x_2,\dotsc,x_n]$,
and it follows from~\cite[\S1.5]{HMP1} that $\pi_1\pi_2\cdots \pi_n(\fkSO_y) = \fkSO_{1\times y}$.
Therefore, if $y \in I_\infty$ is dominant with shape $\lambda(\alpha)$ then
\[
\qkey_\alpha 
=
\pi_{1\times u(\widetilde\alpha)}  \pi_1\pi_2\cdots \pi_n (\fkSO_y)
=
\pi_{1\times u(\widetilde\alpha)}  (\fkSO_{1\times y}).
 \]
Next, define $\etaone$ and $\etatwo$ relative to $\alpha$ as above, and let $\eta = \eta(\alpha)$ and
$h = \diag(\alpha)$. Then
\[\pkey_{\eta} =
   \pi_{1\times u(\widetilde \alpha)} \pi_1\pi_2\cdots \pi_h( \pkey_{\lambda(\eta)} )
   = \pi_{1\times u(\widetilde \alpha)} \partial_1\partial_2\cdots \partial_h( x_1x_2\cdots x_h\pkey_{\lambda(\eta)} ).
   \]
It is clear from Definition~\ref{spo-key-def} (with $w=1$) that
   $2^h x_1x_2\cdots x_h  \pkey_{\lambda(\eta)} = \qkey_{\lambda(\etaone)}$,
   so if $z \in I_\infty$ is dominant with shape $\lambda(\etaone)$ then
\[
 2^h \pkey_{\eta} =\pi_{1\times u(\widetilde \alpha)}\partial_1 \partial_2 \cdots \partial_h (\fkSO_z).
\]
To prove the desired identity it suffices to check that $\partial_1 \partial_2 \cdots \partial_h (\fkSO_z) = \fkSO_y$.
This follows from~\eqref{ischub-eq} since we can express $y = (1 \; b_1)(2 \; b_2)\cdots (h \; b_h)$ and $z =  (1 \; 1+b_1)(2 \; 1+b_2)\cdots (h \; 1+b_h)$ for  distinct integers $b_1,b_2,\dotsc,b_h >h$.
\end{proof}

The simple relationship between $\qkey_\alpha$ and $\pkey_\alpha$ when $\alpha_1=0$ does not
extend to their $K$-theoretic analogues  $\qlascoux_\alpha$ and $\plascoux_\alpha$
defined in Section~\ref{k-sect}.
This is identical to the situation for Schur $P$- and $Q$-polynomials (see~\cite{ChiuMarberg}).
The following converse to Theorem~\ref{alpha1-thm} is supported by computations.

\begin{conjecture}
\label{conj:alpha1_converse}
If $\qkey_\alpha$ is a scalar multiple of a $P$-key polynomial then $\alpha_1=0$.
\end{conjecture}

In all computed examples, $\qkey_\alpha$ is not even a $\ZZ$-linear combination of 
$P$-key polynomials if $\alpha_1>0$.

\subsection{Key expansions}
\label{sec:key_expansion}

Theorem~\ref{thm:key_osp_key_positive} tells us that $\pkey_\alpha$ and $\qkey_\alpha$
both expand positively into key polynomials, but identifying the relevant key terms and their coefficients
seems to be a difficult problem in general. 
In this section we discuss some partial results and conjectures related to this problem.

When $\alpha = \lambda(\alpha)$ is a (skew-)symmetric partition,
  $\pkey_\alpha$ and $\qkey_\alpha$ are instances of (dominant) involution Schubert polynomials $\fkSS_y$ and $\fkSO_z$.
 There are formulas to expand these polynomials into Schubert polynomials $\fkS_w$ \cite[\S6]{HMP2}
and to expand each $\fkS_w$ into key polynomials \cite[Thm.~4]{ReinerShimozono}. However,
it is not clear how to combine these results to address the following simpler problem:

\begin{problem}
\label{probl:key_expansion}
Describe the key expansions of $\pkey_\lambda$ and $\qkey_\lambda$ 
when $\lambda$ is a (skew-)symmetric partition.
\end{problem}

As previously mentioned, solving Problem~\ref{probl:key_expansion} might help to determine the $\QQ$-span of the $P$- and $Q$-key polynomials, addressing Problem~\ref{space-prob}.

\begin{remark}\label{keyexp-rmk}
The key expansions of $\pkey_\lambda$ and $2^{-\diag(\lambda)} \qkey_\lambda$ (with $\lambda$ a symmetric partition) are 
multiplicity-free up to degree $11$, but this property does not hold in general (as mentioned earlier in Remark~\ref{tabular-rmk}).
If $\mu$ is a symmetric partition with $\shalf(\lambda) = \half(\mu) =: \nu$ 
then Theorem~\ref{stab-thm} implies that there is a bijection between the terms in the key expansions of 
$\pkey_\lambda=\pkey_{1,\nu}$ and $2^{-\diag(\mu)} \qkey_\mu= 2^{-\ell(\nu)} \qkey_{1,\nu}$
whenever these are multiplicity-free. In the small number of non-multiplicity-free examples that we can compute,
there is still a multiplicity-preserving bijection between the terms in these key expansions.
It would be interesting to know if this property holds for all strict partitions $\nu$.
A multiplicity-preserving bijection between the key terms in $\pkey_{w,\nu}$ and $2^{-\ell(\nu)} \qkey_{w,\nu}$ does not exist for general $w \in S_\infty$.
If $w = 35421$ and $\nu = (4,2)$, then we have
\begin{align*}
\pkey_{w,\nu} & = \key_{00024} + \key_{00033} + \key_{00114} + \key_{00222} + \key_{01113} + \key_{01122} + 2\key_{00123} ,
\\
2^{-\ell(\nu)} \qkey_{w,\nu} & = \key_{00024} + \key_{00033}  + \key_{10014} + \key_{20022} + \key_{10113} + \key_{10122}  + \key_{00123}+ \key_{10023}.
\end{align*}
\end{remark}

There is one key term in $\pkey_\alpha$ and $\qkey_\alpha$
that we can easily identify.
 We will say that a polynomial is \defn{key positive} if its expansion into key polynomials involves only nonnegative integer coefficients.
 If $w \in S_\infty$ and $\alpha$ is a weak composition, then define $w\circ \alpha$ to be the weak composition with 
 $\pi_w \kappa_\alpha = \kappa_{w\circ \alpha}$.
 Recall the definitions of $\rowCounts(\alpha)$ and $\strictRowCounts(\alpha)$ from Section~\ref{leading-subsect}.
 
 \begin{proposition}\label{one-term-prop}
Let $\alpha$ be a symmetric weak composition with $\lambda :=\lambda(\alpha)$ and define $\delta :=  u(\alpha) \circ \rowCounts(\lambda)$ and $\epsilon:=  u(\alpha)\circ \strictRowCounts(\lambda)$.
Then $\qkey_\alpha -  2^{\diag(\alpha)}\kappa_{\delta}$ and $\pkey_\alpha -  \kappa_{\epsilon}$ are both key positive.
\end{proposition}

\begin{proof}
We see from Definition~\ref{spo-key-def} that
the lexicographically minimal terms in $\qkey_\lambda$ and $\pkey_\lambda$
are respectively $2^{\diag(\alpha)} x^{\rowCounts(\lambda)} $ and $x^{\strictRowCounts(\lambda)}$,
so the differences 
$\qkey_\lambda -  2^{\diag(\alpha)}\kappa_{\rowCounts(\lambda)}$ and $\pkey_\lambda -  \kappa_{\strictRowCounts(\lambda)}$
must both be key positive by~\eqref{leading-eq}.
By definition, applying $\pi_{u(\alpha)}$ to these expressions gives 
$\qkey_\alpha -  2^{\diag(\alpha)}\kappa_{\delta}$ and $\pkey_\alpha -  \kappa_{\epsilon}$,
which are then also key positive (as key positivity is preserved by any $\pi_i$ operator).
\end{proof}

When $\lambda$ is a symmetric partition, Proposition~\ref{one-term-prop} identifies the lexicographically minimal key terms in $\pkey_{\lambda}$ and $\qkey_{\lambda}$.
Indeed, these are indexed by $\widetilde{\rho}(\lambda)$ and $\rho(\lambda)$, respectively, which count the heights of the columns in the shifted Young diagrams of $\shalf(\lambda)$ and $\half(\lambda)$ (moved one column to the right in the $P$-case). 
These terms do not always remain lexicographically minimal when we apply isobaric divided difference operators to 
go from $\lambda$ to a general symmetric weak composition; for example, 
$\qkey_{21} =  2(\key_{11} + \key_{2})$
but $\qkey_{12} = \pi_1 \qkey_{21} = 2(\key_{02} + \key_{11})$. 
Thus, Proposition~\ref{one-term-prop} does not help with proving Conjectures~\ref{leading-conj1} or \ref{leading-conj2}.

\begin{remark}
We can adapt the preceding proof to get a general algorithm for computing the key expansion of $\pkey_{\alpha}$ or $\qkey_{\alpha}$.
Take the key expansion of $\pkey_{\lambda(\alpha)} = \sum_{\mu} c^{\psymbol}_{\mu} \key_{\mu}$ or $\qkey_{\lambda(\alpha)} = \sum_{\mu} c^{\qsymbol}_{\mu} \key_{\mu}$, which there exists an algorithm to compute as indicated at the start of this section.
Then $\pkey_{\alpha} = \sum_{\mu} c^{\psymbol}_{\mu} \key_{u(\alpha) \circ \mu}$ and $\qkey_{\alpha} = \sum_{\mu} c^{\qsymbol}_{\mu} \key_{u(\alpha) \circ \mu}$. A given key polynomial may appear multiple times in each sum,
but one never needs to expand the polynomials into the monomial basis of $\ZZ[x_1, x_2, \ldots]$.
\end{remark}

Using Proposition~\ref{one-term-prop}, we can prove a partial version of Conjecture~\ref{conj:alpha1_converse}.

\begin{corollary}
Suppose $\alpha$ is a nonempty symmetric weak composition with $\lambda=\lambda(\alpha)$.
If there is an index $i$ with $1 \leq i \leq \diag(\alpha)$
such that $(\alpha_1,\alpha_2,\dotsc,\alpha_i)$ is a permutation of $(\lambda_1,\lambda_2,\dotsc,\lambda_i)$,
then $\qkey_\alpha$ is not a scalar multiple of any $P$-key polynomial.
\end{corollary}

\begin{proof}
In this case $u(\alpha)$ has a reduced expression $i_1i_2\cdots i_{\ell}$ with all $i_j \neq i$, so $\qkey_\alpha= \pi_{i_1}\pi_{i_2}\cdots \pi_{i_{\ell}} \qkey_{\lambda(\alpha)}$ is divisible by $x_1x_2\cdots x_i$ since $\qkey_{\lambda(\alpha)}$
is divisible by this factor and $\pi_j(x_1x_2\cdots x_i f) = x_1x_2\cdots x_i \pi_j(f)$ for any polynomial $f$ when $ i\neq j \in \PP$.
However, no $P$-key polynomial is divisible by $x_1x_2\cdots x_i$ as
the leading term of $\kappa_{\epsilon}$ in Proposition~\ref{one-term-prop} is never divisible by $x_1$.
\end{proof}

Calculations suggest that the key polynomials appearing in $\pkey_\alpha$ all have the following property:

\begin{conjecture}\label{pkey-key-conj}
Suppose $\alpha$ and $\gamma$ are weak compositions  with $\alpha$ skew-symmetric.
If $\pkey_\alpha - \kappa_\gamma$ is key positive and $j$ is the first index with $\gamma_j  =0$ then $\gamma_i > 1$ for all $0<i<j$.
\end{conjecture}

This conjecture would imply another partial version of Conjecture~\ref{conj:alpha1_converse}:

\begin{corollary}
If Conjecture~\ref{pkey-key-conj} holds and $\alpha$ is a nonempty symmetric weak composition with $\alpha_1 = \max(\alpha)$,
then $\qkey_\alpha$ is not a $\QQ$-linear combination of $P$-key polynomials.
\end{corollary}

\begin{proof}
Define $\delta$ as in Proposition~\ref{one-term-prop}.
Since $\max(\alpha) = \alpha_1$, the action of $\mu(\alpha)$ does not permute the first entry.
Hence $\delta_1 = \rho(\alpha)_1 = 1$ and $\kappa_\delta$ appears in the key expansion of $\qkey_\alpha$, but if Conjecture~\ref{pkey-key-conj} holds then $\kappa_\delta$ does not appear in any $P$-key polynomial.
\end{proof}

The next result is consistent with Conjecture~\ref{pkey-key-conj}:

\begin{proposition}\label{1n-prop}
 Suppose $w \in S_\infty$ and $n \in \PP$. Then $\fkS_w - \kappa_{1^n}$ is key positive if and only if $w= 234\cdots (n+1)1$ and $\fkS_w = \kappa_{1^n}$.
If $\alpha$ is a skew-symmetric weak composition then $\pkey_\alpha- \kappa_{1^n} $ is never key positive.
\end{proposition}

\begin{proof}
One can derive the first claim from \cite[Thms.~3 and 5]{ReinerShimozono}. \cite[Thm.~3]{ReinerShimozono}
states that $\fkS_w = \sum_{a} \sum_{i} x_i$ where the outer sum ranges over all reduced words $a=a_1a_2\cdots a_p$
for $w \in S_\infty$ and the inner sum ranges over all \defn{$a$-compatible words} $i=i_1i_2\cdots i_p$ with $x_i := x_{i_1}x_{i_2}\cdots x_{i_p}$.
Part (1) of \cite[Thm.~5]{ReinerShimozono} asserts that the same double sum $\sum_{a} \sum_{i} x_i$ gives a key polynomial 
when $a$ is restricted to a single equivalence class for the \defn{Coxeter--Knuth relation} defined in \cite[Eq.~(14)]{ReinerShimozono}.
Therefore each Coxeter--Knuth equivalence class of reduced words for $w$ contributes one term to the key positive expansion of $\fkS_w$.
However, it is clear from the definition of an $a$-compatible word before \cite[Thm.~3]{ReinerShimozono} that the only class that can
contribute the key polynomial $\kappa_{1^n} = x_1x_2\cdots x_n$ is the one that consists of the single reduced word $a=123\cdots n$,
which is present if and only if $w= 234\cdots (n+1)1$.

For the second claim, suppose $\alpha$ is a skew-symmetric weak composition with $\lambda=\lambda(\alpha)$.
Let $z \in \Ifpf_\infty$ be dominant with $\lambdafpf(z) = \lambda$. Since $\pkey_\alpha = \pi_{u(\alpha)} \fkSS_z$
and $\fkSS_z$ is key positive, 
and since the coefficient of $\kappa_{1^n}$ in the key expansion of $\pkey_\alpha$ is the same as in $ \fkSS_z$
as $1^n$ is invariant under permuting entries,
the difference $\pkey_\alpha- \kappa_{1^n} $ can only be key positive if $ \fkSS_z- \kappa_{1^n} $
is key positive. The latter never occurs since $\fkSS_z$ is a sum of distinct Schubert polynomials~\cite[\S1.5]{HMP1} and 
every $\fkS_w$ that appears as a summand has 
the property that $2i-1$ appears before $2i$ in the word $w(1)w(2)w(3)\cdots$ for all $i \in \PP$ by \cite[Thm.~6.22]{HMP2},
so in particular has $w\neq 234\cdots (n+1)1$ for all $n \in \PP$.
\end{proof}

As a corollary, we can derive a property relevant to Problem~\ref{space-prob}.

\begin{corollary}
\label{cor:span_noncontainment}
One has $\QQ\spanning\{\qkey_{\alpha} \mid \deg(\qkey_{\alpha}) = n\} \not\subseteq \QQ\spanning\{\pkey_{\alpha} \mid \deg(\pkey_{\alpha}) = n\}$
for all $ n \in \PP$.
\end{corollary}

\begin{proof}
The first space contains $\qkey_{1,(n)}$, whose positive key expansion involves $\kappa_{1^n}$ by Proposition~\ref{prop:hook_shapes}.
The second space does not contain $\qkey_{1,(n)}$ since $\kappa_{1^n}$ 
is not in any $\pkey_{\alpha}$ by Proposition~\ref{1n-prop}.
\end{proof}

\section{Shifted Lascoux polynomials and atoms}\label{k-sect}

Key polynomials and Schubert polynomials have $K$-theoretic analogues respectively called \defn{Lascoux polynomials}
and \defn{Grothendieck polynomials}. In this section
we review the definitions of these functions
and then study their shifted analogues. 
This leads us to define two new families of shifted Lascoux polynomials generalizing $\pkey_\alpha$ and $\qkey_\alpha$.

\subsection{More divided difference operators}\label{more-sect}

Our notation for the divided difference operators in this section follows \cite[\S2.2 and \S4.2]{MP2019b}.
This reference also discusses the main properties of interest (which we only briefly summarize here) in a little more detail.

Fix a 
formal parameter $\beta$ that commutes with $x_i$ for all $i \in \PP$
and write $\cL := \ZZ[\beta][x_1^{\pm1}, x_2^{\pm1},\dots]$ for the ring of Laurent polynomials 
in $x_1,x_2,\dots$ with coefficients in $\ZZ[\beta]$. 
The group $S_\infty$ acts on $\cL$ by permuting variables.
For each $i \in \PP$ we 
define operators
$\bpartial_i$ and $\bpi_i$  on $\cL$ by 
\begin{subequations}
\begin{align}
 \bpartial_i f &:= \partial_i\bigl( (1+ \beta x_{i+1}) f \bigr) = -\beta f + (1+ \beta x_i) \partial_i f, \\
 \bpi_i f &:=\pi_i\bigl( (1+ \beta x_{i+1}) f \bigr)  = f+ x_{i+1}(1+\beta x_i) \partial_i f= \bpartial_i( x_i f), 
\end{align}
\end{subequations}
and also set $\bopi_i := \bpi_i  - 1$. These operators preserve 
 $\ZZ[\beta][x_1,x_2,\dots]\subset \cL$ and
satisfy
\be\label{pipi-eq}
 \bpartial_i \bpartial_i =-\beta \bpartial_i \quand \bpi_i \bpi_i = \bpi_i\quand \bopi_i\bopi_i= -\bopi_i,
 \ee
as well as the Coxeter braid relations.
Thus, for $w \in S_\infty$ we can define
$ \bpartial_w := \bpartial_{i_1}\bpartial_{i_2}\cdots \bpartial_{i_l}$ where 
$i_1i_2\cdots i_l$ is any reduced word for $w$. We define $\bpi_w$ and $\bopi_w$ similarly for $w \in S_\infty$.
It is useful to note that if $i \in \PP$ and $f,g \in \cL$ then we have 
\be
s_i f = f\quad\Longleftrightarrow\quad  \partial_i f = 0\quad\Longleftrightarrow\quad \bpartial_i f =-\beta f \quad\Longleftrightarrow\quad \bpi_i f = f\quad\Longleftrightarrow\quad \bopi_i f = 0
\ee
and if these equivalent properties hold then 
\be
\bpartial_i(fg) =f \cdot \bpartial_ig
\quand
 \bpi_i(fg) =f \cdot \bpi_i g 
 \quand
  \bopi_i(fg) =f \cdot \bopi_i g .
\ee

Following the convention in \cite{IkedaNaruse}, for any elements $x,y \in \cL$ we let 
\be\label{o-plus-minus-eq}
 x\oplus y  := x+ y + \beta xy
 \quand 
 x\ominus y := \frac{x-y}{1+\beta y}.
 \ee
For $n \in \PP$, let $\delta_n := (n-1,\dotsc,2,1,0) \in \NN^n$ and recall that  $w_n := n\cdots 321\in S_n$.
We have
\be
\bpartial_{w_n} f = \sum_{w \in S_n} w \( \dfrac{f }{\prod_{1 \leq i < j \leq n} x_i \ominus x_j}\)
\quand
\bpi_{w_n} f = \bpartial_{w_n} ( x^{\delta_n} f)
\ee
for all $f \in \cL$ by \cite[Lems. 4.7 and 4.8]{MP2019b}.
In the summation in the first expression, the action of $S_\infty$ on Laurent polynomials is implicitly extended to the field of
rational functions $\QQ(\beta)(x_1,x_2,\dots)$.
We also note that if $f \in \ZZ[\beta][x_1,x_2,\dots]$ 
and $g = \bpi_{w_n} f$ for some $n \in \PP$, then
\be\label{mn-pi-eq}
g(x_1,x_2,\dotsc,x_m) = \bpi_{w_m} \bigl( f(x_1,x_2,\dotsc,x_m) \bigr)
\quad\text{for all $1\leq m \leq n$.}
 \ee

Finally, for $w \in S_\infty$ and $i \in \PP$ define $w\circ s_i$ to be $w$ if $w(i)>w(i+1)$ or $ws_i$ if $w(i)<w(i+1)$.
 This operation extends to an associative product $S_\infty \times S_\infty \to S_\infty$~\cite[Thm.~7.1]{Humphreys},
 which is often called the \defn{Demazure product}.
 Notice that $s_i\circ s_i = s_i$ and so $w\circ s_i \circ s_i = w\circ (s_i\circ s_i)= w\circ s_i$.
 
\begin{lemma}\label{var-bar-combine-lem}
Suppose $u,v \in S_\infty$.
Then 
$\bopi_u \bpi_v \in  \NN\spanning\left\{\bopi_w \mid w \in S_\infty\text{ with }w \leq u\circ v\right\}$.
\end{lemma}

\begin{proof}
We may assume $v \neq 1$.
Choose $i \in \PP$ with $v(i)> v(i+1)$, and let $s=s_i$.
By induction on $\ell(v)$ we may assume that $\bopi_u \bpi_{vs}$ is a $\NN$-linear combination of $\bopi_w$'s for $w \in S_\infty$ with $w \leq u \circ vs$.
If $w(i)>w(i+1)$ then $w$ has a reduced word ending in $i$ 
so $\bopi_w\bopi_s = -\bopi_w$,
and if $w(i) < w(i+1)$ then $\bopi_w \bopi_s = \bopi_{ws}$.
 Thus for any $w \in S_\infty$ the product 
$ \bopi_w \bpi_s = \bopi_w(1+ \bopi_s) $ is  
$\bopi_w + \bopi_{ws}$ when $w(i) < w(i+1)$ or else zero.
The lifting property of the Bruhat order~\cite[Prop.~2.2.7]{CCG} implies that 
if $w \leq u\circ vs$ and $w(i) < w(i+1)$ then $w<ws  \leq u\circ vs \circ s = u\circ v$,
so we conclude that $\bopi_u \bpi_v = \bopi_u \bpi_{vs} \bpi_s$ is a $\NN$-linear 
combination of $\bopi_w$'s with $w \leq u\circ v$.
\end{proof}
 
\subsection{\texorpdfstring{$K$}{K}-theoretic polynomials}\label{lasc-groth-sect}

Let $\alpha$ be a weak composition and define  $u:= u(\alpha)$ and $\lambda:= \lambda(\alpha)$ as in Section~\ref{coh-subsect1}.
Following \cite{Lascoux01}, 
we define the \defn{Lascoux polynomial} and \defn{Lascoux atom polynomial} of $\alpha$ by 
the respective formulas
 \be L_\alpha :=  \bpi_{u} x^{\lambda} \quand \oL_\alpha :=  \bopi_{u} x^{\lambda}.
 \ee
 It follows as in Remark~\ref{u-rem} that 
 \be
 \bpi_i L_\alpha =\begin{cases} L_{s_i \alpha} & \text{if }\alpha_i > \alpha_{i+1},
 \\
 L_\alpha &\text{if }\alpha_i < \alpha_{i+1},
  \\
 L_\alpha &\text{if }\alpha_i = \alpha_{i+1},
 \end{cases}
 \quand
  \bopi_i \oL_\alpha =\begin{cases} \oL_{s_i \alpha} & \text{if }\alpha_i > \alpha_{i+1},
 \\
 -\oL_\alpha &\text{if }\alpha_i < \alpha_{i+1},
  \\
0 &\text{if }\alpha_i = \alpha_{i+1}.
 \end{cases}
 \ee
Both  $L_\alpha$ and $\oL_\alpha $ belong to $\NN[\beta][x_1,x_2,\dots]$ and are homogeneous of degree $\abs{\alpha}$ if $\deg(\beta) := -1$.
Setting $\beta=0$ recovers the definitions of  $ \kappa_\alpha=L_\alpha\bigr\rvert_{\beta=0} $ and $\overline\kappa_\alpha =\oL_\alpha\bigr\rvert_{\beta=0} $.
Taking $u=1$ in Lemma~\ref{var-bar-combine-lem} shows that $L_\alpha = \sum_{\gamma \preceq \alpha} \oL_\gamma$,
where $\preceq$ is the \defn{composition Bruhat order}
with $\gamma \preceq \alpha$ if and only if $\lambda(\gamma)=\lambda(\alpha)$
and $u(\gamma) \leq u(\alpha)$~\cite[Thm 5.1]{Mon16}.
As shown in~\cite{BSW,PY22,ShimozonoYu,Yu21}, these polynomials have combinatorial formulas that generalize what holds for key polynomials.

\begin{example}
Lascoux polynomials are $\ZZ[\beta]$-linear but not necessarily $\NN[\beta]$-linear combinations of key polynomials~\cite[\S3.2]{RY21}.
If $\alpha = (1,0,2,1)$ then $u(\alpha) = s_2s_1s_3$  and 
\[ L_{1021} = 
 \kappa_{1021} -\beta \kappa_{212}  +\beta \kappa_{122}  + \beta \kappa_{2021}+ 2\beta\kappa_{1121} + \beta^2\kappa_{222}+ 
\beta^2\kappa_{1221} + \beta^2\kappa_{2121} + \beta^3 \kappa_{2221}.
\]
The atom expansion of the corresponding Lascoux atom $\oL_{1021}$ has all positive coefficients:
\[
\oL_{1021}= \overline\kappa_{1021} + \beta\overline\kappa_{2021}  + \beta\overline\kappa_{1211}+ 2\beta\overline\kappa_{1121} + \beta^2\overline\kappa_{1221}  +  \beta^2\overline\kappa_{2211}+2\beta^2 \overline\kappa_{2121} + \beta^3\overline\kappa_{2221}.
\]
However, negative coefficients occur in other atom expansions. For example:
\[
\oL_{012} = \overline\kappa_{012} + \beta \overline\kappa_{022} + \beta \overline\kappa_{112} - \beta \overline\kappa_{121}
+2\beta^2 \overline\kappa_{122} + \beta^2 \overline\kappa_{212} + \beta^3\overline\kappa_{222}.
\]
\end{example}

Lascoux polynomials are closely related to \defn{Grothendieck polynomials} $\fkG_w$, which may be defined as the unique elements of $\ZZ[\beta][x_1,x_2,\dots]$ indexed by $w \in S_\infty$ such that 
$\fkG_w = x^{\lambda(w)}$ if $w$ is dominant~\cite[Thms.~2 and 3]{Matsumura} and $\bpartial_i \fkG_w = \fkG_{ws_i}$ for all $i \in \PP$ with $w(i) > w(i+1)$~\cite[\S3]{Matsumura}.
The second property implies that  $\bpartial_i \fkG_w = -\beta\fkG_{ws_i}$ if $i \in \PP$ and $w(i) < w(i+1)$.
The polynomials $\fkG_w $ are homogeneous of degree $\ell(w)$ if we set $\deg(\beta):=-1$.
They generalize Schubert polynomials via the identity $\fkG_w \bigr\rvert_{\beta=0} = \fkS_w$.
Grothendieck polynomials represent connective $K$-theory classes of Schubert varieties in the complete flag variety~\cite[Thm.~1.2]{Hudson2014}.

By definition, $\fkG_w$ is a Lascoux polynomial whenever $w \in S_\infty$ is dominant. 
The set of all Grothendieck polynomials that are Lascoux polynomials has a nice classification.
Recall that the (Lehmer) code of $w \in S_\infty$ is the sequence $c(w)=(c_1,c_2,\ldots)$, where $c_i  =|\{ j \in \PP \mid (i,j) \in D(w)\}|$.

\begin{theorem}\label{vex-thm}
Suppose $w \in S_\infty$ and $\alpha$ is a weak composition. Then the following are equivalent:
(a) $ \fkG_w =  L_\alpha$,
(b) $\fkS_w = \kappa_\alpha$, and
(c) $w$ is vexillary with $c(w)=\alpha$.
\end{theorem}

\begin{proof}
Properties (b) and (c) are equivalent by~\cite[Thm.~22]{ReinerShimozono}.
If (a) holds then setting $\beta=0$ recovers (b) which implies (c).
Conversely, if (c) holds then~\cite[Thms.~2 and 3]{Matsumura} assert
 that $\fkG_w = \pi_u x_1^{a_1}\cdots x_r^{a_r}$
for some $u \in S_\infty$ and $a_1,\dotsc,a_n \in \NN$, so $\fkG_w = L_\gamma$ for a 
weak composition $\gamma$. As setting $\beta=0$ in this case gives $\fkS_w = \kappa_\gamma$, we must have $\gamma = c(w)=\alpha$.
Thus, (c) implies (a).
\end{proof}

The general relationship between $\fkG_w$ and $L_\alpha$ is as follows:
For each $w\in S_\infty$, we have
\be\label{sy-thm} \fkG_w = \sum_{\alpha} n_{w}^{\alpha}  \beta^{\abs{\alpha}-\ell(w)}  L_\alpha
\quand \fkS_w = \sum_{\abs{\alpha}=\ell(w)} n_{w}^{\alpha}  \kappa_\alpha\ee
for  coefficients $n_{w}^{\alpha}\in \NN$ that are zero for all but finitely many 
 $\alpha$~\cite[Thm.~1.9]{ShimozonoYu}.
Shimozono and Yu~\cite{ShimozonoYu} provide a combinatorial interpretation of the coefficients $n_w^\alpha$.

\begin{example}\label{groth-ex}
The first Grothendieck polynomial whose expansion into Lascoux polynomials has multiple terms is
\[
\fkG_{2143} = x_1 x_3 + x_1 x_2 + x_1^2 +\beta  x_1 x_2 x_3 +\beta x_1^2 x_3 + \beta x_1^2 x_2 + \beta^2 x_1^2 x_2 x_3
= L_2 + L_{101} + \beta L_{201}.
\]
\end{example}

\begin{proposition}\label{basis-prop}
For each $n \in \NN$, 
the sets of Grothendieck polynomials $\{ \fkG_w \mid \DesR(w) \subseteq[n]\}$
 and  Lascoux polynomials $\{ L_\alpha \mid \alpha \in \NN^n \}$
are both $\ZZ[\beta]$-bases for $\ZZ[\beta][x_1,x_2,\dotsc,x_n]$.
\end{proposition} 

\begin{proof}
It is known that the set of all Grothendieck polynomials is linearly independent over $\ZZ[\beta]$,
as their lowest degree terms are Schubert polynomials; see \cite{Lenart2003}.
Next, if $x_i$ is the largest variable appearing in $\fkG_w$ then $\bpartial_i \fkG_w \neq -\beta \fkG_w$
so we must have $i \in \DesR(w)$ by the divided difference recurrence defining $\fkG_w$.
Thus if $\DesR(w) \subseteq[n]$ then $\fkG_w  \in \ZZ[\beta][x_1,x_2,\dotsc,x_n]$.
Finally, 
it follows from Lenart's transition for Grothendieck polynomials (see~\cite[Thm.~3.1]{Lenart2003} or~\cite[Thm.~3.1]{MP2019a}) 
that each monomial $x_{i_1}x_{i_2}\cdots x_{i_k}$ with $1 \leq i_1 \leq i_2 \leq \dots \leq i_k \leq n$
is a finite $\ZZ[\beta]$-linear combination of elements of $\{ \fkG_w : \DesR(w) \subseteq[n]\}$,
so this linearly independent set is a basis for $\ZZ[\beta][x_1,x_2,\dotsc,x_n]$.

In view of~\eqref{sy-thm}, each $\fkG_w$ is a $\ZZ[\beta]$-linear combination of 
Lascoux polynomials. If $\DesR(w) \subseteq [n]$ then each $L_\alpha$ appearing in this expansion
must have $\alpha \in \NN^n$ since the minimal term of $L_\alpha$ is $x^\alpha$.
The linearly independent set  $\{ L_\alpha \mid \alpha \in \NN^n \}$
is therefore another $\ZZ[\beta]$-bases for $\ZZ[\beta][x_1,x_2,\dotsc,x_n]$.
\end{proof}

There is a monomial-positive expression for $\fkG_w$ due to Knutson and Miller~\cite{KM04} generalizing the \defn{Billey--Jockusch--Stanley formula}
for $\fkS_w$~\cite[Thm.~1.1]{BJS}.
A \defn{Hecke word} for $w$ is a finite sequence of integers $i_1i_2\cdots i_l$
with $w = s_{i_1} \circ s_{i_2} \circ \cdots \circ s_{i_l}$. Let $\cH(w)$ be the set of Hecke words for $w$ and let 
$\hf(w)$ denote the set of sequences of strictly decreasing words $a=(a^1,a^2,\ldots)$ with concatenation $a^1a^2\cdots \in \cH(w)$.
We refer to elements of $\hf(w)$ as \defn{Hecke factorizations}.
Define $\bhf(w)$ to be the set of \defn{bounded} Hecke factorizations $a \in \hf(w)$ that have $i \leq \min(a^i)$ when $a^i$ is nonempty.
Then 
$
\fkG_w = \sum_{a \in \bhf(w)} \beta^{|\weight(a)| - \ell(w)} x^{\weight(a)},
$ 
where  $\weight(a) := (\ell(a^1), \ell(a^2),\ldots)$ for $a \in \hf(w)$~\cite[Cor.~5.4]{KM04}.

\begin{example}
The Hecke words of $w=2143 \in S_4$ consists of all finite sequences $i_1i_2\cdots i_l$ with $\{i_1,i_2,\dotsc,i_l\} =\{1,3\}$.
The set $\bhf(2143)$ is finite with elements $(1,\emptyset,3)$, $(1,3,\emptyset)$, $(31,\emptyset,\emptyset)$, $(1,3,3)$, $(31,\emptyset,3)$, $(31,3,\emptyset)$, and $(31,3,3)$, corresponding to the monomials in Examples~\ref{groth-ex}.
\end{example}

Following~\cite{FominKirillov}, define $G_w :=  \sum_{a \in \hf(w)} \beta^{|\weight(a)| - \ell(w)} x^{\weight(a)}$ for $w \in S_\infty$.
 Since incrementing all letters by one defines a bijection $\cH(w) \to \cH(1\times w)$,
we can also write $G_w = \lim_{N\to \infty} \fkG_{1^N\times w}$, where the limit is in the sense 
of formal power series
and $1^N\times w \in S_\infty$ denotes the permutation fixing each $i \in [N]$ and mapping $i+N \mapsto w(i)+N$.
The following provides another proof (see, e.g.,~\cite{BKSTY}) that each of the \defn{stable Grothendieck polynomials} $G_w$ is 
a symmetric function in the $x_i$ variables and specializes to $F_w$ from~\eqref{FPQ-eq} when $\beta=0$.

\begin{lemma}
\label{bpi-lem1}
If $w \in S_\infty$ has $\DesR(w)\subseteq[n]$ then $G_w(x_1,x_2,\dotsc,x_n) = \bpi_{w_n} (\fkG_w)$.
\end{lemma}

\begin{proof}
As noted above, we have $\fkG_w \in \ZZ[\beta][x_1,x_2,\dotsc,x_n]$ if $w \in S_\infty$ has $\DesR(w)\subseteq[n]$.
We also have $G_w = \lim_{N\to \infty} \bpi_{w_N} (\fkG_w)$ by
\cite[Corollary~4.6]{MP2019b}.
Thus, the desired identity holds by~\eqref{mn-pi-eq}.
\end{proof}

For each partition $\lambda$
set $G_\lambda := G_w$, where $w \in S_\infty$ is the dominant element of shape $\lambda$.
Clearly $\DesR(w)\subseteq[n]$ since one has $(i,w(i+1))\in D(w)$ if $w(i) >w(i+1)$, so Lemma~\ref{bpi-lem1} recovers the following result from~\cite{Mon16}.

\begin{corollary}[{\cite{Mon16}}] \label{partial-lascoux-cor}
If $\lambda \in \NN^n$ is a partition then $G_\lambda(x_1,x_2,\dotsc,x_n) =  L_{(\lambda_n,\dotsc,\lambda_2,\lambda_1)}$.
\end{corollary}

One can also express $G_\lambda$ as 
a generating function
for \defn{semistandard set-valued tableaux} of shape $\lambda$~\cite[Thm.~3.1]{Buch2002}.
Each $G_w$ is a $\NN[\beta]$-linear combination of $G_\lambda$'s~\cite[Thm.~1]{BKSTY}. More precisely:

\begin{proposition}\label{Gexp-prop}
If $w \in S_\infty$ has $\DesR(w)\subseteq[n]$ then $G_w \in \NN[\beta]\spanning\{ G_\lambda \mid \text{partitions }\lambda \in \NN^n\}$.
\end{proposition}

\begin{proof}
The result~\cite[Thm.~1]{BKSTY} 
asserts that $G_w = \sum_\lambda c_{w,\lambda}  \beta^{|\lambda|-\ell(w)}  G_\lambda$,
where $c_{w,\lambda}$ is the number of \defn{increasing tableaux} of shape $\lambda$
(that is, with strictly increasing rows and columns)  whose column reading words are in $\cH(w^{-1})$.
The first column of such a tableau $T$ is a strictly increasing sequence of positive integers whose 
last element is the first letter of a Hecke word for $w^{-1}$,
and is therefore both a left descent of $w^{-1}$
and a right descent of $w$. The number of rows of $T$ cannot exceed this descent.
Hence if  $\DesR(w) \subseteq [n]$ and $c_{w,\lambda} \neq 0$, then $\lambda$ must have at most $n$ parts. 
\end{proof}


\subsection{Symplectic Grothendieck polynomials}

 Given $z \in \Ifpf_\infty$, let $\DSp(z) := \{(i,j) \in D(z) \mid i>j\}$.
 Define $\oplus$ as in \eqref{o-plus-minus-eq}.
The \defn{symplectic Grothendieck polynomials} $\fkGS_z$
are the unique elements of $\ZZ[\beta][x_1,x_2,\dots]$ indexed by $z \in \Ifpf_\infty$
with
\be \fkGS_z = \prod_{(i,j) \in \DSp(z)} (x_i\oplus x_j)
\quad \text{if $z \in \Ifpf_\infty$ is dominant} 
\ee
and such that for all $i \in \PP$, we have
\be
 \bpartial_i \fkGS_z  = \begin{cases} \fkGS_{s_i z s_i}&\text{if $i+1 \neq z(i) > z(i+1) \neq i$},\\
 -\beta\fkGS_z &\text{otherwise}.\end{cases}
 \ee
The existence of these polynomials follows from~\cite[Thm.~3.8]{MP2019a} and \cite[Thm.~4]{WyserYong}.
These polynomials were first considered in~\cite{WyserYong}, where it is shown that they represent  the connective $K$-theory classes of orbit closures in the complete flag variety for the symplectic group.

If $\deg(\beta) :=-1$ then each $\fkGS_z$ is homogeneous of degree $\abs{\DSp(z)}$, while if $\deg(\beta) := 0$ (or any nonnegative number) then $\fkGS_z$ is inhomogeneous with lowest degree term given by the \defn{symplectic Schubert polynomial} $\fkSS_z := \fkGS_z\bigr\rvert_{\beta=0}$~\cite[\S2.4]{MP2019a}.
It will be clear from the formula~\eqref{pre-GPz-eq} below that the set $\left\{\fkGS_z \mid z \in \Ifpf_\infty\right\}$ is linearly independent over $\ZZ[\beta]$.


For $z \in \Ifpf_\infty$ and $i \in \PP$ define
\[
z\ast s_i := \begin{cases}
z & \text{if } i+1\neq z(i)>z(i+1) \neq i, \\
s_iz s_i & \text{if } z(i)<z(i+1), \\
\emptyset  & \text{if } z(i)= i+1,
\end{cases}
\]
where $\emptyset \notin \Ifpf_\infty $ is a null element.
Also set $\emptyset \ast s_i = \emptyset$ for all $i$.
A \defn{symplectic Hecke word} for $z$ is a finite sequence of integers $i_1i_2\cdots i_l$ with
\[
z = 1_\fpf \ast s_{i_1} \ast s_{i_2} \ast \cdots \ast s_{i_l} := (\cdots ((1_\fpf \ast s_{i_1}) \ast s_{i_2}) \ast \cdots) \ast s_{i_l}.
\]
Let $\cHSp(z)$ be the set of symplectic Hecke words for $z \in \Ifpf_\infty$.
No element of this set may begin with an odd letter since $1_\fpf \ast s_i = \emptyset$ whenever $i$ is odd.
Likewise, if $i_1i_2\cdots i_l$ is a symplectic Hecke word and $i_2$ is odd then we must have $\abs{i_1-i_2} = 1$, 
since if $i_2$ is odd and $\abs{i_1-i_2} \neq 1$
then we must have $\abs{i_1-i_2}>1$ and then $(1_\fpf \ast s_{i_1}) \ast s_{i_2} = \emptyset$. 
 
The following property is useful for generating $\cHSp(z)$.
Define  $\simFKK$ to be the transitive closure of the usual braid relations for the symmetric group (allowing us to commute letters that differ by at least two and transform $\cdots iji\cdots \leftrightarrow \cdots jij\cdots$ if $|i-j|=1$) and the extra relations with 
 \[
 \cdots ii \cdots \simFKK \cdots i \cdots \qquand i_1i_2i_3\cdots i_l \simFKK i_1(i_2+2)i_3\cdots i_l \quad\text{if }i_1 = i_2 +1.
 \]
Then
 for each $z \in \Ifpf_\infty$ the set $\cHSp(z)$ is a single $\simFKK$-equivalence class~\cite[Thm.~2.4]{Marberg2019a}.
It follows that 
there is a finite set $\cBSp(z) \subset S_\infty$ with 
$\cHSp(z) = \bigsqcup_{w \in \cBSp(z)} \cH(w)$; see~\cite[Thm.~2.5]{Marberg2019a}. 

\begin{example}\label{simFKK-ex}
If $z = (1,4)(2,3) \in \Ifpf_4$ then $\cBSp(z) = \{ 1342, 3142, 3124\} \subset S_4$.
\end{example}

Recall from \eqref{visible-eq} that the \defn{fpf-visible descent set} of $z \in \Ifpf_\infty$ is
$ \Des^\fpf_V(z) :=\{ i \in \PP : z(i+1) < \min\{i,z(i)\}\}
.
$
If $i \in \Des^\fpf_V(z)$ then row $i$ of $\DSp(z)$ contains the position $(i,z(i+1))$.

\begin{lemma}\label{desv-lem}
Let  $z \in \Ifpf_\infty$.
Then
$i \in \DesR(w) $ for some $w \in \cBSp(z)$ if and only if $i$ is the last letter of a symplectic Hecke word for $z$.
When this occurs, there is an element $j \in \Des^\fpf_V(z)$ with $i\leq j$.
\end{lemma}

\begin{proof}
If $i \in \DesR(w) $ for some $w \in \cBSp(z)$, then $w$ has a reduced word ending in $i$ and this word is a symplectic Hecke word for $z$.
Conversely, if $i_1i_2\cdots i_l$ is a symplectic Hecke word for $z$ then 
$w=s_{i_1} \circ s_{i_2} \circ \cdots \circ s_{i_l}$ is an element of $\cBSp(z)$ and  $i_l \in \DesR(w)$ since $w\circ s_{i_l} = w$.
This proves our first claim.

If $z$ has a symplectic Hecke word 
 ending in $i$, then $z$ must have $i+1\neq z(i)>z(i+1) \neq i$.
 If $z(i+1) < i$ then $j:=i$ is in $\Des^\fpf_V(z)$. If $z(i+1) > i$ then the pair $(a,b) := (z(i+1), z(i))$
has $z(b) < \min\{ a,z(a)\}$.
Call such pairs
 \defn{visible inversions}.
If $(j,k)$ is the lexicographically maximal visible inversion of $z$ then $j$ must be an fpf-visible descent
(as otherwise $(j+1,k)$ would be a visible inversion),
and we have $i < z(i+1) = a \leq j \in \Des^\fpf_V(z)$.
\end{proof}

Let
$\hfSp(z)$ be the set of sequences of strictly decreasing words $a=(a^1,a^2,\ldots)$ with  $a^1a^2\cdots \in \cHSp(z)$.
Let $\bhfSp(z)$ be the set of  $a \in \hfSp(z)$ with $i \leq \min(a^i)$ if $a^i$ is nonempty.
Then 
\be\label{pre-GPz-eq} \fkGS_z = \sum_{a \in \bhfSp(z)} \beta^{|\weight(a)| - |\DSp(z)|} x^{\weight(a)}=   \sum_{w \in \cBSp(z)} \beta^{\ell(w)- |\DSp(z)|} \fkG_w\ee
by~\cite[Thm.~3.12]{MP2019a}. 
We pause to note a consequence of~\eqref{pre-GPz-eq}.

\begin{proposition}\label{sp-basis-prop}
If  $z \in \Ifpf_\infty$ then $\fkGS_z\in \ZZ[\beta][x_1,x_2,\dotsc,x_n]$ 
if and only if $ \Des^\fpf_V(z)\subseteq[n]$.
\end{proposition}

\begin{proof}
We observed in Remark~\ref{visible-remark} that
$ \Des^\fpf_V(z)\subseteq[n]$  if and only if $\DSp(z)\subseteq [n]\times [n] $.
Thus if $\fkGS_z\in \ZZ[\beta][x_1,\dotsc,x_n]$  then $\fkSS_z\in \ZZ[x_1,\dotsc,x_n]$ 
so $ \Des^\fpf_V(z)\subseteq[n]$ by Proposition~\ref{lexmin-prop}.
If $ \Des^\fpf_V(z)\subseteq[n]$ then each $\fkG_w$ in~\eqref{pre-GPz-eq} is in $\ZZ[\beta][x_1,\dotsc,x_n]$
by  Proposition~\ref{basis-prop} and Lemma~\ref{desv-lem}.
\end{proof}

Removing the boundedness condition in~\eqref{pre-GPz-eq} defines the symmetric functions
\be\label{GPz-eq}
\GP_z := \sum_{a \in \hfSp(z)} \beta^{|\weight(a)| - |\DSp(z)|} x^{\weight(a)}=   \sum_{w \in \cBSp(z)} \beta^{\ell(w)- |\DSp(z)|} G_w\ee
called \defn{symplectic stable Grothendieck polynomials} in~\cite{Marberg2019a,MP2019a,MP2019b}.
As incrementing all letters by two defines a bijection $\cHSp(z) \to \cHSp(21\times z)$,
one has
$\GP_z := \lim_{N\to \infty} \fkGS_{21^N \times z}$ and $\GP_z = \GP_{21\times z}$.
Here $21^N\times z$ denotes the element of $\Ifpf_\infty$ that sends $i \mapsto 1_\fpf(i) = i - (-1)^i$ for $i \in [2N]$ and $i+2N \mapsto z(i) + 2N$ for $i \in \PP$, and $21\times z := 21^1 \times z$.

The following shows that $\GP_z\bigr\rvert_{\beta=0}=P_z$ from~\eqref{FPQ-eq}:

\begin{lemma} \label{bpi-lem2}
Suppose $z \in \Ifpf_\infty$ has $\Des^\fpf_V(z)\subseteq [n]$. Then 
$\GP_z(x_1,x_2,\dotsc,x_n) = \bpi_{w_n} (\fkGS_z) . $
\end{lemma}

\begin{proof}
Lemma~\ref{desv-lem} implies that $\DesR(w) \subseteq[n]$ for all $w \in \cBSp(z)$.
Given this observation and the definition of $\GP_z$ in \eqref{GPz-eq}, the desired formula 
follows from Lemma~\ref{bpi-lem1}.
\end{proof}

The symmetric functions $\GP_z$ are closely related to Ikeda and Naruse's \defn{$K$-theoretic Schur $P$-functions}, which are defined in \cite{IkedaNaruse} in the following way.
 For each strict partition $\mu$ with $r$ parts, there is a unique 
power series $\GP_\mu \in \ZZ[\beta]\llbracket x_1,x_2,\dots\rrbracket $ with
\be\label{GP-def1}
\GP_\mu(x_1,\dotsc,x_n) = \frac{1}{(n-r)!}\sum_{w \in S_n} w\( x^\mu \prod_{i=1}^{r}  \prod_{j=i+1}^n \frac{x_i\oplus x_j}{x_i\ominus x_j}\)
\ee
for all $n\geq r$~\cite[\S2]{IkedaNaruse},
with $\oplus$ and $\ominus$ defined as in \eqref{o-plus-minus-eq}.
This symmetric function also satisfies
\be\label{GP-def2}
\GP_\mu(x_1,\dotsc,x_n) = \bpi_{w_n}\( x^\mu \prod_{i=1}^{r} \prod_{j=i+1}^n \frac{x_i\oplus x_j}{x_i}\)
\ee
for all $n\geq r$ by~\cite[Prop.~4.11]{MP2019b}. 
We recover the ordinary Schur $P$-functions as $P_\mu = \GP_\mu\bigr\rvert_{\beta=0}$~\cite[Ex.~1, \S{III.3}]{Macdonald}.
One can express $\GP_\mu$ as 
a generating function
for \defn{semistandard set-valued shifted tableaux} of shape $\mu$~\cite[Thm.~9.1]{IkedaNaruse}.
As $\mu$ ranges over all strict partitions in $\NN^n$,
the polynomials $\GP_\mu(x_1,x_2,\dotsc,x_n)$ form a $\ZZ[\beta]$-basis for the subring of symmetric elements of $\ZZ[\beta][x_1,x_2,\dotsc,x_n]$ satisfying a certain \defn{$K$-theoretic $Q$-cancellation property}~\cite[Thm.~3.1]{IkedaNaruse}.
 
Each $\GP_z$ is a finite $\NN[\beta]$-linear combination of $\GP_\mu$'s~\cite[Thm.~1.9]{Marberg2019a}. 
More precisely:

\begin{proposition}\label{GSp-exp-prop}
If $z \in \Ifpf_\infty$ has $\Des^\fpf_V(z)\subseteq [n]$ then 
\[\GP_z \in \NN[\beta]\spanning\{ \GP_\mu \mid \text{strict partitions $\mu$ with $\max(\mu)\leq n$}\}.\]
\end{proposition}

\begin{proof}
\cite[Thm.~1.9]{Marberg2019a}
asserts  that
$\GP_z = \sum_\mu a_{z,\mu}  \beta^{|\mu|-|\DSp(z)|}  \GP_\mu$,
where the sum is over strict partitions $\mu$ and 
$a_{z,\mu}$ is the number of increasing shifted tableaux of shape $\mu$
with 
 row reading word  in $\cHSp(z)$.
If $\Des^\fpf_V(z)\subseteq [n]$ then the first row of such a tableau is an increasing sequence of positive integers 
$\leq n$ by Lemma~\ref{desv-lem}, so its shape is a strict partition $\mu$ with $\max(\mu) \leq n$.
\end{proof}

Each $\GP_\mu$ occurs as $\GP_z$ for some $z \in \Ifpf_\infty$~\cite[Thm.~4.17]{MP2019b}. 
We can classify  when  $\GP_z = \GP_\mu$.
For $z \in \Ifpf_\infty$, let $\cSp(z) = (c_1,c_2,\ldots)$ where $c_i = |\{j : (i,j) \in \DSp(z)\}|$, and define  $\lambdaSp(z)$ to be the transpose of the partition
sorting $\cSp(z)$.
This partition is always strict~\cite[Thm.~1.4]{HMP5}, and 
if $z$ is dominant then  $\lambdaSp(z) = \shalf(\lambdafpf(z))$.
Recall that $z$ is fpf-vexillary if the power series $P_z$ from~\eqref{FPQ-eq} is a Schur $P$-function,
or equivalently if $z$ avoids all of the patterns in~\cite[Cor.~7.9]{HMP5}.

\begin{theorem}\label{gp-vex-thm}
Suppose $z \in\Ifpf_\infty$ and $\mu$ is a strict partition. Then the following are equivalent:
(a) $\GP_z=\GP_\mu$, (b) $P_z =P_\mu$, and (c) $z$ is fpf-vexillary with $\mu=\lambdaSp(z)$.
\end{theorem}

\begin{proof}
The result \cite[Thm.~1.4]{HMP5} identifies $P_{\lambdaSp(z)}$ as the leading term of $P_z$ for any $z \in \Ifpf_\infty$.
Hence if $z$ is fpf-vexillary then $P_z = P_{\lambdaSp(z)}$.
Conversely, if $P_z = P_\mu$ then $z$ is fpf-vexillary and we must have $\mu=\lambdaSp(z)$ since Schur $P$-functions
are uniquely indexed by strict partitions. This establishes the equivalence of (b) and (c).

Next, (a) implies (b) since $\GP_z=\GP_\mu$ gives $P_z = P_{\mu}$ by setting $\beta=0$ (which also implies (c): that $\mu=\lambdaSp(z)$).
Thus, it is enough to show that (c) implies (a); that is, if $z$ is fpf-vexillary, then $\GP_z$ is equal to $\GP_\mu$ for some strict partition $\mu$.
We show this below as a consequence of~\cite[Cor.~4.21]{MP2019b} and results in~\cite{HMP5}.

For $y,z \in \Ifpf_\infty$ write $y \lessdot_F z$ if there are positive integers $i < j$
with $z=(i \; j) y (i \; j)$ such that (1) no integer $e$ has both $i<e<j$ and $y(i) <y(e) <y(j)$ and (2) either 
$
y(i)<i<j<y(j)$ or $ y(i) < y(j) < i <j $ or $ i<j<y(i)<y(j).
$
Then $\lessdot_F$ is the covering relation for the \defn{Bruhat order} on $\Ifpf_\infty$ discussed in~\cite[\S3.2]{MP2019b}.

A recurrence relation for $\GP_z$ was given in~\cite[Cor.~4.21]{MP2019b}, which can be described as follows.
Assume $z \in \Ifpf_\infty$ has $z(1) = 2$ and $z \neq 1_\fpf$.
Then there is a maximal $k \in \Des^\fpf_V(z)$. Let $l$ be the largest integer with $z(l) < \min\{k,z(k)\} \leq k <l$, and set $v := (k \; l)z(k \; l)$ and $j := v(k)$.
Let $I(z)$ be the set of $i\in [j-1]$ with $v \lessdot_F (i,j) v(i,j)$.
Then $\GP_z = \sum_{\varnothing \neq A \subseteq I(z)} \beta^{|A|-1} \GP_{\varepsilon_A(z)}$,
where for $A = \{a_1<a_2<\dots<a_q\}$ we set
$
\varepsilon_A(z) := (a_q \; j) \cdots (a_2 \; j)(a_1 \; j) \cdot v \cdot (a_1 \;j)(a_2 \;j) \cdots (a_q \; j).
$

We claim that if $z$ is fpf-vexillary then $\abs{I(z)} = 1$.
We prove this by showing that if any counterexample exists, then it must occur for some $z \in \Ifpf_{12}$.
This reduces our claim to a finite calculation. Our argument is similar to the methods in \cite[\S7]{HMP5} and so we have condensed
some details here.

We need the following notation. Suppose $S$ is a finite set of $n$ positive integers with $z(S) = S$.
Let $\phi_S \colon [n] \to S$ and $\psi_S \colon S \to [n]$ be  order-preserving bijections, and define $\llbracket z\rrbracket _S \in \Ifpf_n$ to be the unique element mapping $i \in [n]$ to $\psi_S\circ z \circ \phi_S(i)$.
Since the set of fpf-vexillary elements is characterized by a pattern avoidance condition~\cite[Cor.~7.9]{HMP5},
if $z$ is fpf-vexillary then  $\llbracket z\rrbracket _S$ remains so. 
It is also clear that if $y \in \Ifpf_\infty$ has $y(S) = S$ and $y \lessdot_F z$ then $\llbracket y\rrbracket _S \lessdot_F \llbracket z\rrbracket _S$.

Now assume that $z \in \Ifpf_\infty$ has $z(1) = 2$ and $z\neq 1_\fpf$. Define $j$, $k$, and $l$ as above.
We cannot have $|I(z)| = 0$ since $\GP_z$ is nonzero.
Suppose that $z$ is fpf-vexillary and $|I(z)|>1$. Let $h$ and $i$ be the largest elements of $I(z)$
and define $S := \{1,2,h,i,j,k,l, z(h), z(i),z(j), z(k),z(l)\}$. Then $\llbracket z\rrbracket _S \neq 1_\fpf$ is also fpf-vexillary with $(1 \; 2)$ as a cycle, 
and it holds that $\{ \psi_S(h),\psi_S(i)\} \subset I(\llbracket z\rrbracket _S)$.
The set $S$ has size at most $12$ so we may assume without loss of generality that $z $ belongs to $ \Ifpf_{12}$.
There are only $9\cdot 7 \cdot 5 \cdot 3 \cdot 1 - 1 = 944$ elements $z\in \Ifpf_{12}$ with $z(1)=2$ and $z \neq 1_\fpf$, and we
have checked using a computer that none of them is fpf-vexillary with $\abs{I(z)} > 1$. 
(The \textsc{Python} code we used to verify this is included as an ancillary file with the {\tt arXiv} version of this paper.)

Finally let $z \in \Ifpf_\infty$ be any fpf-vexillary element.
Then $21\times z$ is also fpf-vexillary, by our claim $I(21 \times z) = \{a\}$ is a singleton set, and $\GP_z = \GP_{21\times z} = \GP_{\varepsilon(z)}$
for the element $\varepsilon(z) := \varepsilon_{\{a\}}(21 \times z)$, which is also fpf-vexillary.
The involution $\varepsilon(z)$ is the unique element of the set $\hat {\fk T}^{\mathtt{FPF}}(21 \times z)$ defined by~\cite[Eq.~(5.1)]{HMP5}.
It follows from~\cite[Thm.~5.21]{HMP5} that some $y \in \{z, \varepsilon(z), \varepsilon^2(z), \varepsilon^3(z),\dots\}$ is \defn{fpf-Grassmannian} in the sense of~\cite[Def.~4.14]{HMP5}.
We have $\GP_z = \GP_y$ for this element, and~\cite[Thm.~4.17]{MP2019b} computes that $\GP_y = \GP_{\mu}$ for the strict partition $\mu = \lambdaSp(y)$ as needed.
\end{proof}

\subsection{\texorpdfstring{$P$}{P}-Lascoux polynomials}\label{p-lasc-sect}

Lascoux polynomials are defined by applying the operators $\bpi_i$ to Grothendieck polynomials of dominant permutations.
This suggests that we can define a symplectic analogue of $L_\alpha$ by applying the same operators to $\fkGS_z$, in the following way:

\begin{definition} \label{LSp-def} Let $\alpha$ be a skew-symmetric weak composition
 with $u := u(\alpha)$ and $\lambda :=\lambda(\alpha)$.
Then the \defn{$P$-Lascoux polynomial} and \defn{$P$-Lascoux atom} of $\alpha$ are respectively
\be
\plascoux_{\alpha} = \bpi_u \(\prod_{\substack{(i,j) \in \D_\lambda \\ i > j} } (x_i \oplus  x_{j})\)
\quand
\plascouxatom_{\alpha} = \bopi_u \(\prod_{\substack{(i,j) \in \D_\lambda \\ i > j} } (x_i \oplus  x_{j})\).
\ee
 \end{definition}
 
 These objects generalize $\pkey_\alpha  $ and $\patom_\alpha$,
 which are recovered from $\plascoux_{\alpha} $ and $\plascouxatom_{\alpha} $ by setting $\beta=0$.
 The polynomial $\plascoux_{\alpha} $ is always nonzero 
 and homogeneous of degree $|\shalf(\lambda(\alpha))|$
 when we set $\deg(\beta) := -1$.
 In this case, the $P$-Lascoux atom $\plascouxatom_{\alpha} $ has the same degree as $\pkey_{\alpha}$ when it is nonzero.  
 As in the key case,
we can have $\plascouxatom_\alpha=0$, and there are coincidences among the $\plascoux_\alpha$'s.
Both~\eqref{sp-red-ex1} and~\eqref{sp-red-ex2} remain true with each ``$\pkey$'' replaced by ``$\plascoux$''.
However, somewhat curiously,
we have not found distinct skew-symmetric weak compositions $\alpha\neq \gamma$ with $\plascouxatom_\alpha = \plascouxatom_\gamma \neq 0$.

 \begin{example}
If $\alpha = (3,1,4,3)$ then $u(\alpha) = s_2 s_1 s_3$ and
\[ 
\ba \plascoux_{3143} &= L_{0022} + L_{0031} + L_{0112} + \beta L_{0032} +\beta L_{0122} +\beta L_{0131} + \beta^2 L_{0132},
\\
\plascouxatom_{3143} &= \oL_{0022} + \oL_{0031} + \beta \oL_{0032} +\beta  \oL_{0122} +\beta \oL_{1022}+ \beta^2 \oL_{0132}  + \beta^2\oL_{1032}.
\ea
\]
Although $\patom_{110115} = \patom_{400313} = x_4 x_5 x_6^2$ we compute that
$0 \neq \plascouxatom_{110115} - \plascouxatom_{400313} \in \beta \NN[\beta][x_1,\dots,x_6]$.
\end{example}

Recall that $\bpi_i f = f $ for $i \in \PP$ and $f \in \cL$ if and only if $s_i f = f$.
 
 \begin{lemma}\label{sp-lasc-lam-lem}
If $\lambda$ is a skew-symmetric partition
then
$
  \Bigl\{ i  \mid \lambda_i=\lambda_{i+1}\Bigr\} = \Bigl\{ i  \mid \bpi_i \plascoux_\lambda= \plascoux_\lambda\Bigr\} .
  $
  \end{lemma}

\begin{proof}
As $\pkey_\lambda$ is a product of factors of the form $x_i+x_j$ and
$\plascoux_\lambda$ is obtained by replacing each of these by $x_i\oplus x_j$,
we can deduce this result by repeating the proofs of Lemmas~\ref{lam-lam-lem} and \ref{ss-lam-lam-lem}.
 \end{proof}

\begin{proposition}\label{sp-lasc-formulas}
If $i \in \PP$ and $\alpha$ is a skew-symmetric weak composition then
\[
 \bpi_i \plascoux_\alpha =\begin{cases} \plascoux_{s_i \alpha} & \text{if }\alpha_i > \alpha_{i+1},
 \\
 \plascoux_\alpha &\text{if }\alpha_i < \alpha_{i+1},
  \\
 \plascoux_\alpha &\text{if }\alpha_i = \alpha_{i+1},
 \end{cases}
 \quand
  \bopi_i \plascouxatom_\alpha =\begin{cases} \plascouxatom_{s_i \alpha} & \text{if }\alpha_i > \alpha_{i+1},
 \\
 -\plascouxatom_\alpha &\text{if }\alpha_i < \alpha_{i+1},
  \\
0 &\text{if }\alpha_i = \alpha_{i+1}.
 \end{cases}
\]
 \end{proposition}
 
\begin{proof}
Using  Lemma~\ref{sp-lasc-lam-lem}, this result has the same proof as Proposition~\ref{formulas-prop}.
\end{proof}

It appears that each set $\mathcal{S}$ of skew-symmetric weak compositions indexing the same $P$-Lascoux polynomial has a unique minimum  in the partial order $\prec$, with all other elements $\alpha \in \mathcal{S}$ having $\plascouxatom_\alpha \neq 0$ in view of Proposition~\ref{sp-lasc-formulas}.
However, it can happen that $\plascouxatom_\alpha=0$ for every skew-symmetric weak composition $\alpha$ indexing the same $P$-Lascoux polynomial.
For example, $\plascouxatom_{04313} = \plascouxatom_{04133} = 0$.

 \begin{proposition}
 If $\alpha$ is a skew-symmetric weak composition then
 $\plascoux_\alpha = \sum_{\gamma \preceq \alpha} \plascouxatom_\gamma$.
 \end{proposition}
 
 \begin{proof}
 This follows from Lemma~\ref{var-bar-combine-lem} with $u=1$ in view of Proposition~\ref{sp-lasc-formulas}.
 \end{proof}

 \begin{proposition}\label{sp-lascoux-positive-prop}
If $\alpha$ is a skew-symmetric weak composition then  $\plascoux_\alpha$ 
is a nonzero $\NN[\beta]$-linear combination of Lascoux polynomials 
and  $\plascouxatom_\alpha$ is a $\NN[\beta]$-linear combination of Lascoux atom polynomials.
\end{proposition}

 \begin{proof}
  Suppose $\alpha = \lambda(\alpha)$ is a skew-symmetric partition.
 Then $\plascoux_\alpha = \plascouxatom_\alpha=\fkGS_z$ for the unique dominant $z \in \Ifpf_\infty$
 with $\lambdafpf(z) =\alpha$, which exists by \cite[Prop.~4.31]{HMP6}.
  Each $\fkGS_z$ is a $\NN[\beta]$-linear combination of a finite set of $\fkG_w$'s~\cite[Thm.~3.12]{MP2019a}
(see the discussion below), 
so by~\eqref{sy-thm}
 both $\plascoux_\alpha$ and $\plascouxatom_\alpha$ are  $\NN[\beta]$-linear combinations of Lascoux polynomials.
 This property is preserved by applying the relevant divided difference operators $\bpi_i$ and $\bopi_i$.
Hence, the result for arbitrary  $\alpha$ follows from this property given Lemma~\ref{var-bar-combine-lem} and Proposition~\ref{sp-lasc-formulas}.
 \end{proof}

 For a skew-symmetric weak composition $\alpha$, let $\DSp(\alpha) := \{ (i,j) \in \DSym_\alpha : i > j\}$.
Our motivation for considering the polynomials $\plascoux_\alpha$
derives primarily from the following generalization of Conjecture~\ref{fkSS-conj}, which we have checked by computer
for all $z \in \Ifpf_n$ for $n\leq 8$.

\begin{conjecture}\label{fkGS-conj}
Let $z\in \Ifpf_\infty$. Then $\fkGS_z $ is a $\{\beta^n : n \in \NN\}$-linear combination of $\plascoux_\alpha$'s. In fact,  
there is a set of skew-symmetric weak compositions $\cKSp(z)$ with
 $\fkGS_z = \sum_{\alpha \in \cKSp(z)} \beta^{|\DSp(\alpha)| - |\DSp(z)|} \plascoux_\alpha$.
\end{conjecture}

\begin{example}
For $z = (1,4)(2,3)(5,8)(6,7) \in \Ifpf_8$ one has 
\[
 \fkGS_{z} 
= \plascoux_{5110011} + \plascoux_{3340001} + \plascoux_{43103} + \beta \plascoux_{42402} + \beta \plascoux_{5330011} + \beta \plascoux_{5310301} + \beta^2 \plascoux_{5240201}.
\] 
\end{example}

Based on these results and that Lascoux polynomials have monomial-positive expansions, we have the analog of Problem~\ref{prob:monomial_expansion}.

\begin{problem}
\label{prob:K_monomial_expansion}
Determine combinatorial objects and weight functions whose corresponding generating functions are equal to $\plascoux_{\alpha}$,  $\qlascoux_{\alpha}$, $\plascouxatom_{\alpha}$, and $\qlascouxatom_{\alpha}$.
\end{problem}

One can also view $\plascoux_\alpha$ as generalizations of certain $K$-theoretic Schur $P$-functions.
 
\begin{corollary}\label{gp-dom-cor}
If $\lambda \in \NN^{n}$ is a skew-symmetric partition  and $\mu := \shalf(\lambda)$ 
then 
\[
\GP_\mu(x_1,x_2,\dotsc,x_n) =  \plascoux_{(\lambda_n,\dotsc,\lambda_2,\lambda_1)}.
\]
\end{corollary}

\begin{proof}
Let $z \in \Ifpf_\infty$ be dominant with $\lambdafpf(z) = \lambda$.
Then $\Des^\fpf_V(z)\subseteq [n]$ so $\GP_z(x_1,x_2,\dotsc,x_n) = \bpi_{w_n} (\fkGS_z) =  \plascoux_{(\lambda_n,\dotsc,\lambda_2,\lambda_1)}$
by Lemma~\ref{bpi-lem2}, and it holds that $\GP_{\mu} = \GP_z$ by Theorem~\ref{gp-vex-thm}.
\end{proof}

Assuming our stronger form of Conjecture~\ref{fkSS-conj}, we get a stronger form of Proposition~\ref{ss-vex-prop}:

 \begin{corollary}\label{sp-vex-conj}
If $z \in \Ifpf_\infty$ is fpf-vexillary and
 Conjecture~\ref{fkGS-conj} holds, then there is a skew-symmetric weak composition $\alpha$
 with $\fkGS_z = \plascoux_\alpha$ and $\fkSS_z = \pkey_\alpha$.
\end{corollary}

\begin{proof}
In this case $\GP_z= \GP_{\lambdaSp(z)}$ by Theorem~\ref{gp-vex-thm},
and applying $\bpi_{w_n}$ to the formula for $ \fkGS_z$ in  Conjecture~\ref{fkGS-conj}
for   sufficiently large $n$ implies that 
\[\textstyle \GP_{\lambdaSp(z)}=\GP_z = \sum_{\alpha \in \cKSp(z)}  \beta^{|\DSp(\alpha)| - |\DSp(z)|} \GP_{\shalf(\lambda(\alpha))} \]
 by Lemma~\ref{bpi-lem2} and Corollary~\ref{gp-dom-cor}. The set of power series $\{\GP_\mu : \text{strict partitions }\mu\}$ is linearly independent over $\ZZ[\beta]$
 by \cite[Prop.~3.4]{IkedaNaruse}, so  the set $\cKSp(z) = \{\alpha\}$ must
consist of a single skew-symmetric composition $\alpha$ with $\shalf(\lambda(\alpha)) = \lambdaSp(z)$,
and then   $\fkGS_z = \plascoux_\alpha$ and $\fkSS_z = \pkey_\alpha$.
\end{proof}

See Example~\ref{ss-vex-prop-ex} for instances of the preceding result.

Each $\GP_\mu$ is equal to $\GP_z$ for some dominant $z \in \Ifpf_\infty$ by Theorem~\ref{gp-vex-thm}
and is therefore a finite $\NN[\beta]$-linear combination of $G_w$'s by~\eqref{GPz-eq}, and hence also
of $G_\lambda$'s.
This property is stated as~\cite[Cor.~4.18]{MP2019b}. 
We can prove a more effective version of this expansion.

Namely, suppose $\mu$ is a strict partition and   $\lambda$ is the skew-symmetric partition with $\shalf(\lambda)=\mu$.
If $n := \max(\mu)>0$ and
 $z \in \Ifpf_\infty$ is     dominant   with $D(z) = \D_\lambda$,
then $\lambda_1 = n+1$ so  $\Des^\fpf_V(z) \subseteq [n+1]$
as  row $i$ of $\DSp(z)$ is nonempty whenever $ i \in \Des^\fpf_V(z)$.
We deduce that
\[
\GP_\mu = \GP_z \in   \NN[\beta]\spanning\left\{ G_w \mid w \in \cBSp(z)\right\}
 \subseteq   \NN[\beta]\spanning\left\{ G_\lambda \mid \text{partitions }\lambda \in \NN^{n+1}\right\}
\]
by Proposition~\ref{Gexp-prop} and Lemma~\ref{desv-lem}.

This property is nearly optimal, but it turns out that we can improve things slightly.
In the following lemma,
let $\row(T)$ denote the usual row reading word of a tableau $T$ 
(formed by reading entries left to right, but starting with the last row),
and let $\revrow(T)$ be its reverse.

\begin{lemma} \label{vdes-tab-lem}
Let $z \in \Ifpf_\infty$.
Suppose $T$ is an increasing tableau (of some unshifted partition shape)
whose first column is $1,2,\dotsc,n$. 
If $\revrow(T) \in \cHSp(z)$ then $\Des^\fpf_V(z) \not\subseteq [n]$.
\end{lemma}

\begin{proof}
Assume  $\revrow(T) \in \cHSp(z)$.
By Lemma~\ref{desv-lem}, it suffices to show that $\cHSp(z)$ has an element that ends with a letter greater than $n$.
Let $\sim$ be the equivalence relation on words that has $i_1i_2\cdots i_p \sim j_1j_2\cdots j_q$ if and only if $i_p \cdots i_2i_1 \simFKK j_q \cdots j_2j_1$.
Since $\cHSp(z)$ is a single $\simFKK$-equivalence class~\cite[Thm.~2.4]{Marberg2019a},
the relation $\sim$ preserves the set of reverse symplectic Hecke words, and
it is enough to produce a word $i_1i_2\cdots i_p \sim \row(T)$ with $i_1 > n$.

Given a sequence of positive integers $\alpha = (\alpha_1,\alpha_2,\dotsc,\alpha_{n})$,
let $U_\alpha$ 
be the tableau of shape $D_\alpha := \{ (i,j) \in [n]\times \PP : 1 \leq j \leq \alpha_i\}$
with entry $i+j-1$ in position $(i,j)$.
Assume the reverse row reading word of $U_\alpha$ is a symplectic Hecke word for some element of $ \Ifpf_\infty$.
Then $q:=\alpha_1$ must be even since $\row(U_\alpha)$ cannot end in an odd letter.
Moreover, we cannot have $\alpha_i =1$ for any $i$, for if $i \in [n]$ is minimal with this property and $\rho_j := j(j+1)(j+2)\cdots  (\alpha_j+j-1) $
is row $j$ of $U_\alpha$, then 
\[\ba  \row(U_\alpha) &= \rho_n\cdots \rho_{i+1}\cdot  i\cdot  \rho_{i-1}\cdots \rho_1
\\&\sim \rho_n\cdots \rho_{i+1}\rho_{i-1}\cdot (i-1) \cdot \rho_{i-2} \cdots \rho_1 
\\&\sim \rho_n\cdots \rho_{i+1}\rho_{i-1} \rho_{i-2}\cdot (i-2) \cdot \rho_{i-3} \cdots \rho_1 
\\&\ \vdots 
\\& \sim \rho_n \cdots  \rho_{i+1}\rho_{i-1} \cdots \rho_2 \cdot 2 \cdot\rho_1  \sim \rho_n \cdots  \rho_{i+1}\rho_{i-1} \cdots  \rho_1 \cdot 1
\ea
\]
and the last word ends in an odd letter (which cannot occur for any reverse symplectic Hecke word).

 We now show by induction that there is always a word $i_1i_2\cdots i_p \sim \row(U_\alpha)$ with $i_1 > n$. This holds when $n=1$ since 
 then  $\row(U_\alpha) = 123\cdots q \sim 123\cdots (q-2)(q+1)q  \sim (q+1)  123\cdots (q-2)q.$
For the general case, we consider the following insertion process.
Given a positive integer $x$, form $U_\alpha \leftarrow x$ by inserting $x$ into $U_\alpha$ according to the following procedure:
\ben
\item[] Start by inserting $x$ into the first row. 
If the row is empty or $x=y+1$ where $y$ is the row's last entry, then we add $x$ to the end of the row and the insertion process ends.
If $x=y$ then we erase $x$ and the insertion process also ends.
Otherwise, if $x> y+1$ (respectively, $ x <y$) then we insert $x$ (respectively, $x+1$) into the next row by same procedure.
\een
The output tableau $U_\alpha \leftarrow x$ either has the form $U_{\alpha+ \e_i}$ for some $i \in [n]$
or is formed from $U_\alpha$ by adding a box containing an integer greater than $n$ to row $n+1$. 
For example, we have
\[
\ytabsmall{ 
1 & 2 &3 & 4\\
2 & 3 & 4 & 5\\
3 & 4
} \leftarrow 3 = \ytabsmall{ 
1 & 2 &3 & 4 \\
2 & 3 & 4 & 5\\
3 & 4 & 5},
\quad
\ytabsmall{ 
1 & 2 &3 & 4\\
2  \\
3 & 4
} \leftarrow 3 = \ytabsmall{ 
1 & 2 &3 & 4 \\
2  \\
3 & 4 },
\quand
\ytabsmall{ 
1 & 2 &3 & 4\\
2  \\
3 & 4 & 5 
} \leftarrow 3 = \ytabsmall{ 
1 & 2 &3 & 4 \\
2  \\
3 & 4  & 5 \\ 
5 }.
\]

We claim that one always has  $\row(U_\alpha\leftarrow x) \sim \row(U_\alpha)$.
To see this, 
suppose the insertion process that forms $U_\alpha \leftarrow x$ results in a positive integer $x_j$ being inserted into the $j$th row $\rho_j=j(j+1)(j+2)\cdots (\alpha_j+j-1)$ of $U_\alpha$. Consulting the definition of $\simFKK$ before Example~\ref{simFKK-ex},
 we see that if $x_j>\alpha_j+j$, then $\rho_j x \sim  x \rho_j$ and if $x_j = \alpha_j+j-1$ then $\rho_j x_j \sim \rho_j$.
On the other hand, if   $  x_j < \alpha_j+j-1$,
then it follows by induction on $j$ that $j \leq x$  and so
 $\rho_j x_j \sim (x_j+1)\rho_j$.
Combing these observations produces our claim.

Now assume $n > 1$. Since $q := \alpha_1$ is even and 
$123 \cdots q \sim 123\cdots (q-2) (q+1)q \sim (q+1)q 123\cdots (q-1)$,
it follows by induction that $123\cdots q \sim (q+1)\cdots 432$. Let
$U^0$ be the tableau formed from $U_\alpha$ by the deleting its first row,
so that $\row(U_\alpha) \sim \row(U^0) (q+1)\cdots 432$.
Then construct $U^i$ for $i \in[q]$ by inserting $a_i := q+2-i$ into $U^{i-1}$ according to the process defined above.
Because each $a_i \geq 2$, this has the same effect as successively inserting $q,\dotsc,3,2,1$ into $U_{(\alpha_2,\alpha_3,\dotsc,\alpha_n)}$ and then adding one to all entries.

This process is only well-defined up to the minimal index $i \in [q]$ such that $U^{i}$ has more rows than $U^{i-1}$.
Such an index exists because if $U^{q-1}$ still has the same number of rows as $U^0$, then the first two columns of $U^{q-1}$ must be 
$2,3,\dotsc,n$ and $3,4,\dotsc,n+1$ as all parts of $\alpha$ are at least two, and so finally inserting $a_q=2$ will add $n+1$ to a new row. 
For this index, we have 
 $\row(U^i) a_{i+1}a_{i+2}\cdots a_q \sim \row(U_\alpha)$ and the first word starts with a letter greater than $n$ as desired.

Returning to the lemma, let $b_1b_2 \cdots b_r$ be the row reading row of $T$ with its first column removed.
Then $\row(T) \sim n \cdots 321 b_1b_2 \cdots b_r$.
Let $T^0$ be the first column of $T$ and define $T^i = T^{i-1}\leftarrow b_i$ for $i=1,2,\dotsc,r$.
If there is a minimal index $i \in [r]$ such that $T^i$ has more rows than $T^{i-1}$,
then  $\row(T^i) b_{i+1}b_{i+2}\cdots b_r \sim \row(T)$ and the first word starts with a letter greater than $n$.
Otherwise, we have $\row(T) \sim \row(T^r)$ and $T^r=U_\alpha$ for a strict weak composition  $\alpha$,
so by the claim proved above there is a word 
$i_1i_2\cdots i_p \sim \row(U_\alpha)\sim \row(T)$ with $i_1>n$ as needed.
\end{proof}

\begin{theorem}\label{eff-thm}
Let $\mu $ be a strict partition and suppose $n\in \NN$.
Then
$\GP_\mu \in  \NN[\beta]\spanning\{ G_\lambda \mid \text{partitions }\lambda \in \NN^n\}$ 
if and only if 
$n \geq \max(\mu)$.
\end{theorem}

\begin{proof} 
If $\mu$ is empty then $\GP_\mu = 1 = G_\mu$, so the result is clear. Assume $\mu$ is nonempty. 
Recall that  $G_\lambda$ appears in  $G_w$ if and only if there is an increasing tableau of shape $\lambda$ 
whose column reading word is in $\cH(w^{-1})$~\cite[Thm.~1]{BKSTY},
or equivalently whose reverse row reading word is in $\cH(w)$~\cite[Lem.~2.7]{Marberg2019a}.
It follows from~\eqref{GPz-eq} that $G_\lambda $ appears in $\GP_\mu = \GP_z$
if and only if there is an increasing tableau of shape $\lambda$ whose reverse row reading word is in $\cHSp(z)$.

Let $T$ be such a tableau and suppose $n = \max(\mu)$. By the remarks preceding Lemma~\ref{vdes-tab-lem}, $T$ cannot have more than $n+1$ rows and $\Des^\fpf_V(z)\subseteq [n+1]$.
We claim that if $T$ has exactly $n+1$ rows then its first column must be $1,2,3,\dotsc,n+1$.
To see this, notice that the last entry of the first column is also the last entry of the reverse row reading word of $T$, which is an element of $\cHSp(z)$ by assumption, and so this 
 number is contained in $\Des^\fpf_V(z)\subseteq[n+1]$ by Lemma~\ref{desv-lem}.
 As each column of $T$   is a strictly increasing sequence of positive integers, the first column therefore would have to be  $1,2,3,\dotsc,n+1$.
 
In view of the claim just shown, if $T$ had exactly $n+1$ rows that Lemma~\ref{vdes-tab-lem} would imply the contradictory property $\Des^\fpf_V(z)\not\subseteq [n+1]$, so in fact $T$ must have at most $n$ rows. 
We conclude that if $n \geq \max(\mu)$ then 
$\GP_\mu \in  \NN[\beta]\spanning\{ G_\lambda \mid \text{partitions }\lambda \in \NN^n\}$.

For the converse implication, 
recall that the lowest degree term of $\GP_\mu$ when $\deg(\beta):=1$ is 
the Schur $P$-function $P_\mu$, whose Schur expansion  includes
both $s_\mu$ and $s_{\mu^\T}$~\cite[Ex.~3, {\S}III.8]{Macdonald}.
Since $s_\lambda$ is the lowest degree term of $G_\lambda$,
the unique expansion of $\GP_\mu$ into $G_\lambda$'s 
must include a term indexed by $\lambda = \mu^\T \in \NN^{\max(\mu)}$.
Hence if $\GP_\mu \in  \NN[\beta]\spanning\{ G_\lambda \mid \text{partitions }\lambda \in \NN^n\}$ then $n \geq \max(\mu)$.
\end{proof}

\subsection{\texorpdfstring{$Q$}{Q}-Lascoux polynomials}\label{q-lasc-sect}

There is also a family of \defn{orthogonal Grothendieck polynomials} $\fkGO_z \in \ZZ[\beta][x_1,x_2,\dots]$ indexed by involutions $z \in I_\infty$,
constructed geometrically in~\cite[Def. 2.18]{MP2019a}. Unlike the two previous families,
these polynomials do not have a general algebraic definition in terms of divided difference operators. However, it
is known that if $\DO(z) := \{ (i,j) \in D(z) \mid i \geq j\}$ then 
\be \fkGO_z = \prod_{(i,j) \in \DO(z)} (x_i\oplus x_j)
\quad \text{when $z \in I_\infty$ is dominant~\cite[Thm.~3.8]{MP2019a},} 
\ee
and if $i \in \PP$ is such that $z$ and $ s_i z s_i$ are {distinct} and both {vexillary} then~\cite[Prop.~3.23]{MP2019a} implies
\be
 \bpartial_i \fkGO_z  = \begin{cases} \fkGO_{s_i z s_i}&\text{if } z(i) > z(i+1),\\
 -\beta\fkGO_z &\text{if } z(i) < z(i + 1). \end{cases}
 \ee  
It can shown that these formulas determine the values of $\fkGO_z$ for all vexillary $z \in I_\infty$; see \cite[\S2.3]{MarWen}.
 
 For any $z \in I_\infty$, it holds that
if $\deg(\beta) :=-1$ then $\fkGO_z$ is homogeneous of degree $|\DO(z)|$, and
if $\deg(\beta) := 0$ then $\fkGO_z$ is inhomogeneous
with lowest degree term  
$\fkSO_z := \fkGO_z\bigr\rvert_{\beta=0}$ by~\cite[\S2.4]{MP2019a}.
The latter polynomials are linearly independent over $\ZZ$~\cite[\S2.3]{HMP4}.
Thus $\{\fkGO_z : z \in I_\infty\}$ is linearly independent over $\ZZ[\beta]$,
since if we had a nontrivial linear dependence using coefficients in $\ZZ[\beta]$,
then taking lowest degree terms (with the convention that $\deg(\beta)=0$) and dividing out powers of $\beta$
would given a nontrivial linear dependence among the $\fkSO$-polynomials using coefficients in $\ZZ$.
\begin{example}
 %
 %
There are six dominant elements $z \in I_4$ 
whose orthogonal Grothendieck polynomials have simple product formulas:
$1$, $(1 \; 2)$, $(1 \; 3)$, $(1 \; 4)$, $(1 \; 3)(2 \; 4)$, and $(1 \; 4)(2 \; 3)$.
The 3 remaining transpositions in $I_4$ are vexillary, so their orthogonal Grothendieck polynomials can be computed
from the dominant formulas using divided difference operators:
\[
\ba
\fkGO_{(2 \; 3)} &= 2x_2 + 2x_1 + \beta x_2^2 + 4 \beta x_1 x_2 + \beta x_1^2 + 2 \beta^2 x_1 x_2^2 + 2\beta^2 x_1^2 x_2 + \beta^3 x_1^2 x_2^2, \\
\\[-10pt]
\fkGO_{(3 \; 4)} &= 2x_3 + 2x_2 + 2x_1 + \beta x_3^2 + 4\beta x_2 x_3 + \beta x_2^2 + 4\beta x_1 x_3 + 4\beta x_1 x_2 + \beta x_1^2 
\\& \qquad + 2\beta^2 x_2 x_3^2 + 2\beta^2 x_2^2 x_3 + 2\beta^2 x_1 x_3^2 + 8\beta^2 x_1 x_2 x_3 + 2\beta^2 x_1 x_2^2 + 2\beta^2 x_1^2 x_3 
\\&\qquad + 2\beta^2 x_1^2 x_2 + \beta^3 x_2^2 x_3^2 + 4\beta^3 x_1 x_2 x_3^2 + 4\beta^3 x_1 x_2^2 x_3 +\beta^3  x_1^2 x_3^2 + 4\beta^3 x_1^2 x_2 x_3
\\&\qquad  + \beta^3 x_1^2 x_2^2 + 2\beta^4 x_1 x_2^2 x_3^2 + 2\beta^4 x_1^2 x_2 x_3^2 + 2\beta^4 x_1^2 x_2^2 x_3 + \beta^5 x_1^2 x_2^2 x_3^2, \\
\\[-10pt]
\fkGO_{(2 \; 4)} &= 2x_2 x_3 + 2x_2^2 + 2x_1 x_3 + 4x_1 x_2 + 2x_1^2 + 3\beta x_2^2 x_3 + \beta x_2^3 + 8\beta x_1 x_2 x_3
\\&\qquad + 6\beta x_1 x_2^2 + 3\beta x_1^2 x_3 + 6\beta x_1^2 x_2 + \beta x_1^3 + \beta^2 x_2^3 x_3 + 8\beta^2 x_1 x_2^2 x_3 + 2\beta^2 x_1 x_2^3 
\\&\qquad + 8\beta^2 x_1^2 x_2 x_3 + 5\beta^2 x_1^2 x_2^2 + \beta^2 x_1^3 x_3 + 2\beta^2 x_1^3 x_2 + 2\beta^3 x_1 x_2^3 x_3 + 6\beta^3 x_1^2 x_2^2 x_3 
\\&\qquad + \beta^3 x_1^2 x_2^3 + 2\beta^3 x_1^3 x_2 x_3 +\beta^3 x_1^3 x_2^2 + \beta^4 x_1^2 x_2^3 x_3 + \beta^4 x_1^3 x_2^2 x_3.
\ea
\]
The final element $(1,2)(3,4) \in I_4$ is not vexillary, but we 
can compute $\fkGO_{(1, 2)(3, 4)}$  directly from the geometric definition in~\cite{MP2019a} using \textsc{Macaulay2}:
\[
\ba
\fkGO_{(1, 2)(3, 4)} &= 4x_1 x_3 + 4x_1 x_2 + 4x_1^2 + 2\beta x_1 x_3^2 + 8\beta x_1 x_2 x_3 + 2\beta x_1 x_2^2 
\\&\qquad  + 10\beta x_1^2 x_3 + 10\beta x_1^2 x_2 + 4\beta x_1^3 + 4\beta^2 x_1 x_2 x_3^2 + 4\beta^2 x_1 x_2^2 x_3 + 5\beta^2 x_1^2 x_3^2 
\\&\qquad  + 20\beta^2 x_1^2 x_2 x_3 + 5\beta^2 x_1^2 x_2^2 + 8\beta^2 x_1^3 x_3 + 8\beta^2 x_1^3 x_2 + \beta^2 x_1^4 + 2\beta^3 x_1 x_2^2 x_3^2 
\\&\qquad  + 10\beta^3 x_1^2 x_2 x_3^2 + 10\beta^3 x_1^2 x_2^2 x_3 + 4\beta^3 x_1^3 x_3^2 + 16\beta^3 x_1^3 x_2 x_3 + 4\beta^3 x_1^3 x_2^2 
\\&\qquad  + 2\beta^3 x_1^4 x_3 + 2\beta^3 x_1^4 x_2 + 5\beta^4 x_1^2 x_2^2 x_3^2 + 8\beta^4 x_1^3 x_2 x_3^2 + 8\beta^4 x_1^3 x_2^2 x_3 
\\&\qquad  + \beta^4 x_1^4 x_3^2 + 4\beta^4 x_1^4 x_2 x_3 + \beta^4 x_1^4 x_2^2 + 4\beta^5 x_1^3 x_2^2 x_3^2 + 2\beta^5 x_1^4 x_2 x_3^2 
\\&\qquad  + 2\beta^5 x_1^4 x_2^2 x_3 + \beta^6 x_1^4 x_2^2 x_3^2.
\ea
\]

\end{example}

Our next result requires some background on the geometric meaning of $\fkGO_z$. 
Let $\GL_n = \GL_n(\CC)$ be the complex general linear group and $\O_n = \{ g \in \GL_n \mid g = (g^{-1})^\T \}$ be the subgroup of orthogonal matrices.
Identify $w \in S_n$ with the permutation matrix in $\GL_n$ with $1$ in each position $(i,w(i))$ for $i \in [n]$.
Write $B_n$ for the subgroup of upper triangular matrices in $\GL_n$, and let $B_n^-$ be its transpose.
Define the \defn{complete flag variety} to be $\Fl_n  := B_n^- \backslash \GL_n$.

The group $B_n$ acts on $\Fl_n$ on the right with finitely many orbits, given by the double cosets $B_n^- w B_n$ for $w \in S_n$.
Let $X_w^{B_n}$ denote the closure of this orbit. 
As explained in~\cite[\S10]{RichSpring}, the group $\O_n$ also acts on $\Fl_n$ with finitely many orbits indexed by the elements of $I_n := S_n \cap I_\infty$. 
We write $X_z^{\O_n}$ to denote the closure of the $\O_n$-orbit in $\Fl_n$ corresponding to $z \in I_n$.
 
 The \defn{connective $K$-theory ring} of $\Fl_n$ can be realized as 
 $\CK(\Fl_n) \cong \ZZ[\beta][x_1,x_2,\dotsc,x_n] / \ILambda_n[\beta]$, where $\ILambda_n$ is the ideal generated by
 the symmetric polynomials in $\ZZ[x_1,x_2,\dotsc,x_n] $ without constant term.
Results in~\cite{FominKirillov, FultonLascoux, HorKir, Hudson2014} show that if $w \in S_\infty$ then $\fkG_w$ is the unique element of  
 $\ZZ[\beta][x_1,x_2,\dotsc,x_n]$ such that 
 the connective $K$-theory class $[X_w^{B_n}] \in \CK(\Fl_n)$
 is equal to $\fkG_w + \ILambda_n[\beta]$ whenever $n$ is large enough that $w \in S_n$; see~\cite[Thm.~1.1]{MP2019a}.
 Similarly,  the polynomial $ \fkGO_z$ for $z \in I_\infty$
 is the unique 
  element of  
 $\ZZ[\beta][x_1,x_2,\dotsc,x_n]$
 such that $[X_z^{\O_n}] = \fkGO_z + \ILambda_n[\beta] \in \CK(\Fl_n)$ whenever $z \in I_n$~\cite[Thm.~1.4]{MP2019a}.
 For an analogous  definition of $\fkGS_z$, see~\cite[Thm.~1.3]{MP2019a}.

Recall from \eqref{visible-eq} that the visible descent set of $z \in I_\infty$ is  $\Des_V(z) := \{ i \in \PP : z(i+1) \leq \min\{i,z(i)\}\}$.
If $i \in \Des_V(z)$ then row $i$ of $\DO(z)$ contains the position $(i,z(i+1))$.


\begin{proposition}\label{ratsing-prop}
 If $z \in I_\infty$ is vexillary with $\Des_V(z) \subseteq [n]$ then
 \[
 \fkGO_z \in \NN[\beta]\spanning\{ \fkG_w \mid w \in S_\infty \text{ with }\DesR(w) \subseteq [n] \} \subseteq \ZZ[\beta][x_1,x_2,\dotsc,x_n].
 \]
 \end{proposition}
 
 \begin{proof}
 A result of Brion~\cite[Thm.~1]{Brion2002} states that if $Y$ is a closed subvariety of $\Fl_m$ with rational singularities,
 then the $K$-theory class $[Y]_K$ of $Y$  expands as a linear combination of the $K$-theory classes 
 $[X^{B_m}_w]$ for $w \in S_m$ with coefficients $c_Y^w \in \ZZ$ satisfying $(-1)^{\codim(X^{B_m}_w) - \codim(Y)} c_Y^w \geq 0$.
The Schubert variety $X^{B_m}_w$ has codimension $\ell(w) = \abs{D(w)}$,
and its $K$-theory class is represented by $\fkG_w\bigr\rvert_{\beta=-1}$ in the quotient ring $K(\Fl_m) \cong \ZZ[x_1,x_2,\dotsc,x_m]/\ILambda_m$~\cite[Thm.~2.8]{MP2019a}.
Similarly, if  $z \in I_m$ then the closed variety $Y = X^{\O_m}_z$ has codimension $\abs{\DO(z)}$~\cite[Thm.~4.6]{RichSpring},
 and $K$-theory class $[Y]_K = \fkGO_z\bigr\rvert_{\beta=-1} + \ILambda_m$~\cite[Cor.~2.9 and Thm.~2.19]{MP2019a}.
 As explained in the proof of~\cite[Prop.~3.23]{MP2019a}, if $z \in I_m$ is vexillary then $X^{\O_m}_z$ has rational singularities,
 so we can express
 \be\label{Go-eq}
 \textstyle\fkGO_z\bigr\rvert_{\beta=-1} + \ILambda_m  = \sum_{w \in S_m} (-1)^{|D(w)|- |\DO(z)|} \cdot  c_z^w\cdot   \fkG_w\bigr\rvert_{\beta=-1} + \ILambda_m
 \ee
 for nonnegative integer coefficients $c_z^w\in \NN$. 
 
 On the other hand, if $z\in I_\infty$ is vexillary with $\Des_V(z) \subseteq [n]$ then~\cite[Thm.~3.26]{MP2019a} gives a formula for $\fkGO_z$ that belongs to $\ZZ[\beta][x_1,x_2,\dotsc,x_n]$. By Proposition~\ref{basis-prop},
 since $\fkG_w$ and $\fkGO_z$ are homogeneous if $\deg(\beta) := -1$, there are unique integers $b_z^w \in \ZZ$
 with $\fkGO_z = \sum_{\DesR(w)\subseteq [n]} \beta^{|D(w)| - |\DO(z)|} \cdot b_z^w\cdot \fkG_w$.
 Setting $\beta=-1$ in this formula turns $\beta^{|D(w)| - |\DO(z)|}$ into the sign $(-1)^{|D(w)| - |\DO(z)|}$ appearing in \eqref{Go-eq}.
 Since $\{ \fkG_w\bigr\rvert_{\beta=-1} + \ILambda_m  \mid w \in S_m\}$ is a $\ZZ$-basis for the $K$-theory ring of $\Fl_m$ \cite[\S1]{Brion2002},
 it follows that if $m$ is sufficiently large then $b^w_z = c^w_z \geq 0$.
 \end{proof}
 
Since $\fkGO_z\in \ZZ[\beta][x_1,\dotsc,x_n]$  implies $\fkSO_z\in \ZZ[x_1,\dotsc,x_n]$, by Proposition~\ref{lexmin-prop} we get:
 
\begin{corollary}
 If $z \in I_\infty$ is vexillary then $\fkGO_z\in \ZZ[\beta][x_1,\dotsc,x_n]$ if and only if $ \Des_V(z) \subseteq [n]$.
\end{corollary}

 
 In contrast to the $\fkGS_z$ case, it is an open problem to identify the summands $\fkG_w$
that appear in $\fkGO_z$, as well as their coefficients (which are no longer always equal).

\begin{example}
If $z \in I_4$ is vexillary then $\fkGO_z$ expands positively into $\fkG_w$'s as follows:
\begingroup  
\allowdisplaybreaks
\begin{align*}
\fkGO_{1} &= \fkG_{1}, \\
\fkGO_{(1 \; 2)} &= 2 \fkG_{21} + \beta \fkG_{312}, \\
\fkGO_{(2 \; 3)} &= 2 \fkG_{132} + \beta \fkG_{1423} + \beta \fkG_{231} + \beta^2 \fkG_{2413}, \\
\fkGO_{(3 \; 4)} &= 2 \fkG_{1243} + \beta \fkG_{12534} + \beta \fkG_{1342} + \beta^2 \fkG_{13524}, \\
\fkGO_{(1 \; 3)} &= 2 \fkG_{231} + 2 \fkG_{312} + 3\beta \fkG_{321} + \beta \fkG_{4123} + \beta^2 \fkG_{4213}, \\
\fkGO_{(2 \; 4)} &= 2 \fkG_{1342} + 2 \fkG_{1423} + 3\beta \fkG_{1432} + \beta \fkG_{15234} + \beta^2 \fkG_{15324} + \beta \fkG_{2341} + \beta^2 \fkG_{2431}, \\
\fkGO_{(1 \; 4)} &= 2 \fkG_{2341} + 2 \fkG_{3142} + 3\beta \fkG_{3241} + 2 \fkG_{4123} + 3\beta \fkG_{4132} + \beta^2 \fkG_{4231} + \beta \fkG_{51234} + \beta^2 \fkG_{51324}, \\
\fkGO_{(1 \; 3)(2 \; 4)} &= 4 \fkG_{2413} + 2\beta \fkG_{25134} + 2\beta \fkG_{3412} + \beta^2 \fkG_{35124}, \\
\fkGO_{(1 \; 4)(2 \; 3)} &= 4 \fkG_{2431} + 2\beta \fkG_{25314} + 4 \fkG_{3412} + 6\beta \fkG_{3421} + 2\beta \fkG_{35124} + 3\beta^2  \fkG_{35214} + 4 \fkG_{4213} 
\\&\qquad + 4\beta \fkG_{4231} + 6\beta \fkG_{4312} + 6\beta^2  \fkG_{4321} + \beta^2 \fkG_{45123} + \beta^3 \fkG_{45213} + 2\beta \fkG_{52134} 
\\&\qquad + 2\beta^2  \fkG_{52314} + 3\beta^2  \fkG_{53124} + 3\beta^3 \fkG_{53214} + \beta^3 \fkG_{54123} + \beta^4 \fkG_{54213}.
\end{align*}
\endgroup
There is also a positive expansion of $\fkGO_z$ into $\fkG_w$'s
for the non-vexillary element $z = 2143 = (1 \; 2)(3 \; 4) \in I_4$, although this is not guaranteed by Proposition~\ref{ratsing-prop}:
{\[
\ba
\fkGO_{(1 \; 2)(3 \; 4)} &= 4 \fkG_{2143} + 2\beta \fkG_{21534} + 2\beta \fkG_{2341} + 2\beta^2 \fkG_{23514} + 4\beta \fkG_{3142} + 3\beta^2 \fkG_{31524}
\\&\qquad + 3\beta^2 \fkG_{3241} + 3\beta^3 \fkG_{32514} + 2\beta \fkG_{4123} + 3\beta^2 \fkG_{4132}  + \beta^3 \fkG_{41523} + \beta^3 \fkG_{4231} 
\\&\qquad+ \beta^4 \fkG_{42513} + \beta^2 \fkG_{51234} + 2\beta^3 \fkG_{51324}+ \beta^4 \fkG_{51423} + \beta^4 \fkG_{52314} + \beta^5 \fkG_{52413}.
\ea
\]}
\end{example}

 Let $\alpha$ be a symmetric weak composition  with $u := u(\alpha)$ and $\lambda := \lambda(\alpha)$.
Comparison with the previous section suggests that it would be natural to define
\be\label{tql-eq}
\tqlascoux_{\alpha} = \bpi_u \(\prod_{\substack{(i,j) \in \D_\lambda \\ i \geq j} } (x_i \oplus  x_{j})\)
\quand
\tqlascouxatom_{\alpha} = \bopi_u \(\prod_{\substack{(i,j) \in \D_\lambda \\ i \geq j} } (x_i \oplus  x_{j})\)
\ee
and consider these as $Q$-versions of Lascoux polynomials and Lascoux atoms.
Indeed, the polynomials $\tqlascoux_{\alpha}$ and $\tqlascouxatom_{\alpha}$ have some nice properties.
They specialize to $\qkey_\alpha$ and $\qatom_\alpha$ when $\beta=0$,
and are homogeneous of degree $|\half(\lambda(\alpha))|$
 when $\deg(\beta) := -1$. The obvious analogues of
 Lemma~\ref{sp-lasc-lam-lem} and
 Proposition~\ref{sp-lasc-formulas}
 hold for these polynomials with essentially the same proofs.
These results imply that  $\tqlascoux_\alpha = \sum_{\gamma \preceq \alpha} \tqlascouxatom_\gamma$ by Lemma~\ref{var-bar-combine-lem}
as well as the following proposition:

 \begin{proposition}\label{o-lascoux-positive-prop}
If $\alpha$ is a symmetric weak composition then  $\tqlascoux_\alpha$ 
is a nonzero $\NN[\beta]$-linear combination of Lascoux polynomials 
and  $\tqlascouxatom_\alpha$ is a $\NN[\beta]$-linear combination of Lascoux atom polynomials.
\end{proposition}

 \begin{proof}
 As in the proof of Proposition~\ref{sp-lascoux-positive-prop}, the result holds
 if $\alpha = \lambda(\alpha)$
as then
 $\tqlascoux_\alpha = \tqlascouxatom_\alpha=\fkGO_z$ for some dominant $z \in I_\infty$ by the remarks
 before Example~\ref{321-ex}.
The general statement follows from the partition case given Lemma~\ref{var-bar-combine-lem}.
 \end{proof}

 At this point the properties of $\tqlascoux_\alpha$ and $\tqlascouxatom_\alpha$ diverge from what we expect 
 of good notions of $Q$-Lascoux polynomials and atoms.
Most notably, 
$\fkGO_z $ does not always expand as a $\NN[\beta]$-linear combination of $\tqlascoux_\alpha$'s,
and not every $\fkGO_z $  with $z$ vexillary occurs as an instance of $\tqlascoux_\alpha$; see Example~\ref{LO-ex} below.
Moreover, when $\lambda \in \NN^n$ is a symmetric partition 
 $\tqlascoux_{(\lambda_n,\dotsc,\lambda_2,\lambda_1)}$ 
does not usually coincide with any of the \defn{$K$-theoretic Schur $Q$-polynomials} $\GQ_\lambda$ defined in~\cite{IkedaNaruse} and discussed below.

Conjecture~\ref{vex-conj} suggests an alternate definition of $Q$-Lascoux polynomials that avoids these issues, but which has the drawback of being incomplete. Recall that each permutation in $S_\infty$ is uniquely determined by its code.

\begin{definition}
If $\alpha$ is the code of a vexillary involution $z \in I_\infty$,
then define the \defn{$Q$-Lascoux polynomial} of $\alpha$ to be $\qlascoux_\alpha := \fkGO_z$.
\end{definition}

\begin{example}\label{LO-ex}
We already have $\qlascoux_\alpha \neq \tqlascoux_\alpha$ for $\alpha =(0,1)$, as 
\[
\ba
\qlascoux_{01} &=
 2x_2 + 2x_1 + \beta x_2^2 + 4\beta x_1 x_2 + \beta x_1^2 + 2\beta^2 x_1 x_2^2 + 2\beta^2 x_1^2 x_2 + \beta^3 x_1^2 x_2^2,
 \\  \tqlascoux_{01}&=  2x_2 + 2x_1 + \beta x_2^2 + 3\beta x_1 x_2 + \beta x_1^2 + \beta^2 x_1 x_2^2 + \beta^2 x_1^2 x_2.
\ea
\]
Note that $\qlascoux_{01}  = \fkGO_z$ for $z = s_2 \in I_\infty$.
This orthogonal Grothendieck polynomial is not a $\NN[\beta]$-linear combination of 
$\tqlascoux_\alpha$'s since 
any such expression equal to $\fkGO_z$ must be in
$\tqlascoux_{01}+ \sum_\alpha \beta\NN[\beta] \tqlascoux_\alpha$ to account for the lowest order term $2x_2+2x_1$, but then there is no way to contribute the first term in
\[
\fkGO_z - \tqlascoux_{01} =  \beta x_1 x_2 +\beta^2  x_1 x_2^2 + \beta^2 x_1^2 x_2 + \beta^3 x_1^2 x_2^2 \]
as it is clear from \eqref{tql-eq} that the lowest order term of $\tqlascoux_\alpha$ is divisible by $2$ when $\alpha$ is nonempty.
We also remark that while $\qkey_{01} = 2 \pkey_{12}$ by Theorem~\ref{alpha1-thm},  
we have $\qlascoux_{01} \neq 2 \plascoux_{12}$
since 
$\plascoux_{12} = x_2 + x_1 + \beta x_1 x_2$.
\end{example}

The polynomial $\qlascoux_\alpha $ shares the following property with $\tqlascoux_\alpha $:

 \begin{corollary}
If $\alpha$ is the code of a vexillary involution then $\qlascoux_\alpha $
is a nonzero $\NN[\beta]$-linear combination of Lascoux polynomials.
\end{corollary}

\begin{proof}
Expand $\qlascoux_\alpha $ into $\fkG_w$'s using Proposition~\ref{ratsing-prop} and then into $L_\gamma$'s using Theorem~\ref{sy-thm}.
\end{proof}

If $w \in S_\infty$ is vexillary then the rows and columns of $D(w)$ can be rearranged to form the diagram of a partition $\lambda$~\cite[\S2.2.1]{Manivel}.
When $w=w^{-1}$ is vexillary, this partition $\lambda$ is symmetric as the set of $D(w)$ is invariant under transpose.
It follows that the code $ c(z)$ is a symmetric weak composition whenever $z \in I_\infty$ is vexillary.
We pose the following open problem.

\begin{problem}
\label{prob:nonvex}
Extend the definition of $\qlascoux_\alpha$ to all symmetric weak compositions $\alpha$.
\end{problem}

Such an extension should satisfy $\qlascoux_\alpha\bigr\rvert_{\beta=0} = \qkey_\alpha$,
and each $\fkGO_z$ should be a $\NN[\beta]$-linear combination of $\qlascoux_\alpha$'s. 
Since $\fkGO_z$ for general $z \in I_\infty$ does not have a simple inductive definition in terms of divided difference operators,
the ``correct'' general definition of $\qlascoux_\alpha$ is likely to be geometric in nature.
It is not yet clear to us how to define a meaningful $Q$-analogue of Lascoux atoms.

\subsection{Orthogonal stable Grothendieck polynomials}

We conclude with a few miscellaneous results about the stable limits of $\fkGO_z$.
We do not have an analogue of the generating function formula~\eqref{pre-GPz-eq} for these polynomials.
However, a natural way to define \defn{orthogonal stable Grothendieck polynomials},
following~\cite[\S4.2]{MP2019a}, is as the limit
\be\label{GQz-eq}
 \GQ_z := \lim_{N\to \infty} \fkGO_{1^N \times z}\quad\text{for }z \in I_\infty.
 \ee
It is an open problem to show this limit converges for arbitrary $z \in I_\infty$,
but when $z$ is vexillary, the formula~\cite[Thm.~3.26]{MP2019a} makes it clear that~\eqref{GQz-eq} is well-defined.
In this case $\GQ_z$ is an instance of Ikeda and Naruse's \defn{$K$-theoretic Schur $Q$-functions},
as we explain below.

 Ikeda and Naruse show in~\cite{IkedaNaruse} that for each strict partition $\mu$ with $r$ parts, there is a unique 
power series $\GQ_\mu \in \ZZ[\beta]\llbracket x_1,x_2,\dots\rrbracket $ with
\be
\GQ_\mu(x_1,\dotsc,x_n) = \frac{1}{(n-r)!}\sum_{w \in S_n} w\( x^\mu \prod_{i=1}^{r}  (2+ \beta x_i) \prod_{j=i+1}^n \frac{x_i\oplus x_j}{x_i\ominus x_j}\)
\ee
for all $n\geq r$~\cite[\S2]{IkedaNaruse}.
By repeating the proof of 
\cite[Prop.~4.11]{MP2019b}, just changing ``$\GP$'' to ``$\GQ$'' and ``$\prod_{i=1}^r \prod_{j=i+1}^n$''
to ``$\prod_{i=1}^r (2+\beta x_i)\prod_{j=i+1}^n$'',
one can also show that if  $\mu$ is strict with $r$ parts then
\be\label{GQ-def2-eq}
\GQ_\mu(x_1,\dotsc,x_n) = \bpi_{w_n}\( x^\mu \prod_{i=1}^{r}  \prod_{j=i}^n \frac{x_i\oplus x_j}{x_i}\).
\ee
The power series $\GQ_\mu$ is symmetric in the $x_i$ variables,
and one has $Q_\mu = \GQ_\mu\bigr\rvert_{\beta=0}$ by~\cite[Ex. 1, \S{III.3}]{Macdonald}.
Like $\GP_\mu$, one can express $\GQ_\mu$ as 
a generating function
for semistandard set-valued shifted tableaux of shape $\mu$~\cite[Thm.~9.1]{IkedaNaruse}.
As $\mu$ ranges over all strict partitions in $\NN^n$,
the polynomials $\GQ_\mu(x_1,x_2,\dotsc,x_n)$ form a $\ZZ[\beta]$-basis for a subring of symmetric elements of $\ZZ[\beta][x_1,x_2,\dotsc,x_n]$ characterized by a  cancellation property plus an extra divisibility condition~\cite[\S3.3]{IkedaNaruse}.
 
Returning to the stable limits of $\fkGO_z$.
If $z \in I_\infty$ is vexillary then $\GQ_z = \GQ_{\lambdaO(z)}$ by~\cite[Thm.~4.11]{MP2019a},
and therefore $\GQ_z\bigr\rvert_{\beta=0} = Q_{\lambdaO(z)} = Q_z$ by~\cite[Thm.~4.67]{HMP4}.

Something that makes it difficult to evaluate~\eqref{GQz-eq} for non-vexillary $z \in I_\infty$
is the lack of a general result relating $ \lim_{N\to \infty} \fkGO_{1^N \times z}$ to the symmetrization operator 
$\bpi_{w_n}$.
This is in contrast to the classical and symplectic cases, where one can invoke Lemmas~\ref{bpi-lem1} and \ref{bpi-lem2}.
We can fill in this gap for a subset of vexillary involutions.

An element $z \in I_\infty$ is \defn{I-Grassmannian} if its visible descent set $\Des_V(z)$
has at most one element. If $\Des_V(z) = \{n\}$ then $z$ must have the form 
$z = (\phi_1 \; n+1)(\phi_2 \; n+2) \cdots (\phi_r \; n+r)$ for some integers $r \in \PP$ and $1\leq \phi_1 < \phi_2 < \dots < \phi_r \leq n$,
and in this case $\lambdaO(z) = (n+1-\phi_1,n+1-\phi_2,\dotsc,n+1-\phi_r)$~\cite[Prop.-Def. 4.16]{HMP4}. It follows that if $z$ is I-Grassmannian then $z$ is vexillary; therefore $\GQ_z = \GQ_{\lambdaO(z)}$.

\begin{proposition}\label{GO-bpi-prop}
Suppose $z \in I _\infty$ has $\Des_V(z) = \{n\}$ and $\mu = \lambdaO(z)$. Then 
\[\GQ_z(x_1,x_2,\dotsc,x_n) = \GQ_\mu(x_1,x_2,\dotsc,x_n) = \bpi_{w_n} (\fkGO_z) . \]
\end{proposition}

\begin{proof}
The first equality holds since $z$ is I-Grassmannian and therefore vexillary. 
For the second equality, we follow the proof of~\cite[Prop.~4.14]{MP2019a}, which derives a 
similar identity for $\fkGS_z$.
Suppose $\mu$ has $r$ nonzero parts, so that 
$z = (\phi_1 \; n+1)(\phi_2 \; n+2) \cdots (\phi_r \; n+r)$ for $\phi_i = n+1-\mu_i$.
Then
\be\label{GQ-GO-eq1}
x^{\mu} \prod_{i=1}^{r}  \prod_{j=i}^n \dfrac{x_i\oplus x_j}{x_i} = x_1^{1-\phi_1}x_2^{2-\phi_2}\cdots x_r^{r-\phi_r} \fkGO_w
\ee
for the dominant involution $w = (1 \; n+1)(2 \; n+2)\cdots (r \; n+r)$, which has $\fkGO_w = \prod_{i=1}^r \prod_{j=i}^n x_i\oplus x_j$.

The expression $x_1^{1-\phi_1}x_2^{2-\phi_2}\cdots x_r^{r-\phi_r} \fkGO_w$ is symmetric in $x_{r+1}, x_{r+2}, \dots, x_n$
and therefore fixed by $\bpi_i$ for all $r+1 \leq i < n$.
For integers $0 <a \leq b$ let $\bpi_{b \searrow a} := \bpi_{b-1}\bpi_{b-2}\cdots \bpi_a$
so that $\bpi_{a\searrow a} = 1$, and define $\bpartial_{b\searrow a}$ analogously.
Since $\bpi_i\bpi_{b \searrow a} = \bpi_{b \searrow a} \bpi_{i+1}$
for $a \leq i < b-1$ as $\bpi_i$ satisfy the braid relations for $S_\infty$, one can check that 
$ \bpi_{\phi_j \searrow j} \bpi_{\phi_{j+1}\searrow j+1} \cdots \bpi_{\phi_r \searrow r} \( x_1^{1-\phi_1}x_2^{2-\phi_2}\cdots x_r^{r-\phi_r} \fkGO_w\)$ is fixed by $\bpi_i$ for all $j -1 < i < \phi_{j-1} < \phi_j$.
This calculation is very similar to the argument in the proof of \cite[Prop.~4.14]{MP2019a} and so we have omitted a detailed
explanation here.

Now we appeal to~\cite[Lem.~4.16]{MP2019b}, which asserts that if $f$ is any Laurent polynomial fixed by $\bpi_i$ for all $a < i <b$ then $\bpi_{b\searrow a} f = \bpartial_{b\searrow a}(x^{b-a}_a f)$.
It follows by successively applying this identity along with $x_i (\bpartial_j f) = \bpartial_j (x_i f)$ for $i<j$ that 
\be\label{GQ-GO-eq2}
\bpi_{\phi_1 \searrow 1} \bpi_{\phi_{2}\searrow 2} \cdots \bpi_{\phi_r \searrow r} \( x_1^{1-\phi_1}x_2^{2-\phi_2}\cdots x_r^{r-\phi_r} \fkGO_w\)
=
\bpartial_{\phi_1 \searrow 1} \bpartial_{\phi_{2}\searrow 2} \cdots \bpartial_{\phi_r \searrow r} \fkGO_w.
\ee

Finally, we claim that 
\be\label{GQ-GO-eq3}
\bpartial_{\phi_1 \searrow 1} \bpartial_{\phi_{2}\searrow 2} \cdots \bpartial_{\phi_r \searrow r} \fkGO_w = \fkGO_z.
\ee
One can check this by induction from the fact  that if $y $ and $ s_i y s_i$ are vexillary involutions with $\ell(y) > \ell(s_i y s_i)$
then $\bpartial_i \fkGO_y = \fkGO_{s_i y s_i}$~\cite[Prop.~3.23]{MP2019a}. 
If we write out  $\bpartial_{\phi_1 \searrow 1} \bpartial_{\phi_{2}\searrow 2} \cdots \bpartial_{\phi_r \searrow r}$ in
full as the product $\bpartial_{i_k}\cdots \bpartial_{i_2} \bpartial_{i_1}$ for $k := (\phi_1 + \phi_2  + \dots + \phi_r) - (1 + 2 + \dots + r)$, 
and set $y_0 := w$ and $y_{j} := s_{i_j} y_{j-1} s_{i_j}$ for $j \in [k]$,
then each $y_j$ is I-Grassmannian with $\ell(y_{j-1}) > \ell(y_j)$ and $y_k=z$.
As all I-Grassmannian involutions are vexillary, the identity~\eqref{GQ-GO-eq3} follows.

Combining~\eqref{GQ-GO-eq1},~\eqref{GQ-GO-eq2}, and~\eqref{GQ-GO-eq3}
gives 
$\bpi_{\phi_1 \searrow 1} \bpi_{\phi_{2}\searrow 2} \cdots \bpi_{\phi_r \searrow r} \(x^{\mu} \prod_{i=1}^{r}  \prod_{j=i}^n \frac{x_i\oplus x_j}{x_i} \) =\fkGO_z.$
Since $\bpi_{w_n} \bpi_i = \bpi_{w_n}$ for all $i \in [n-1]$, applying $\bpi_{w_n}$ to both sides 
transforms this to the desired identity
$\GQ_z(x_1,x_2,\dotsc,x_n) = \bpi_{w_n} (\fkGO_z)  $
in view of~\eqref{GQ-def2-eq}. 
\end{proof}

This result only asserts that $\GQ_z(x_1,x_2,\dotsc,x_n) = \bpi_{w_n} (\fkGO_z)$
when $n = \max (\Des_V(z))$ rather than for all $n \geq \max (\Des_V(z))$ as in Lemma~\ref{bpi-lem2}.
 The stronger claim does not always hold.
 
 \begin{example}
 \label{ex:GQ_long_element}
 Let $z = (1,2) \in I_\infty$.
 Then $z$ is I-Grassmannian with $\Des_V(z) = \{1\}$ and $\lambda(z)=\lambdaO(z) = (1)$, 
 and we have  
 $ \GQ_z = \GQ_{(1)} $ and $ \GQ_z(x_1) = \bpi_{w_1} (\fkGO_z) = \fkGO_z = 2x_1 + \beta x_1^2,$
 but $\GQ_z(x_1,x_2) = \qlascoux_{01} \neq \tqlascoux_{01} =  \bpi_{w_2} (\fkGO_z)$ by Example~\ref{LO-ex}.
\end{example}

Computations suggest that the discrepancy in Example~\ref{ex:GQ_long_element} only occurs when $z(1) \neq 1$.
 
\begin{conjecture} 
If $z \in I_\infty$ has $z(1)=1$  then the limit~\eqref{GQz-eq} converges as a power series to a symmetric function with
$\GQ_z(x_1,x_2,\dotsc,x_{n}) = \bpi_{w_{n}} (\fkGO_{ z}) $ for all $n \in \PP$ with $\Des_V(z)\subseteq [n]$.
\end{conjecture}
 
Finally, we can use these results to prove 
an orthogonal analogue of Theorem~\ref{eff-thm}. This resolves an open question~\cite[Problem~5.6]{MP2019a}
about whether $\GQ_\mu$ has a finite, positive expansion into $G_\lambda$'s.

 \begin{theorem}\label{GQ-thm}
 Let $\mu $ be a strict partition and suppose $n\in \NN$.
Then
$\GQ_\mu \in  \NN[\beta]\spanning\{ G_\lambda \mid \text{partitions }\lambda \in \NN^{n}\}$
if and only if $n> \max(\mu)$ or $\mu$ is the empty partition.
\end{theorem}

\begin{proof} 
Let $\Lambda^{n}$ be the set of partitions with at most $n$ parts.
We may assume that $\mu$ is nonempty as there is nothing to prove for $\GQ_\emptyset = 1$.
Fix an integer $n > \max(\mu)$.
It is known that $\GQ_\mu$ is a finite $\ZZ[\beta]$-linear combination of 
$\GP_\nu$'s with $\max(\nu) \leq \max(\mu) + 1 \leq n$~\cite[Thm.~1.1]{ChiuMarberg}.
Therefore, by Theorem~\ref{eff-thm} we have 
$\GQ_\mu  = \sum_{\lambda \in \Lambda^{n}} a_\mu^\lambda  \beta^{|\lambda|-|\mu|}  G_\lambda$
for integers $a_\mu^\lambda \in \ZZ$ that are nonzero for only finitely many partitions $\lambda$,
  each with at most $n$ parts and $|\lambda| \geq |\mu|$.
On the other hand, if 
 $z \in I_\infty$ is the unique I-Grassmannian element with shape $\mu$ and maximal visible descent $n$, then
\[\ba
\GQ_\mu(x_1,x_2,\dotsc,x_{n})   =  \bpi_{w_{n}} (\fkGO_z) & \in \NN[\beta]\spanning\{G_w(x_1,x_2,\dotsc,x_{n}) \mid \Des(w) \subseteq [n]\} \\& \qquad \subseteq \NN[\beta]\spanning\{G_\lambda(x_1,x_2,\dotsc,x_{n}) : \lambda \in \Lambda^{n}\}
\ea
\]
by Proposition~\ref{GO-bpi-prop} (for the first equality), Lemma~\ref{bpi-lem1} and Proposition~\ref{ratsing-prop}
(for the next containment), and Proposition~\ref{Gexp-prop} (for the last inclusion).
Since the polynomials $G_\lambda(x_1,x_2,\dotsc,x_{n}) = L_{(\lambda_{n},\dotsc,\lambda_2,\lambda_1)}$ for $\lambda \in \Lambda^{n}$ are linearly independent over $\ZZ[\beta]$, it follows that   $a_\mu^\lambda \in \NN$ for all $\lambda$.

This shows the ``if'' direction of the theorem. For the converse implication,
we appeal again to~\cite[Thm.~1.1]{ChiuMarberg}, which shows that
\[
\GQ_\mu \in 2^{\ell(\mu)} \GP_\mu + 2^{\ell(\mu) - 1} \beta \sum_\lambda \GP_\lambda + \beta^2 \ZZ[\beta] \spanning\{ \GP_\nu \mid \text{strict }\nu \},
\]
where the sum is over all strict partitions of the form $\lambda = \mu + \e_i$ with $ 1\leq i \leq \ell(\mu)$.
These summands always include $ \mu+\e_1$, whose maximum part is $\max(\mu)+1$,
so it follows from Theorem~\ref{eff-thm} that the stable Grothendieck expansion 
of $\GQ_\mu$ must involve $G_\lambda$'s with $\ell(\lambda) = \max(\mu)+1$.
\end{proof}

\section*{Declarations}

\subsection*{Funding}

The first author was partially supported by Hong Kong RGC grants 16306120 and 16304122.
The second author was partially supported by Grant-in-Aid for JSPS Fellows 21F51028.
The authors have no other relevant financial or non-financial interests to disclose.
 
 \subsection*{Data availability}
 
 Computations and source code related to the examples and conjectures in this article are available from the corresponding author on reasonable request.
 
 \subsection*{Acknowledgments}
 
 The authors thank Jonathan Brundan, Takeshi Ikeda, Brendan Pawlowski, Anne Schilling, and Alexander Yong for useful discussions.
 The authors also thank the anonymous referees for many helpful comments and corrections.
 This work benefited from computations using \textsc{SageMath}~\cite{sage}.

\printbibliography

\end{document}